\documentclass{amsart}

\usepackage{graphicx}
\usepackage{amsmath}
\usepackage{amsfonts}
\usepackage{amssymb}
\usepackage{bbm}
\usepackage{epic}
\usepackage{amscd}
\usepackage{amsthm}
\usepackage{amssymb}

\usepackage{mathtools}

\usepackage[margin=1.0in]{geometry}

\usepackage[all]{xy}

\usepackage{multirow}
\usepackage{tikz-cd}

\newtheorem{theorem}{Theorem}[section]
\newtheorem{corollary}[theorem]{Corollary}
\newtheorem{lemma}[theorem]{Lemma}
\newtheorem{conjecture}[theorem]{Conjecture}

\theoremstyle{definition}

\DeclareMathOperator{\tr}{tr}

\DeclareMathOperator{\Conf}{Conf}

\DeclareMathOperator{\Khr}{Khr}

\newcommand{\Reals}{\mathbbm{R}}

\newcommand{\Ints}{\mathbbm{Z}}
\newcommand{\CP}{\mathbbm{CP}}
\newcommand{\RP}{\mathbbm{RP}}

\newcommand{\id}{{\mathbbm{1}}}

\newcommand{\timesbar}{{\overline{\times}}}

\newcommand{\wbar}{{\overline{w}}}
\newcommand{\rbar}{{\overline{r}}}
\newcommand{\sbar}{{\overline{s}}}

\newcommand{\thetabar}{{\overline{\theta}}}

\newcommand{\dtilde}{{\tilde{d}}}

\newcommand{\qx}{{\bf{i}}}
\newcommand{\qy}{{\bf{j}}}
\newcommand{\qz}{{\bf{k}}}

\newcommand{\calA}{{\mathcal A}}
\newcommand{\calG}{{\mathcal G}}

\newcommand{\calP}{{\mathcal P}}

\newcommand{\calF}{{\mathcal F}}

\newcommand{\calB}{{\mathcal B}}

\newcommand{\calL}{{\mathcal L}}

\newcommand{\partialbar}{\overline{\partial}}

\DeclareMathOperator{\MCG}{MCG}
\DeclareMathOperator{\PMCG}{PMCG}
\DeclareMathOperator{\PB}{PB}

\DeclareMathOperator{\Out}{Out}

\DeclareMathOperator{\Tw}{Tw}
\DeclareMathOperator{\Ch}{Ch}

\DeclareMathOperator{\etadot}{\dot{\eta}}
\DeclareMathOperator{\epsilondot}{\dot{\epsilon}}

\DeclareMathOperator{\F}{{\mathbbm{F}}}

\newcommand{\abar}{{\overline{a}}}
\newcommand{\bbar}{{\overline{b}}}
\newcommand{\Abar}{{\overline{A}}}
\newcommand{\Bbar}{{\overline{B}}}
\newcommand{\Wbar}{{\overline{W}}}

\newcommand{\gammabar}{{\overline{\gamma}}}

\let\amsamp=&

\begin{document}

\title[Khovanov homology and the Fukaya category of $R^*(T^2,2)$]
      {Khovanov homology and the Fukaya category of the traceless
        character variety for the twice-punctured torus}

\begin{abstract}
We describe a strategy for constructing reduced Khovanov homology for
links in lens spaces by generalizing a symplectic interpretation of
reduced Khovanov homology for links in $S^3$ due to Hedden, Herald,
Hogancamp, and Kirk.
The strategy relies on a partly conjectural description of the Fukaya
category of the traceless $SU(2)$ character variety of the 2-torus
with two punctures.
From a diagram of a 1-tangle in a solid torus, we construct a
corresponding object $(X,\delta)$ in the $A_\infty$ category of
twisted complexes over this Fukaya category.
The homotopy type of $(X,\delta)$ is an isotopy invariant of the
tangle diagram.
We use $(X,\delta)$ to construct cochain complexes for links in
$S^3$ and some links in $S^2 \times S^1$.
For links in $S^3$, the cohomology of our cochain complex reproduces
reduced Khovanov homology, though the cochain complex itself is not the
usual one.
For links in $S^2 \times S^1$, we present results that suggest the
cohomology of our cochain complex may be a link invariant.
\end{abstract}

\author{David Boozer} 

\date{\today}

\maketitle
\setcounter{tocdepth}{1}
\tableofcontents

\section{Introduction}

Khovanov homology is an invariant of oriented links in $S^3$ that
categorifies the Jones polynomial \cite{Jones,Khovanov}.
An important open problem is to generalize Khovanov homology to links
in arbitrary 3-manifolds.
Candidate generalizations have been defined for links in certain
specific 3-manifolds, including links in $I$-bundles over
arbitrary surfaces by Asaeda, Przytycki, and Sikora \cite{Asaeda-1},
links in $S^2 \times S^1$ by Rozansky \cite{Rozansky}, links in
all connected sums of $S^2 \times S^1$ by Willis \cite{Willis}, and
links in $\RP^3$ by Gabrov\v{s}ek \cite{Gabrovsek}.

We describe here a strategy for constructing Khovanov homology for
links in arbitrary lens spaces.
Our strategy is inspired by a symplectic interpretation of Khovanov
homology for links in $S^3$ due to Hedden, Herald, Hogancamp, and
Kirk \cite{Hedden-3}, which originated as an outgrowth of their
project to construct
\emph{pillowcase homology} \cite{Hedden-1,Hedden-2}, a symplectic
counterpart to Kronheimer and Mrowka's singular instanton link
homology \cite{Kronheimer-2,Kronheimer-1,Kronheimer-3}.
The setup for pillowcase homology is as follows.
Given an oriented link $L$ in $S^3$, we consider a Heegaard splitting
\begin{align*}
  (S^3,L) = (B^3,T_0) \cup_{(S^2,4)} (B^3,T_1)
\end{align*}
such that the Heegaard surface $(S^2,4)$ is a 2-sphere that
transversely intersects $L$ in four points, the handlebodies
$(B^3,T_0)$ and $(B^3,T_1)$ are closed 3-balls containing 2-tangles
$T_0$ and $T_1$, and $T_0$ is trivial.
To the Heegaard surface $(S^2,4)$ we associate the irreducible locus
$R^*(S^2,4)$ of the traceless $SU(2)$ character variety for the
2-sphere with four punctures, a symplectic manifold known as the
\emph{pillowcase}.
To the handlebodies $(B^3,T_0)$ and $(B^3,T_1)$ we associate traceless
$SU(2)$ character varieties $R_\pi^\natural(B^3,T_0)$ and $R^*(B^3,T_1)$.
By pulling back $SU(2)$ representations along the inclusions
$(S^2,4) \hookrightarrow (B^3,T_0)$ and
$(S^2,4) \hookrightarrow (B^3,T_1)$, we obtain maps
$R_\pi^\natural(B^3,T_0) \rightarrow R^*(S^2,4)$ and
$R^*(B^3,T_1) \rightarrow R^*(S^2,4)$ of the corresponding
character varieties.
The images of these maps define Lagrangians $L_0$ and $W_1$ in
$R^*(S^2,4)$ that can be viewed as symplectic representations of the
tangles $T_0$ and $T_1$.
Roughly speaking, and ignoring certain technical complications, the
pillowcase homology of $(S^3,L)$ is defined to be the Lagrangian Floer
homology of the pair of Lagrangians $(L_0,W_1)$.

In \cite{Hedden-3}, Hedden, Herald, Hogancamp, and Kirk give a
symplectic interpretation of reduced Khovanov homology by modifying
the construction used to define pillowcase homology.
Instead of working directly with the tangle $T_1$, they consider a
cube of resolutions of a planar projection of $T_1$.
The natural setting to symplectically encode this cube
of resolutions is the twisted Fukaya category of $R^*(S^2,4)$.
To the cube of resolutions they assign an object $(X,\delta)$ of this
$A_\infty$ category, where $X$ consists of shifted copies of
Lagrangians corresponding to planar tangles at the vertices of the
cube and $\delta:X \rightarrow X$ consists of maps of Lagrangians
corresponding to saddles at the edges of the cube.
They prove:

\begin{theorem}(Hedden--Herald--Hogancamp--Kirk
  \cite[Theorem 1.1]{Hedden-3})
\label{theorem:hhhk-invar}
Given a planar projection of $(B^3,T_1)$, there is a corresponding
object $(X,\delta)$ of the twisted Fukaya category of $R^*(S^2,4)$.
The homotopy type of $(X,\delta)$ is an isotopy invariant of $T_1$ rel
boundary.
\end{theorem}

To the trivial tangle $T_0$ they assign an object $(W_0,0)$ of the
twisted Fukaya category.
The morphism spaces of the twisted Fukaya category have the structure
of cochain complexes, so in particular the space of morphisms from
$(W_0,0)$ to $(X,\delta)$ is a cochain complex.
They show that this cochain complex is identical to the usual cochain
complex for the reduced Khovanov homology $\Khr(L)$ of $L$, thus
proving:

\begin{theorem}(Hedden--Herald--Hogancamp--Kirk
  \cite[Corollary 1.2]{Hedden-3})
\label{theorem:hhhk-khr}
We have an isomorphism of bigraded vector spaces
\begin{align*}
  \Khr(L) \rightarrow H^*(\hom((W_0,0), (X,\delta))),
\end{align*}
where the $\hom$ space is taken within the twisted Fukaya category of
$R^*(S^2,4)$.
\end{theorem}

Theorem \ref{theorem:hhhk-khr} shows that the Fukaya
category of $R^*(S^2,4)$ knows about Khovanov homology.
Our strategy for generalizing Khovanov homology is based on
generalizing this observation.
Theorem \ref{theorem:hhhk-khr} is formulated in terms of $R^*(S^2,4)$
because this is the character variety corresponding to the Heegaard
surface $(S^2,4)$ for the chosen Heegaard splitting of $(S^3,L)$.
More generally, one can split a 3-manifold $Y$ containing a
link $L$ along a Heegaard surface $(\Sigma_g,n)$ of genus $g$ that
intersects $L$ in $n$ points.
In light of Theorem \ref{theorem:hhhk-khr}, one might hope that the
Fukaya category of the corresponding character variety
$R^*(\Sigma_g,n)$ contains information that could be used to
generalize Khovanov homology to links in $Y.$

As a first step towards this goal, we consider the case of links in
lens spaces.
Given an oriented link $L$ in a lens space $Y$, we consider a Heegaard
splitting
\begin{align*}
  (Y,L) = (U_0,T_0) \cup_{(T^2,2)} (U_1,T_1)
\end{align*}
such that the Heegaard surface $(T^2,2)$ is a 2-torus that
transversely intersects $L$ in two points, the handlebodies
$(U_0,T_0)$ and $(U_1,T_1)$ are solid tori containing 1-tangles $T_0$
and $T_1$, and $T_0$ is trivial.
To the Heegaard surface $(T^2,2)$ we associate the irreducible locus
$R^*(T^2,2)$ of the traceless $SU(2)$ character variety for the torus
with two punctures.
We project $(U_1,T_1)$ onto the plane to obtain a 1-tangle diagram $T$
in the annulus, as shown in Figure \ref{fig:tangle-diagram}.
One of the two boundary points of the 1-tangle diagram lies on the
inner bounding circle of the annulus and the other lies on the outer
bounding circle.
We define the \emph{loop number} $\ell$ of the tangle diagram to be
the number of arcs that loop around the annulus.
We would like to encode the cube of resolutions of the tangle diagram
$T$ as an object $(X,\delta)$ of the twisted Fukaya category of
$R^*(T^2,2)$.

\begin{figure}
  \centering
  \includegraphics[scale=0.60]{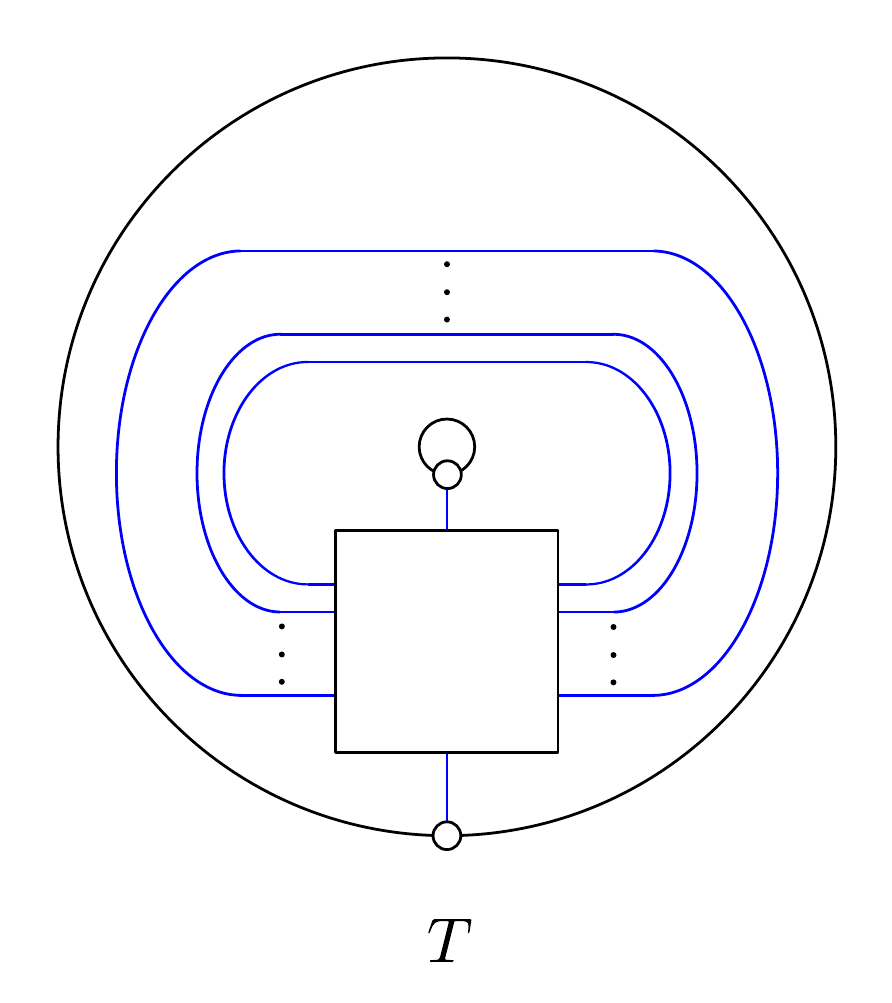}
  \caption{
  \label{fig:tangle-diagram}
    A planar projection of $(U_1,T_1)$ yields a 1-tangle diagram $T$
    in the annulus.
    One of the two boundary points of $T$ lies on the inner bounding
    circle of the annulus and the other lies on the outer bounding
    circle.
    For simplicity, we will usually omit these bounding circles when
    depicting tangle diagrams.
  }
\end{figure}

Our first task is thus to understand the relevant structure of the
Fukaya category of $R^*(T^2,2)$.
We can explicitly describe the Lagrangians that are needed to
construct the object $X$ corresponding to the planar tangles at the
vertices of the cube of resolutions, and we can describe the morphism
spaces for pairs of Lagrangians, the generators of which are needed to
construct the endomorphism $\delta:X \rightarrow X$.
However, to verify that $(X,\delta)$ is indeed an object of the
twisted Fukaya category, and to construct cochain complexes
from the morphism spaces of the twisted Fukaya category, we also
need to know certain $A_\infty$ operations, which are obtained by
counting pseudo-holomorphic disks in $R^*(T^2,2)$, and these
operations are more difficult to describe.
In the case of $(S^2,4)$, the character variety $R^*(S^2,4)$ is
two-dimensional, so the relevant $A_\infty$ operations can be computed
combinatorially using the Riemann mapping theorem.
But $R^*(T^2,2)$ is four-dimensional, and with our current methods we
are unable to compute the operations we need.
Instead, we conjecture these operations based on the information we do
have.
One useful guide in this process is that the Fukaya categories of
$R^*(T^2,2)$ and $R^*(S^2,4)$ appear to be closely related.
For example, we show that $R^*(S^2,4)$ is a symplectic submanifold of
$R^*(T^2,2)$.
The $A_\infty$ operations that we conjecture for $R^*(T^2,2)$ are
natural generalizations of the known $A_\infty$ operations for
$R^*(S^2,4)$.
Our hope is that the information we have regarding the Lagrangians and
morphism spaces, together with the conjectured $A_\infty$ operations,
will provide the necessary clues to generalize reduced Khovanov
homology.

The first step towards this goal is to generalize Theorem
\ref{theorem:hhhk-invar} for $R^*(S^2,4)$.
We prove:

\begin{theorem}
\label{theorem:intro-invar}
If the $A_\infty$ operations of $R^*(T^2,2)$ are as conjectured, then
given a planar projection $T$ of $(U_1,T_1)$ there are
two corresponding objects $(X,\delta_+)$ and $(X,\delta_-)$ of the
twisted Fukaya category of $R^*(T^2,2)$.
The homotopy type of $(X,\delta_\pm)$ is an invariant of $T_1$ rel
boundary.
\end{theorem}

The twisted complex $(X, \delta_\pm)$ has several unusual features.
The complex is constructed from a cube of resolutions of the 1-tangle
diagram $T$ whose vertices are planar 1-tangles.
Each planar 1-tangle contains an  arc component connecting the
boundary points on the outer and inner bounding circles of the
annulus.
By orienting the arc component so it is directed outer-to-inner, we
can distinguish between saddles that split/merge circle components
from/with the left and right sides of the arc component.
In the twisted complex $(X,\delta_\pm)$, the differentials
corresponding to such \emph{left} and \emph{right} saddles are
distinct.
Also, the circle components that are split or merged by left and right
saddles can enclose additional circle components, and the
differentials corresponding to left and right saddles act nontrivially
on vector space factors corresponding to the enclosed circle
components. So, there is a sense in which the differential $\delta$
for our twisted complex is \emph{nonlocal}.

Our next task is to generalize Theorem \ref{theorem:hhhk-khr} by using
the twisted objects $(X,\delta_+)$ and $(X,\delta_-)$ to reproduce the
reduced Khovanov homology for links in $S^3$.
Given an oriented 1-tangle diagram $T$ in the annulus, we can
construct links $L_{T^+}$ and $L_{T^-}$ in $S^3$ by closing $T$ with an
overpass arc $A_+$ or underpass arc $A_-$, as shown in Figure
\ref{fig:s3-lp-lm}.
The handlebodies for the arcs $A_+$ and $A_-$ correspond to
Lagrangians $W_+$ and $W_-$ in $R^*(T^2,2)$, and hence to
twisted objects $(W_+,0)$ and $(W_-,0)$.
We define bigraded cochain complexes
\begin{align*}
  (C_\pm, \partial_\pm) &= \hom((W_\pm,0), (X,\delta_\pm)).
\end{align*}
We use the conjectured $A_\infty$ relations to explicitly construct
$(C_\pm, \partial_\pm)$ from the cube of resolutions of $T$.
The cohomology of the cochain complex $(C_\pm, \partial_\pm)$ is shown
in \cite{Boozer-3} to reproduce the reduced Khovanov homology for the
link $L_{T^\pm}$:

\begin{figure}
  \centering
  \includegraphics[scale=0.60]{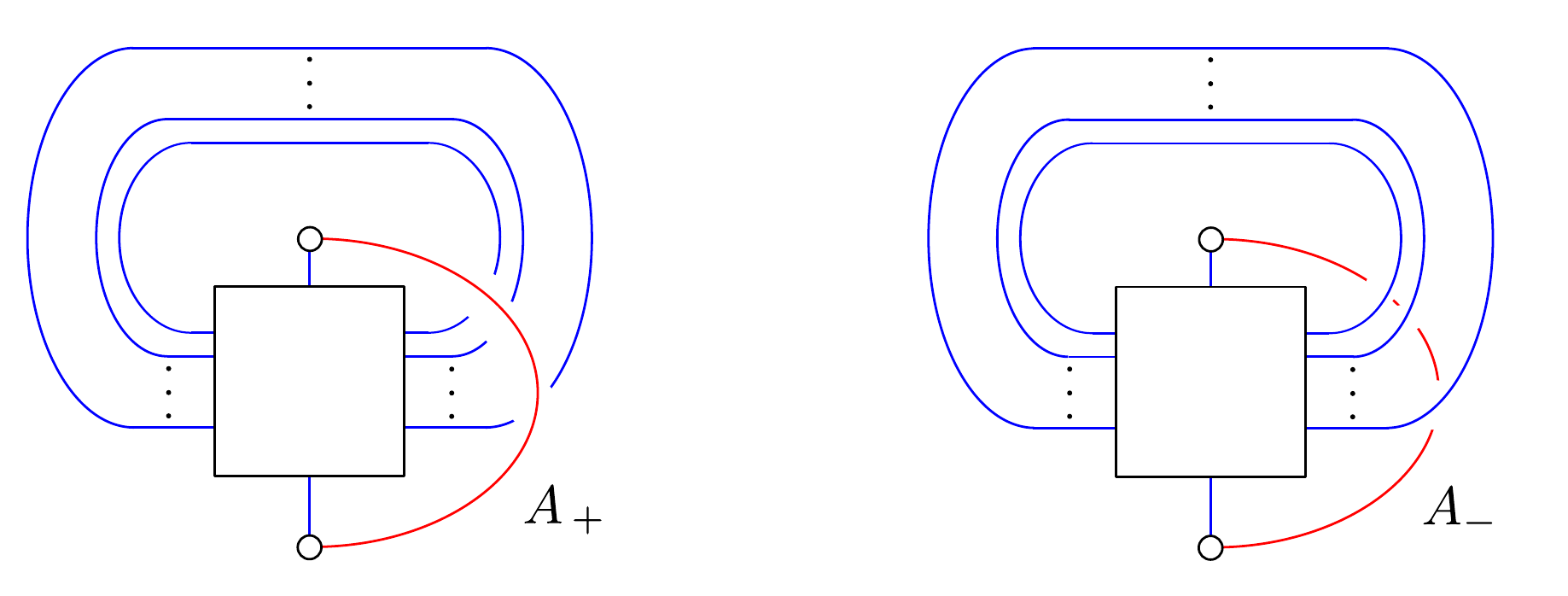}
  \caption{
    \label{fig:s3-lp-lm}
    We can close a 1-tangle diagram $T$ with an overpass arc $A_+$ or
    underpass arc $A_-$ to obtain a link diagram.
    The link diagram can be interpreted as the planar projection of
    either a link $L_{T^\pm}$ in $S^3$ or a link $L_{T^0}$ in
    $S^2 \times S^1$.
  }
\end{figure}

\begin{theorem}(Boozer \cite{Boozer-3})
\label{theorem:intro-khr-s3}
There is an isomorphism of bigraded vector spaces
\begin{align*}
  \Khr(L_{T^\pm}) \rightarrow H^*((C_\pm, \partial_\pm)).
\end{align*}
\end{theorem}

The cochain complex $(C_\pm, \partial_\pm)$ is typically not the same
as the usual cochain complex for the reduced Khovanov homology of
$L_{T^\pm}$, though the two complexes do agree for tangle diagrams with
loop number 0.
For example, the complex $(C_\pm, \partial_\pm)$ has
\emph{long differentials} corresponding to pairs of successive saddles
in the cube of resolutions of $T$.
The complex $(C_\pm, \partial_\pm)$ can be smaller than the usual
complex, since crossings between the overpass or underpass arc and the
tangle diagram are not included in the cube of resolutions.
Example cochain complexes for the unknot and right trefoil are given
in Sections \ref{ssec:example-unknot} and
\ref{ssec:example-trefoil}.

From Theorem \ref{theorem:intro-khr-s3}, we obtain the following
generalization of Theorem \ref{theorem:hhhk-khr}:

\begin{theorem}
If the $A_\infty$ operations of $R^*(T^2,2)$ are as conjectured, then
there are isomorphisms of bigraded vector spaces
\begin{align*}
  &\Khr(L_{T^\pm}) \rightarrow H^*(\hom((W_\pm,0), (X,\delta_\pm))).
\end{align*}
where the $\hom$ space is taken within the twisted Fukaya category of
$R^*(T^2,2)$.
\end{theorem}

We next turn our attention to links in $S^2 \times S^1$.
Given an oriented 1-tangle diagram $T$ in the annulus, we can
construct links $L_{T^+}'$ and $L_{T^-}'$ in $S^2 \times S^1$ by
closing $T$ with an overpass arc $A_+$ or underpass arc $A_-$, as
shown in Figure \ref{fig:s3-lp-lm}.
The links $L_{T^+}'$ and $L_{T^-}'$ are isotopic, since one can
isotope $A_+$ to $A_-$ by moving $A_+$ around the $S^2$ factor, and
we will let $L_{T^0}$ denote either of these isotopic links.
The handlebodies for the arcs $A_+$ and $A_-$ correspond to the same
Lagrangian $W_0$ in $R^*(T^2,2)$.
We define a bigraded cochain complex
\begin{align*}
  (C_0, \partial_0) = \hom((W_0,0), (X,\delta_-)).
\end{align*}
For a 1-tangle diagram $T$ with loop number $\ell \in \{0,2\}$, we use
the conjectured $A_\infty$ relations to explicitly construct
$(C_0, \partial_0)$ from the cube of resolutions of $T.$

One might hope that the cohomology of $(C_0,\partial_0)$ depends only
on the isotopy class of the link $L_{T^0}$ and not on its description
as the closure of the particular tangle diagram $T$.
From Theorem \ref{theorem:intro-invar}, it follows that the cohomology
is an isotopy invariant of the tangle diagram $T$.
But it is possible for nonisotopic tangle diagrams $T_1$ and $T_2$ to
yield isotopic links $L_{T_1^0}$ and $L_{T_2^0}$ in $S^2 \times S^1$,
and example calculations show that in this situation the cohomology
for $T_1$ and $T_2$ need not be the same.
In all the examples we have checked, however, the dependence on the
tangle diagram is reflected only in the bigradings of generators, and
the cohomology does agree if we collapse bigradings from $\Ints$ to
$\Ints_2$.
Several specific examples are described in Sections
\ref{ssec:example-s2-s1-1}, \ref{ssec:example-s2-s1-2}, and
\ref{ssec:example-s2-s1-3}.
On the basis of such examples, we make the following conjecture:

\begin{conjecture}
\label{conj:intro-s2-s1}
The cohomology of the cochain complex $(C_0,\partial_0)$ with bigradings
collapsed from $\Ints$ to $\Ints_2$ is an isotopy invariant of the
link $L_{T^0}$.
\end{conjecture}

We prove an invariance result in support of Conjecture
\ref{conj:intro-s2-s1}.
Given a 1-tangle diagram $T$ with loop number 2, we define a 1-tangle
diagram $T_\tau$ by adding a full twist to $T$ as shown in Figure
\ref{fig:diagram-twist}.
The diagrams $T$ and $T_\tau$ are typically not isotopic, but the
corresponding links $L_{T^0}$ and $L_{T_\tau^0}$ in $S^2 \times S^1$
are isotopic, since one can unwind the full twist in $T_\tau$ by
moving one of the strands around the $S^2$ factor as shown in Figure
\ref{fig:s2-s1-twist}.
Let $(C_0,\partial_0)$ and $(C_0^\tau, \partial_0^\tau)$ denote the
cochain complex for $T$ and $T_\tau.$
We prove:

\begin{figure}
  \centering
  \includegraphics[scale=0.60]{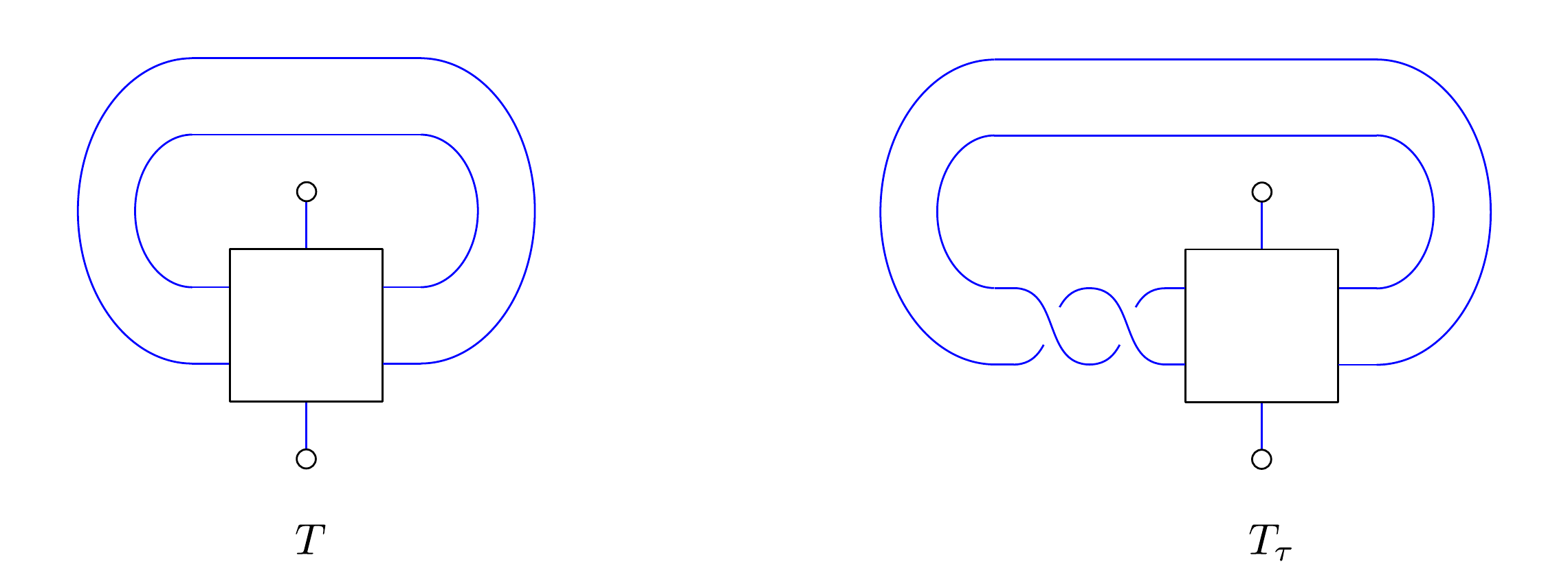}
  \caption{
    \label{fig:diagram-twist}
    Given a 1-tangle diagram $T$ with loop number 2, we add a full
    twist to obtain a 1-tangle diagram $T_\tau$.
  }
\end{figure}

\begin{figure}
  \centering
  \includegraphics[scale=0.60]{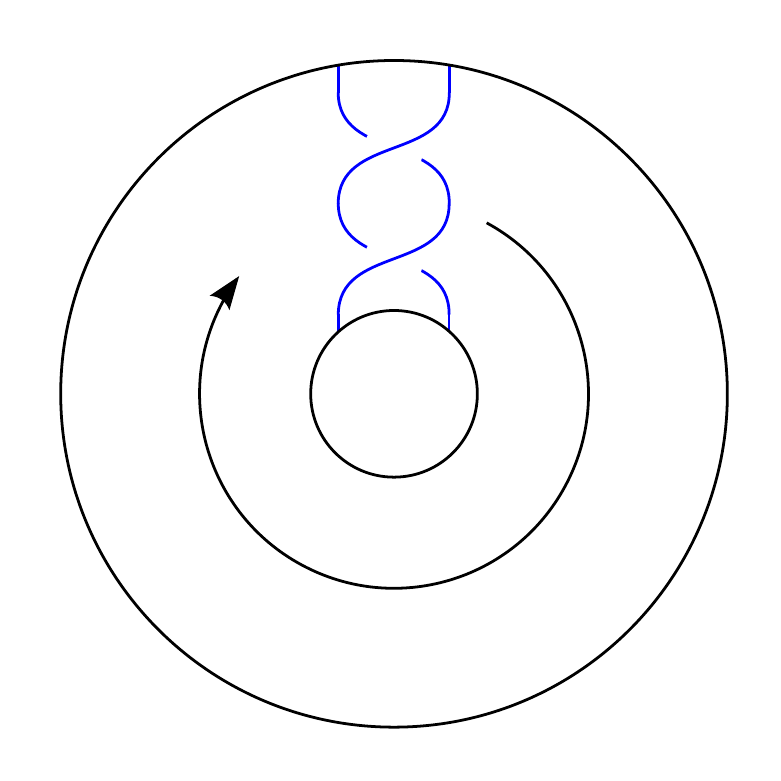}
  \caption{
    \label{fig:s2-s1-twist}
    The inner and outer circles represent two copies of $S^2$ that are
    identified to give $S^2 \times S^1$.
    We can unwind the full twist in the indicated link in $S^2 \times
    S^1$ by moving one of the strands around the $S^2$ factor.
  }
\end{figure}

\begin{theorem}
\label{theorem:intro-s2xs1-invar}
With bigradings collapsed from $\Ints$ to $\Ints_2$, there is an
isomorphism of bigraded vector spaces
\begin{align*}
  H^*((C_0,\partial_0)) \rightarrow H^*((C_0^\tau, \partial_0^\tau)).
\end{align*}
\end{theorem}

If the cohomology of $(C_0, \partial_0)$ with bigradings collapsed to
$\Ints_2$ is indeed an isotopy invariant of $L_{T^0}$, there is a
sense in which it is a natural generalization of the reduced Khovanov
homology for links in $S^3$.
Consider a tangle diagram $T$ with loop number 0.
The corresponding link $L_{T^0}$ can be contained inside an open
3-ball in $S^2 \times S^1$.
By collapsing the complement of the open 3-ball to a point, we obtain
a link $L_{T^0}'$ in $S^3$.
We prove:

\begin{theorem}
\label{theorem:intro-s2xs1-khr-0}
There is an isomorphism of bigraded vector spaces
\begin{align*}
  V \otimes \Khr(L_{T^0}') \rightarrow H^*((C_0,\partial_0)),
\end{align*}
where $V$ is a two-dimensional bigraded vector space.
\end{theorem}

Theorem \ref{theorem:intro-s2xs1-khr-0} does not hold at the level of
cochain complexes; that is, the cochain complex $(C_0,\partial_0)$ is
generally not the tensor product of $V$ with the usual cochain
complex for the reduced Khovanov homology of $L_{T^0}'$.

If the cohomology of $(C_0, \partial_0)$ with bigradings collapsed to
$\Ints_2$ is not an isotopy invariant of $L_{T^0}$, it may still be of
interest due to a possible relationship to Kronheimer and Mrowka's
singular instanton link homology.
Given a knot $K$ in $S^3$, Kronheimer and Mrowka show in
\cite{Kronheimer-2} that there is a spectral sequence whose $E_2$ page
is the reduced Khovanov homology of the mirror knot $K^m$ and
that converges to the singular instanton homology of $K$.
Singular instanton homology is closely related to Lagrangian Floer
homology in character varieties of punctured surfaces.
For example, generating sets for the singular instanton homology of
knots in a 3-manifold $Y$ can be constructed from Lagrangian
intersections in $R^*(S^2,4)$ when $Y = S^3$, as described in
\cite{Hedden-1}, and in $R^*(T^2,2)$ when $Y$ is a lens space, as
described in \cite{Boozer-2}.
Thus, one might conjecture that there is a spectral sequence from the
cohomology of $(C_0,\partial_0)$ to the singular instanton homology of
the corresponding link in $S^2 \times S^1$ that would generalize
Kronheimer and Mrowka's spectral sequence for links in $S^3$.

The paper is organized as follows.
In Section \ref{sec:prelim},
we review the material we will need regarding the group $SU(2)$ and
$A_\infty$ categories.
In Section \ref{sec:fukaya},
we describe the character variety $R^*(T^2,2)$ and its Fukaya
category.
In Section \ref{sec:complexes},
we show how to construct an object $(X,\delta)$ in the twisted Fukaya
category of $R^*(T^2,2)$ from an oriented 1-tangle diagram $T$ in the
annulus.
In Section \ref{sec:invariance},
we show that the homotopy type of $(X,\delta)$ is an isotopy invariant
of the corresponding tangle diagram $T$.
In Section \ref{sec:s3},
we prove Theorem \ref{theorem:intro-khr-s3} for links in $S^3$.
In Section \ref{sec:s2xs1},
we prove Theorems \ref{theorem:intro-s2xs1-invar} and
\ref{theorem:intro-s2xs1-khr-0} for links in $S^2 \times S^1$.
In Appendix \ref{sec:pillowcase},
we describe the character variety $R^*(S^2,4)$ and compare it with
$R^*(T^2,2)$.
In Appendix \ref{sec:chain-homotopy-equivalences},
we fill in some technical details in the proof of Theorem
\ref{theorem:intro-s2xs1-invar}.

\section{Preliminary material}
\label{sec:prelim}

\subsection{The group $SU(2)$}

We briefly review here some facts about the group $SU(2)$.
We define the Pauli spin matrices:
\begin{align*}
  \sigma_x &=
  \left(\begin{array}{cc}
    0 & 1 \\
    1 & 0
  \end{array}\right), &
  \sigma_y &=
  \left(\begin{array}{cc}
    0 & -i \\
    i & 0
  \end{array}\right), &
  \sigma_z &=
  \left(\begin{array}{cc}
    1 & 0 \\
    0 & -1
  \end{array}\right).
\end{align*}
We define $SU(2)$ matrices $\qx$, $\qy$, and $\qz$:
\begin{align*}
  \qx &= -i\sigma_x, &
  \qy &= -i\sigma_y, &
  \qz &= -i\sigma_z.
\end{align*}
These matrices satisfy the quaternion multiplication laws:
\begin{align*}
  \qx^2 = \qy^2 = \qz^2 = \qx \qy \qz = -\id.
\end{align*}
Any $SU(2)$ matrix $A$ can be uniquely expressed as
\begin{align*}
  A = t\,\id + x\,\qx + y\,\qy + z\,\qz,
\end{align*}
where
$(t,x,y,z) \in S^3 =
\{(t,x,y,z) \in \Reals^4 \mid t^2 + x^2 + y^2 + z^2 = 1\}$.
We can thus identify $SU(2)$ with the unit quaternions.
We will refer to $t\,\id$ and $x\,\qx + y\,\qy + z\,\qz$ as the
\emph{scalar} and \emph{vector}
parts of the matrix $A$.
Note that $\tr A = 2t$, so traceless $SU(2)$ matrices are precisely
those for which the scalar part is zero.

We define a surjective group homomorphism
$SU(2) \rightarrow SO(3)$ by $g \mapsto (\hat{v} \mapsto \hat{v}')$,
where the unit vectors $\hat{v} = (v_x, v_y, v_z)$ and
$\hat{v}' = (v_x', v_y', v_z')$ are related by
\begin{align}
  \nonumber  
  g(v_x\,\qx + v_y\,\qy + v_z\,\qz)g^{-1} =
  v_x'\,\qx + v_y'\,\qy + v_z'\,\qz.
\end{align}
In general, conjugating an arbitrary $SU(2)$ matrix preserves the
scalar part of the matrix and rotates the vector part of the matrix:
\begin{align}
  \nonumber  
  g(t\,\id + r_x\,\qx + r_y\,\qy + r_z\,\qz)g^{-1} =
  t\,\id + r_x'\,\qx + r_y'\,\qy + r_z'\,\qz,
\end{align}
where $(r_x',r_y',r_z')$ is given by multiplying
$(r_x,r_y,r_z)$ by the $SO(3)$ matrix corresponding to $g \in SU(2)$.

\subsection{$A_\infty$ categories}
\label{ssec:a-infinity}

We briefly review here the material we will need regarding $A_\infty$
categories.
For a more extensive discussion of this material, see Section 2 of
\cite{Hedden-3} or the book \cite{Seidel-Fukaya}.
We define $\F$ to be the field of two elements.

An \emph{$A_\infty$ category} consists of a set of objects
$\{a_i \}$, for each pair of objects $(a_i, a_j)$ an integer-graded
$\F$-vector space $\hom_\calA(a_i, a_j)$ called the
\emph{space of morphisms from $a_i$ to $a_j$}, and for each integer
$m \geq 1$ a linear \emph{structure map}, also called an
\emph{operation}, with grading $2 - m$:
\begin{align*}
  \mu_{\calA}^m: \hom_\calA(a_{m-1}, a_m) \otimes \cdots \otimes
  \hom_\calA(a_0, a_1) \rightarrow \hom_\calA(a_0, a_m).
\end{align*}
The structure maps are required to satisfy certain relations, which we
do not describe here.
The first relation asserts that the endomorphism
$\mu_\calA^1:(a_0,a_1) \rightarrow (a_0,a_1)$ squares to zero:
\begin{align*}
  \mu_\calA^1(\mu_\calA^1(f)) = 0.
\end{align*}
If $\mu_{\calA}^m = 0$ for $m \neq 2$, the relations reduce to a
single relation asserting the associativity of $\mu_\calA^2$:
\begin{align*}
  \mu_\calA^2(\mu_\calA^2(h,g),f) = \mu_\calA^2(h,\mu_\calA^2(g,f)).
\end{align*}

An important example of an $A_\infty$ category is the category $\Ch$
of cochain complexes over $\F$.
The objects of $\Ch$ are pairs $(C,d)$ consisting of an
integer-graded $\F$-vector space $C$ and an endomorphism
$d:C \rightarrow C$ with grading 1 such that
$d \circ d = 0$.
We define $\hom_{\Ch}((C_1,d_1),\, (C_2,d_2))$ to be the
graded $\F$-vector space of all linear maps $C_1 \rightarrow C_2$.
We define structure maps $\mu_{\Ch}^1$ and $\mu_{\Ch}^2$ such that
\begin{align*}
  \mu_{\Ch}^1(f) &= d_2 \circ f + f \circ d_1, &
  \mu_{\Ch}^2(g,f) &= g \circ f
\end{align*}
for $f:(C_1,d_1) \rightarrow (C_2,d_2)$ and
$g:(C_2,d_2) \rightarrow (C_3, d_3)$.
We define $\mu_{\Ch}^m = 0$ for $m > 2$.

Given an $A_\infty$ category $\calA$, we define an
$A_\infty$ category $\Sigma\calA$ called the
\emph{additive enlargement} of $\calA$.
The objects of $\Sigma\calA$ are pairs
$(I,\, \{(a_i, V_i) \mid i \in I\})$, where
$I$ is a finite indexing set and for each $i \in I$ the pair
$(a_i, V_i)$ consists of an object $a_i$ of $\calA$ and an
integer-graded $\F$-vector space $V_i$.
For simplicity, we will denote such an object of $\Sigma \calA$ as
\begin{align*}
  \bigoplus_{i \in I} (a_i \otimes V_i).
\end{align*}
The morphism spaces of $\Sigma\calA$ are given by
\begin{align*}
  \hom_{\Sigma \calA}( \bigoplus_{i \in I} a_i \otimes V_i,\,
  \bigoplus_{j \in J}  a_j \otimes V_j ) =
  \bigoplus_{i \in I,j \in J} \hom_{\calA}(a_i, a_j) \otimes
  \hom_{\F}(V_i,V_j).
\end{align*}
The $A_\infty$ operation $\mu_{\Sigma \calA}^n$ is given by applying
the operation $\mu_{\calA}^n$ to the $\hom_{\calA}(a_i, a_j)$
factors and composing the $\hom_{\F}(V_i,V_j)$ factors.

Given an $A_\infty$ category $\calA$, we define an $A_\infty$ category
$\Tw \calA$ called the \emph{category of twisted complexes over}
$\calA$.
The objects of $\Tw \calA$ are pairs $(X,\delta)$, where
$X$ is an object of $\Sigma \calA$ and
$\delta \in \hom_{\Sigma A}(X,X)$ is homogeneous with grading 1 and
satisfies
\begin{align*}
  \sum_m \mu_{\Sigma \calA}^m(\delta, \cdots, \delta) = 0.
\end{align*}
We also require that $X$ admit a filtration with respect to which
$\delta$ is strictly lower triangular.
The morphism spaces of $\Tw \calA$ are given by
\begin{align*}
  \hom_{\Tw \calA}((X,\delta_X),\,(Y,\delta_Y)) =
  \hom_{\Sigma \calA}(X,Y).
\end{align*}
The $A_\infty$ operations $\mu_{\Tw \calA}^m$ for $\Tw \calA$ can be
expressed in terms of the $A_\infty$ operations $\mu_\calA^m$ for
$\calA$.
If $\mu_\calA^m = 0$ for $m \neq 2$, then
$\mu_{\Tw \calA}^m = 0$ for $m > 2$ and we have
\begin{align*}
  \mu_{\Tw \calA}^1(F) &=
  \mu_{\Sigma\calA}^2(\delta_Y,F) + 
  \mu_{\Sigma\calA}^2(F,\delta_X), &
  \mu_{\Tw \calA}^2(G,F) &=
  \mu_{\Sigma \calA}^2(G,F)
\end{align*}
for $F:(X,\delta_X) \rightarrow (Y,\delta_Y)$ and
$G:(Y,\delta_Y) \rightarrow (Z,\delta_Z)$.

An \emph{$A_\infty$ functor} $\calF:\calA \rightarrow \calB$ of
$A_\infty$ categories $\calA$ and $\calB$ is an assignment of an
object $\calF(a_i)$ of $\calB$ for each object $a_i$ of $\calA$, and
for each integer $m \geq 1$ a linear map of grading $1-m$:
\begin{align*}
  \calF^m: \hom_\calA(a_{m-1}, a_m) \otimes \cdots \otimes
  \hom_\calA(a_0, a_1) \rightarrow
  \hom_{\calB}(\calF(a_0),\calF(a_m)).
\end{align*}
The maps $\calF^m$ are required to satisfy certain conditions, which
we do not describe here.

Given an object $b$ of an $A_\infty$ category $\calA$, one can
define an $A_\infty$ functor $\calG_b:\calA \rightarrow \Ch$.
On objects, we have
\begin{align*}
  \calG_b(a) = (\hom_\calA(b,a),\mu_\calA^1).
\end{align*}
The linear maps $\calG_b^m$ are given by
\begin{align*}
  &\calG_b^m:
  \hom_\calA(a_{m-1},a_m) \otimes \cdots \otimes \hom_\calA(a_0,a_1)
  \rightarrow
  \hom_{\Ch}(\calG_b(a_0),\, \calG_b(a_m)), \\
  &x_m \otimes \cdots \otimes x_1 \mapsto
  \mu_\calA^{m+1}(x_m, \cdots, x_1,-).
\end{align*}
One can extend the $A_\infty$ functor $\calG_b: \calA \rightarrow \Ch$
to an $A_\infty$ functor $\Tw \calG_b: \Tw \calA \rightarrow \Ch$.
On objects, we have
\begin{align*}
  (\Tw \calG_b)((X,\delta)) = (C, d+\partial),
\end{align*}
where the graded $\F$-vector space $C$ and endomorphisms
$d, \partial:C \rightarrow C$ are defined as follows.
The object $X$ of $\Sigma \calA$ has the form
\begin{align*}
  X = \bigoplus_{i \in I} a_i \otimes V_i.
\end{align*}
For each $i \in I$, we define a cochain complex
\begin{align*}
  (C_i, d_i) = \calG_b(a_i) = (\hom_\calA(b,a_i),\mu_\calA^1).
\end{align*}
We define
\begin{align*}
  C &= \bigoplus_{i \in I} C_i \otimes V_i, &
  d &= \sum_{i \in I} d_i \otimes \id_{V_i}, &
  \partial &=
  \sum_{m=1}^\infty \calG_b^m(\delta, \cdots, \delta) =
  \sum_{m=2}^\infty \mu_{\Sigma \calA}^m(\delta, \cdots, \delta, -).
\end{align*}
The differential for the cochain complex $(C, d+\partial)$
combines the differential $d$, which is due to the structure map
$\mu_\calA^1$ for $\calA$, and the map $\partial$, which is
due to the endomorphism $\delta:X \rightarrow X$ that constitutes part
of the data of the twisted object $(X,\delta)$.

\section{The Fukaya category of $R^*(T^2,2)$}
\label{sec:fukaya}

\subsection{The character variety $R^*(T^2,2)$}

Consider a surface $S$ with $n$ marked points $p_1, \cdots, p_n$.
We define a punctured surface
$S^\circ = S - \{p_1, \cdots, p_n\}$.
We define the traceless $SU(2)$ character variety $R(S,n)$ to be the
set of conjugacy classes of $SU(2)$ representations of the fundamental
group $\pi_1(S^\circ)$ that map loops around the punctures to
traceless matrices.
We say that a representation $\rho:\pi_1(S^\circ) \rightarrow SU(2)$
is \emph{irreducible} if the image of $\rho$ is a nonabelian subgroup
of $SU(2)$ and \emph{reducible} otherwise.
We define the \emph{irreducible locus} $R^*(S,n)$ of $R(S,n)$ to be
the set of conjugacy classes of irreducible representations.
The set $R^*(S,n)$ has the structure of a smooth manifold with a
canonical symplectic form.

We consider here the character variety $R(T^2,2)$, whose topology is
described in \cite{Boozer-2}.
We define fundamental cycles $A$ and $B$ of $T^2$ and loops $a$ and
$b$ around the punctures $p_1$ and $p_2$.
A presentation of the fundamental group of $T^2 - \{p_1,p_2\}$ is
\begin{align*}
  \pi_1(T^2 - \{p_1, p_2\}) = \langle A, B, a, b \mid
  [A,B]ab = 1 \rangle.
\end{align*}
We can specify a representation
$\pi_1(T^2 - \{p_1, p_2\}) \rightarrow SU(2)$ by specifying the images
of the generators $(A,B,a,b)$.
The character variety $R(T^2,2)$ is thus given by
\begin{align*}
  R(T^2,2) = \{(A,B,a,b) \in SU(2)^4 \mid
  [A,B]ab = \id,\, \tr a = \tr b = 0\}/{\textup{conjugation}},
\end{align*}
where for simplicity we use the same notation for generators of
$\pi_1(T^2 - \{p_1, p_2\})$ and their images under a given
representation.
The character variety $R(T^2,2)$ and its irreducible locus
$R^*(T^2,2)$ can be also be interpreted as moduli spaces of semistable
and stable parabolic bundles over an elliptic curve with two marked
points.
Using this interpretation, one can show that $R(T^2,2)$ and
$R^*(T^2,2)$ can be given the structure of complex manifolds
isomorphic to $(\CP^1)^2$ and the complement of a holomorphically
embedded elliptic curve in $(\CP^1)^2$, respectively \cite{Vargas}.
In particular, the irreducible locus $R^*(T^2,2)$ has the structure of
a smooth manifold of real dimension 4.

It is useful to define a subset
\begin{align*}
  R_3(T^2,2) = \{ab = \id\}
\end{align*}
of $R(T^2,2)$.
Note that $ab = \id$ is invariant under conjugation, so
$R_3(T^2,2)$ is well-defined.
We define $R_3^*(T^2,2)$ to be the irreducible locus of $R_3(T^2,2)$.
As we will see, the set $R_3^*(T^2,2)$ has the structure of a
submanifold of $R^*(T^2,2)$ of real dimension 3.

We define coordinates on $R_3(T^2,2)$ as follows.
Given a point $[A,B,a,b]$ of $R_3(T^2,2)$, we can always choose a
representative of the form
\begin{align}
  \label{eqn:alpha-beta-s}
  A &= \cos\alpha\,\id + \sin\alpha\,\qz, &
  B &= \cos\beta\,\id + \sin\beta\,\qz, &
  a &= b^{-1} = \cos s\,\qx + \sin s\,\qz.
\end{align}
The parameters $(\alpha,\beta,s)$ are subject to $2\pi$-periodicity
identifications, as well as the identifications
\begin{align*}
  (\alpha,\beta,s) &\sim (-\alpha,-\beta,-s), &
  (\alpha,\beta,s) &\sim (\alpha,\beta,\pi-s),
\end{align*}
which follow from the fact that we can conjugate
$(A,B,a,b)$ by $\qx$ or $\qz$ to obtain an equivalent representative.
Conjugation by $\qy$ does not yield an independent
identification, since $\qx \qy \qz = -\id$.
For points $(\alpha,\beta,s)$ for which
$\sin\alpha = \sin\beta = 0$, we have the additional identification
\begin{align*}
  (\alpha,\beta,s) \sim (\alpha,\beta,0),
\end{align*}
since in this case we can conjugate $(A,B,a,b)$ so as to force
$a = b^{-1} = \qz$.
We will use $(\alpha,\beta,s)$, subject to these identifications,
as standard coordinates on $R_3(T^2,2)$.

It is also useful to define a subset
\begin{align*}
  R_2(T^2,2) &= R_3(T^2,2) \cap \{s = 0\}
\end{align*}
of $R(T^2,2)$.
We define $R_2^*(T^2,2)$ to be the irreducible locus of $R_2(T^2,2)$.
Setting $s=0$ in equation (\ref{eqn:alpha-beta-s}) for the standard
coordinates $(\alpha,\beta,s)$ on $R_3(T^2,2)$, we find that points
$[A,B,a,b]$ of $R_2(T^2,2)$ have the form
\begin{align}
  \label{eqn:alpha-beta}
  A &= \cos\alpha\,\id + \sin\alpha\,\qz, &
  B &= \cos\beta\,\id + \sin\beta\,\qz, &
  a &= b^{-1} = \qx.
\end{align}
The standard coordinates $(\alpha,\beta,s)$ on $R_3(T^2,2)$ restrict
to coordinates $(\alpha,\beta)$ on $R_2(T^2,2)$.
The identifications for $(\alpha,\beta,s)$ restrict to
$2\pi$-periodicity identifications for $(\alpha,\beta)$,
as well as the identification
\begin{align*}
  (\alpha,\beta) \sim (-\alpha, -\beta).
\end{align*}
From the identifications on the coordinates $(\alpha,\beta)$, it
follows that $R_2(T^2,2)$ is homeomorphic to $S^2$.
From equation (\ref{eqn:alpha-beta}), it follows that the reducible
locus of $R_2(T^2,2)$ is the set $\{\sin\alpha = \sin\beta = 0\}$,
which consists of four points.
Thus the irreducible locus $R_2^*(T^2,2)$ is diffeomorphic to a sphere
minus four points, as shown in Figure \ref{fig:lagrangians-rectangle}.

To visualize $R_3(T^2,2)$, we define a homeomorphic space $Y$:
\begin{align*}
  Y = \{(\alpha,\beta,z) \mid
  \alpha \in [0,2\pi],\, \beta \in [0,\pi],\,
  |z| \leq \sin^2\alpha + \sin^2\beta \}/{\sim},
\end{align*}
where the equivalence relation $\sim$ is defined such that
\begin{align*}
  (\alpha,0,z) &\sim (2\pi-\alpha, 0, -z), &
  (\alpha,\pi,z) &\sim (2\pi-\alpha, \pi, -z), &
  (0,\beta,z) &\sim (2\pi,\beta,z).
\end{align*}
A specific homeomorphism $R_3(T^2,2) \rightarrow Y$ is given by
\begin{align*}
  (\alpha,\beta,s) \mapsto
  (\alpha,\beta, (\sin^2\alpha + \sin^2\beta)(\sin s)).
\end{align*}
We note that $Y$, hence $R_3(T^2,2)$, is compact with boundary
homeomorphic to $T^2$.
A specific homeomorphism $T^2 \rightarrow \partial R_3(T^2,2)$ is
given by $(e^{i\alpha},e^{i\beta}) \mapsto [A,B,a,b]$, where
\begin{align*}
  A &= \cos\alpha\,\id + \sin\alpha\,\qz, &
  B &= \cos\beta\,\id + \sin\beta\,\qz, &
  a &= b^{-1} = \qz.
\end{align*}
From equation (\ref{eqn:alpha-beta-s}), it follows that the reducible
and irreducible loci of $R_3(T^2,2)$ are its boundary and interior.

We would like to extend the standard coordinates $(\alpha,\beta,s)$ on
$R_3^*(T^2,2)$ to a system of coordinates defined on a neighborhood
of $R_3^*(T^2,2)$ in $R^*(T^2,2)$.
There are two natural extensions:

\begin{theorem}
\label{theorem:coords-1}
We have coordinates $(\alpha_1,\beta_1,s_1,t_1)$ defined on a
neighborhood of $R_3^*(T^2,2) \cap \{\sin\alpha \neq 0\}$:
\begin{align*}
  A &= \cos\alpha_1\,\mathbbm{1} + \sin\alpha_1\,\qz, \\
  B &=
  \cos t_1 \cos\beta_1\,\mathbbm{1} + \sin t_1\,\qx +
  \cos t_1 \sin\beta_1\,\qz, \\
  a &=
  \cos s_1 \cos \theta_1\,\qx +
  \cos s_1 \sin \theta_1\,\qy +
  \sin s_1\,\qz, \\
  b &= (ABA^{-1}B^{-1}a)^{-1},
\end{align*}
where
\begin{align*}
  \theta_1 &=
  \alpha_1 + \beta_1 - \arcsin(\cos\alpha_1 \tan s_1 \tan t_1).
\end{align*}
\end{theorem}

\begin{proof}
A calculation shows that the matrices $(A,B,a,b)$ satisfy the
constraint $[A,B]ab = \id$ and thus represent a point in
$R(T^2,2)$.
Comparing the above matrices with the definition of the standard
coordinates $(\alpha,\beta,s)$ on $R_3^*(T^2,2)$ in equation
(\ref{eqn:alpha-beta-s}), we see that if $t_1 = 0$ then $[A,B,a,b]$
lies in $R_3^*(T^2,2)$ and
\begin{align*}
  (\alpha,\beta,s) = (\alpha_1,\beta_1,s_1).
\end{align*}
It follows that
$\{\partial_{\alpha_1}, \partial_{\beta_1}, \partial_{s_1}\}$ are
linearly independent on $R_3^*(T^2,2)$.
Since irreducibility is an open condition, for sufficiently small
values of $|t_1|$ the point $(\alpha_1,\beta_1,s_1,t_1)$ lies in
$R^*(T^2,2)$.
Define a function $f^4:R(T^2,2) \rightarrow \Reals$:
\begin{align*}
  f^4 &= -\frac{1}{4}(\tr Aa + \tr Ab).
\end{align*}
A calculation shows that the first derivatives of $f^4$ are given by
\begin{align*}
  (\partial_{\alpha_1} f^4)(\alpha,\beta,s,0) &= 0, &
  (\partial_{\beta_1} f^4)(\alpha,\beta,s,0) &= 0, \\
  (\partial_{s_1} f^4)(\alpha,\beta,s,0) &= 0, &
  (\partial_{t_1} f^4)(\alpha,\beta,s,0) &= \sin^2\alpha \cos s.
\end{align*}
So
$\{\partial_{\alpha_1}, \partial_{\beta_1}, \partial_{s_1},
\partial_{t_1}\}$ are
linearly independent on $R_3^*(T^2,2) \cap \{\sin \alpha \neq 0\}$.
\end{proof}

\begin{theorem}
\label{theorem:coords-2}
We have coordinates $(\alpha_2,\beta_2,s_2,t_2)$ defined on a
neighborhood of $R_3^*(T^2,2) \cap \{\sin\beta \neq 0\}$:
\begin{align*}
  A &=
  \cos t_2 \cos\alpha_2\,\mathbbm{1} + \sin t_2\,\qx +
  \cos t_2 \sin\alpha_2\,\qz, \\
  B &= \cos\beta_2\,\mathbbm{1} + \sin\beta_2\,\qz, \\
  a &=
  \cos s_2 \cos \theta_2\,\qx +
  \cos s_2 \sin \theta_2\,\qy +
  \sin s_2\,\qz, \\
  b &= (ABA^{-1}B^{-1}a)^{-1},
\end{align*}
where
\begin{align*}
  \theta_2 &=
  \alpha_2 + \beta_2 - \arcsin(\cos\beta_2 \tan s_2 \tan t_2).
\end{align*}
\end{theorem}

\begin{proof}
The proof is similar to the proof of Theorem \ref{theorem:coords-1},
only we now define a function a function
$h^4:R(T^2,2) \rightarrow \Reals$:
\begin{align*}
  h^4 &= \frac{1}{4}(\tr Ba + \tr B b).
\end{align*}
A calculation shows that the first derivatives of $h^4$ are given by
\begin{align*}
  (\partial_{\alpha_2} h^4)(\alpha,\beta,s,0) &= 0, &
  (\partial_{\beta_2} h^4)(\alpha,\beta,s,0) &= 0, \\
  (\partial_{s_2} h^4)(\alpha,\beta,s,0) &= 0, &
  (\partial_{t_2} h^4)(\alpha,\beta,s,0) &= \sin^2\beta \cos s.
\end{align*}
So
$\{\partial_{\alpha_2}, \partial_{\beta_2}, \partial_{s_2},
\partial_{t_2}\}$ are
linearly independent on $R_3^*(T^2,2) \cap \{\sin \beta \neq 0\}$.
\end{proof}

For both systems, the coordinates $(\alpha_k,\beta_k,s_k,t_k)$ are
subject to $2\pi$-periodicity identifications together with the
identifications
\begin{align}
  \label{eqn:coords-ident}
  (\alpha_k,\beta_k,s_k,t_k)
  &\sim (-\alpha_k,-\beta_k,-s_k,t_k), &
  (\alpha_k,\beta_k,s_k,t_k)
  &\sim (\alpha_k,\beta_k,\pi-s_k,-t_k),
\end{align}
which follow from the fact that we can conjugate
$(A,B,a,b)$ by $\qx$ or $\qz$ to obtain an equivalent representative.
Conjugation by $\qy$ does not yield an independent identification,
since $\qx \qy \qz = -\id$.

\subsection{The character variety $R^*(S^2,4)$}
\label{ssec:pillowcase}

The Fukaya category of $R^*(T^2,2)$ appears to be closely related to
the Fukaya category of the pillowcase $R^*(S^2,4)$.
We review here some well-known results regarding $R^*(S^2,4)$ and
describe its relationship to $R^*(T^2,2)$.

The topology of $R(S^2,4)$ is discussed in \cite{Hedden-1}.
As a set, the character variety $R(S^2,4)$ consists of
conjugacy classes $SU(2)$ representations of the
fundamental group of $S^2 - \{p_1,p_2,p_3,p_4\}$ that map loops around
the punctures to traceless matrices.
We define loops $a$, $b$, $c$, $d$ around the punctures
$p_1$, $p_2$, $p_3$, $p_4$.
A presentation for the fundamental group of
$S^2 - \{p_1,p_2,p_3,p_4\}$ is
\begin{align*}
  \pi_1(S^2 - \{p_1,p_2,p_3,p_4\}) =
  \langle a, b, c, d \mid ba = cd \rangle.
\end{align*}
The character variety $R(S^4,4)$ is thus given by
\begin{align*}
  R(S^2,4) = \{(a,b,c,d) \in SU(2)^4 \mid ba = cd,\,
  \tr a = \tr b = \tr c = \tr d = 0\}/{\textup{conjugation}},
\end{align*}
where for simplicity we use the same notation for generators of
$\pi_1(S^2 - \{p_1,p_2,p_3,p_4\})$ and their image under a given
representation.
The character variety $R(S^2,4)$ and its irreducible locus
$R^*(S^2,4)$ can also be interpreted as moduli spaces of semistable
and stable parabolic bundles over $\CP^1$ with four marked points.
Using this interpretation, one can show that $R(S^2,4)$ and
$R^*(S^2,4)$ can be given the structure of complex manifolds
isomorphic to $\CP^1$ and the complement of four points in $\CP^1$,
respectively.
In particular, the irreducible locus $R^*(S^2,4)$ has the structure of
a smooth manifold diffeomorphic to $S^2$ minus four points.

We define coordinates on $R(S^2,4)$ as follows.
Given a point $[a,b,c,d]$ of $R(S^2,4)$, one can always choose a
representative of the form
\begin{align}
  \label{eqn:r-s2-4}
  a &= \qx, &
  b &= \cos\gamma\,\qx + \sin\gamma\,\qy, &
  c &= \cos\theta\,\qx + \sin\theta\,\qy, &
  d &= c^{-1}ba.
\end{align}
The parameters $(\gamma,\theta)$ are subject to $2\pi$-periodicity
identifications, as well as the identification
\begin{align*}
  (\gamma,\theta) \sim (-\gamma, -\theta),
\end{align*}
which follows from conjugation by $\qx$.
We will use $(\gamma,\theta)$ as standard coordinates on $R(S^2,4)$.
From the identifications on the coordinates $(\gamma,\theta)$, it
follows that $R(S^2,4)$ is homeomorphic to $S^2$.
From equation (\ref{eqn:r-s2-4}) it follows that the reducible locus
of $R(S^2,4)$ is the set $\{\sin\gamma = \sin\theta = 0\}$, which
consists of four points.

Since the identifications imposed on the coordinates $(\gamma,\theta)$
for $R(S^2,4)$
are the same as the identifications imposed on the imposed on the
coordinates $(\alpha,\beta)$ for $R_2(T^2,2)$, we can define a
homeomorphism $R(S^2,4) \rightarrow R_2(T^2,2)$,
$(\gamma,\theta) \mapsto (\alpha,\beta)$.
The homeomorphism restricts to a diffeomorphism
$R^*(S^2,4) \rightarrow R_2^*(T^2,2)$ of the irreducible loci.

\subsection{Symplectic form on $R^*(T^2,2)$}
\label{ssec:symplectic-form}

The irreducible locus $R^*(S,n)$ of the character variety $R(S,n)$ for
a surface $S$ with $n$ marked points is a smooth manifold with a
canonical symplectic form \cite{Goldman}.
The existence of the canonical symplectic form is due to the fact that
the character variety can be interpreted as the Hamiltonian reduction
of the space of connections on a trivial rank two complex vector
bundle over the surface with respect to the action of a group of gauge
transformations.
The space of connections has the structure of an infinite-dimensional
symplectic manifold, as was first noted by Atiyah and Bott
\cite{Atiyah-Bott}.
To compute the symplectic form on $R^*(T^2,2)$, we will use the
formalism of \emph{quasi-Hamiltonian reduction} described in
\cite{Alekseev}.
For our purposes, it will suffice to compute the restriction of the
symplectic form to the submanifold $R_3^*(T^2,2)$.

\begin{theorem}
\label{theorem:symplectic-form}
The restriction of the canonical symplectic form on $R^*(T^2,2)$ to
$R_3^*(T^2,2) \cap \{\sin \alpha \neq 0\}$ is
\begin{align*}
  \omega =
  d\alpha_1 \wedge d\beta_1 + \sin\alpha_1\,ds_1 \wedge dt_1.
\end{align*}
The restriction of the canonical symplectic form on $R^*(T^2,2)$ to
$R_3^*(T^2,2) \cap \{\sin \beta \neq 0\}$ is
\begin{align*}
  \omega =
  d\alpha_2 \wedge d\beta_2 - \sin\beta_2\,ds_2 \wedge dt_2.
\end{align*}
\end{theorem}

\begin{proof}
A general expression for the symplectic form on the character variety
$R^*(S,n)$ is given in Theorem 9.3 of \cite{Alekseev}.
We apply this formula to the case of $R^*(T^2,2)$.
For our application we use the Lie group $G = SU(2)$ and its Lie
algebra $\mathfrak{g} = \mathfrak{su}(2)$.
Define the Killing form $( - , - )$ on $\mathfrak{g}$ such that
\begin{align*}
  (x, y) = -\frac{1}{2}\tr xy.
\end{align*}
Define left-invariant and right-invariant Maurer-Cartan forms
$\theta,\bar{\theta} \in \Omega^1(G,\mathfrak{g})$.
At the point $g \in G$, we have
\begin{align*}
  \theta_g &= g^{-1}\,dg, &
  \thetabar_g &= dg\,g^{-1}.
\end{align*}
From Theorem 9.3 of \cite{Alekseev}, it follows that the canonical
symplectic form on $R^*(T^2,2)$ is given by
\begin{align*}
  \omega =
  \frac{1}{2}(A^*\theta,\, B^*\thetabar) +
  \frac{1}{2}(A^*\thetabar,\, B^*\theta) +
  \frac{1}{2}((AB)^*\theta,\, (A^{-1}B^{-1})^*\thetabar) +
  \frac{1}{2}(b^*\theta,\, (ABA^{-1}B^{-1})^* \thetabar).
\end{align*}
We substitute for the matrices $(A,B,a,b)$ using the coordinate
systems described in Theorems \ref{theorem:coords-1} and
\ref{theorem:coords-2} to obtain the result.
\end{proof}

A similar computation gives the canonical symplectic form on
$R^*(S^2,4)$ that is used in \cite{Hedden-3}:

\begin{theorem}
\label{theorem:r-s2-4}
The canonical symplectic form on $R^*(S^2,4)$ is
$d\gamma \wedge d\theta$.
\end{theorem}

In Section \ref{ssec:pillowcase} we defined a diffeomorphism
$R^*(S^2,4) \rightarrow R_2^*(T^2,2)$,
$(\gamma,\theta) \mapsto (\alpha,\beta)$.
Theorems \ref{theorem:symplectic-form} and \ref{theorem:r-s2-4} give:

\begin{theorem}
The diffeomorphism $R^*(S^2,4) \rightarrow R_2^*(T^2,2)$ is a
symplectomorphism.
\end{theorem}

\subsection{Hamiltonian perturbation}
\label{ssec:hamiltonian}

We will want to perturb Lagrangians in $R^*(T^2,2)$ by applying a
small Hamiltonian flow, which we refer to as a
\emph{Hamiltonian pushoff}.
We define a Hamiltonian function $H:R^*(T^2,2) \rightarrow \Reals$ by
\begin{align}
  \label{eqn:H}
  H = \frac{1}{2}(\tr A + \eta \tr B).
\end{align}
On the submanifold $R_3^*(T^2,2)$, we can express the Hamiltonian in
terms of the standard coordinates $(\alpha,\beta,s)$:
\begin{align*}
  H|_{R_3^*(T^2,2)} = \cos \alpha + \eta \cos \beta.
\end{align*}
We note that the Hamiltonian flow fixes the symplectic submanifold
$R_2^*(T^2,2)$ as a set.
From the expression for the Hamiltonian given in equation
(\ref{eqn:H}) and the symplectic form given in Theorem
\ref{theorem:symplectic-form}, it follows that the Hamiltonian flow
equations for the coordinates $(\alpha,\beta)$ on $R_2^*(T^2,2)$ are
\begin{align*}
  \dot{\alpha} &= \partial_\beta H =
  -\eta \sin\beta, &
  \dot{\beta} &= -\partial_\alpha H =
  \sin\alpha.
\end{align*}
We will choose $\eta = 0.2$ and evolve in time by $\tau = -0.2$ for
each Hamiltonian pushoff.

\subsection{The mapping class group}
\label{ssec:mcg}

Given an orientable surface $S$ with $n$ distinct marked points
$p_1, \cdots, p_n$, we define the \emph{mapping class group}
$\MCG_n(S)$ to be the group of isotopy classes of
orientation-preserving homeomorphisms of $S$ that fix
$\{p_1, \cdots, p_n\}$ as a set.
We define the \emph{pure mapping class group} $\PMCG_n(S)$ to be the
subgroup of $\MCG_n(S)$ that fixes the individual points
$p_1, \cdots, p_n$.

The mapping class group $\MCG_n(S)$ acts on the character variety
$R(S,n)$ and its irreducible locus $R^*(S,n)$.
The action is defined as follows.
Choose a basepoint $x_0 \in S^\circ = S - \{p_1, \cdots, p_n\}$.
We first define a homomorphism from $\MCG_n(S)$ to
$\Out(\pi_1(S^\circ,x_0))$,
the group of outer automorphisms of the fundamental group
$\pi_1(S^\circ,x_0)$.
Given an element $[\phi] \in \MCG_n(S)$ represented by a homeomorphism
$\phi:S^\circ \rightarrow S^\circ$, we have an induced
homomorphism
$\phi_*:\pi_1(S^\circ,x_0) \rightarrow \pi_1(S^\circ,\phi(x_0))$,
$[\alpha] \mapsto [\alpha \circ \phi]$.
We choose a path $\gamma:I \rightarrow S^\circ$ from
$x_0$ to $\phi(x_0)$ and define an induced isomorphism
$\pi_1(S^\circ,x_0) \rightarrow \pi_1(S^\circ,\phi(x_0))$,
$[\alpha] \mapsto [\gamma \alpha \gammabar]$.
We now define $f:\MCG_n(S) \rightarrow \Out(\pi_1(S^\circ,x_0))$,
$[\phi] \mapsto [\gamma_*\phi_*]$.
One can show that $f$ is independent of the choice of representative
$\phi$ and path $\gamma$, and is hence well-defined, and is a
homomorphism.
We define a left action of $\MCG_n(S)$ on the character variety
$R(S,n)$ by
\begin{align*}
  [\phi] \cdot [\rho] = [\rho \circ \tilde{f}([\phi])^{-1}],
\end{align*}
where $\tilde{f}([\phi])$ is an automorphism of $\pi_1(S^\circ,x_0)$
representing $f([\phi])$.
The action fixes the irreducible locus $R^*(S,n)$ of
$R(S,n)$ as a set, so we can restrict the action to $R^*(S,n)$.
One can show that the action of $\MCG_n(S)$ on $R^*(S,n)$ is
symplectic.

Here we describe the group $\MCG_2(T^2)$ and its action on
$R(T^2,2)$.
Presentations for mapping class groups are described in
\cite{Cattabriga,Gervais,Labruere}, and more details regarding the
case $(T^2,2)$ can be found in \cite{Boozer-2}.
For our purposes it will suffice to consider the pure mapping
class group $\PMCG_2(T^2)$.
The pure mapping class group is generated by Dehn twists about the
simple closed curves $a$, $A$, $b$, and $B$ shown in Figure
\ref{fig:mcg-braid}.
For simplicity, we will use the same notation for the curves and
the corresponding Dehn twists.
We denote the inverse Dehn twists by $\abar$, $\Abar$, $\bbar$,
$\Bbar$.
We have:

\begin{figure}
  \centering
  \includegraphics[scale=0.65]{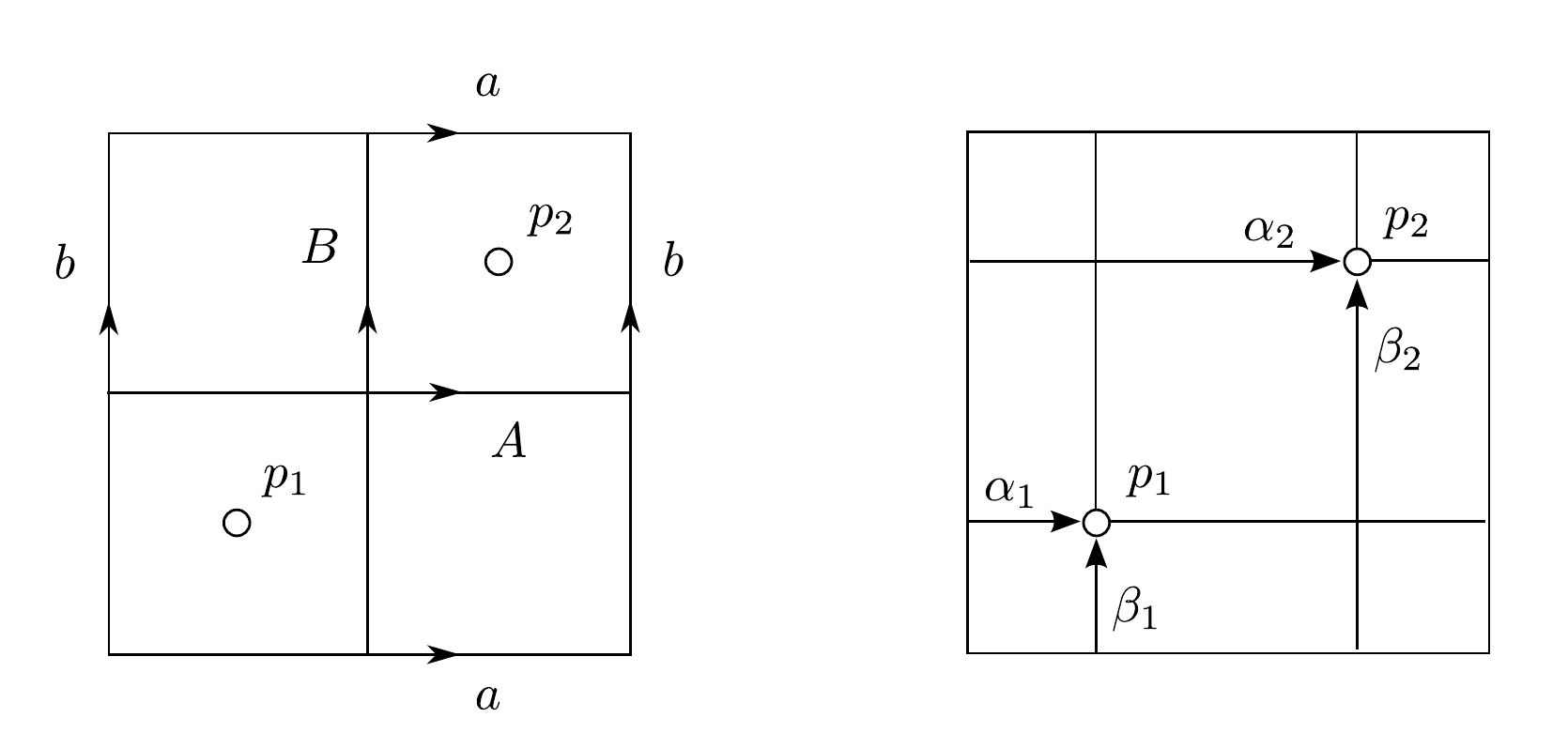}
  \caption{
    \label{fig:mcg-braid}
    Dehn twists $a$, $A$, $b$, and $B$ that generate
    $\PMCG_2(T^2)$ and
    braids $\alpha_1$, $\alpha_2$, $\beta_1$, and $\beta_2$ that
    generate $\PB_2(T^2)$.
  }
\end{figure}

\begin{theorem}
\label{theorem:pmcg-t2-2}
The pure mapping class group $\PMCG_2(T^2)$ is generated by the Dehn
twists $a$, $A$, $b$, and $B$, with relations
\begin{align*}
  &a \bbar a = \bbar a \bbar, &
  &A \bbar A = \bbar A \bbar, &
  &a A = A a, &
  & (\bbar a A)^4 = 1, &
  & \abar B a = \Abar b A.
\end{align*}
\end{theorem}

The relations in Theorem \ref{theorem:pmcg-t2-2} can be used to derive
the additional relation in $\PMCG_2(T^2)$:
\begin{align*}
  b B = B b.
\end{align*}
The fact that $b$ and $B$ commute is intuitively clear, since $b$ and
$B$ describe Dehn twists around nonintersecting cycles.

The pure mapping class group is closely related to the pure braid
group, which we define as follows.
Given a surface $S$, we define the
\emph{configuration space for $n$ ordered points}
\begin{align*}
  \Conf_n(S) &= \{(x_1, \cdots, x_n) \in S^n \mid
  \textup{$x_i \neq x_j$ if $i \neq j$}\}.
\end{align*}
Given a surface $S$ and $n$ distinct marked points
$p_1, \cdots, p_n$, we define the \emph{pure braid group}
$\PB_n(S)$ to be the fundamental group of $\Conf_n(S)$ with basepoint
$(p_1,\cdots,p_n)$.
Presentations for braid groups are described in \cite{Bellingeri}.
The pure braid group $\PB_2(T^2)$ is generated by braids
$\alpha_i$ and $\beta_i$ for $i=1, 2$ that drag the marked point $p_i$
rightward and upward around cycles parallel to $a$ and $b$, as shown
in Figure \ref{fig:mcg-braid}.

We have a group homomorphism
$\PB_2(S) \rightarrow \PMCG_2(S)$ called the \emph{push map}.
Intuitively, to obtain the image of a braid under the push map we
view $S$ as an elastic membrane and push the marked points along the
braid until they return to their starting locations.
The map from the initial to the final state of the membrane gives a
homeomorphism of $S$ that represents image of the braid under the
push map.
The push map $\PB_2(T^2) \rightarrow \PMCG_2(T^2)$ is given by
\begin{align*}
  &\alpha_1 \mapsto a\Abar, &
  &\beta_1 \mapsto b \Bbar, &
  &\alpha_2 \mapsto \abar A, &
  &\beta_2 \mapsto \bbar B.
\end{align*}
For simplicity, we will use the same notation for elements of
$\PB_2(T^2)$ and their images in $\PMCG_2(T^2)$ under the push map.

We define an element $s$ of $\PMCG_2(T^2)$ by
\begin{align*}
  s = \alpha_1 \abar b \abar = \Abar b \abar.
\end{align*}
Roughly speaking, if we view $T^2$ as a square with opposite edges
identified, the element $s$ rotates the square counterclockwise by
$\pi/2$.
We can use the relations in Theorem \ref{theorem:pmcg-t2-2} to derive
the following relations:
\begin{align*}
  s^4 &= 1, &
  s \alpha_1 s^{-1} &= \beta_1, &
  s \beta_1 s^{-1} &= \alpha_1^{-1}.
\end{align*}

The action of $\PMCG_2(T^2)$ on $R(T^2,2)$ is described explicitly in
\cite{Boozer-2}.
Using this description, one can show:

\begin{lemma}
The action of $\alpha_1^2$ fixes $R_2^*(T^2,2)$ as a set.
The action of $\alpha_1^2$ on the coordinates $(\alpha,\beta)$ is
given by
\begin{align*}
  (\alpha,\,\beta) \mapsto (\alpha,\, \pi + \beta - 2\alpha).
\end{align*}
\end{lemma}

\subsection{Lagrangians}

Our goal in this section is to define Lagrangians in $R^*(T^2,2)$
corresponding to solid tori containing trivial 1-tangles.
Given a trivial 1-tangle that winds $n$ times around the solid torus,
we will define a corresponding noncompact embedded Lagrangian $W_n$
and a compact immersed Lagrangian $L_n$.
In Section \ref{sec:complexes} we will use these Lagrangians to
construct objects in the twisted Fukaya category of $R^*(T^2,2)$
corresponding to planar 1-tangle diagrams in the annulus.

Consider a solid torus $U_1$ containing the trivial tangle $T_1$ shown
in Figure \ref{fig:torus-arc}(a).
We define the \emph{unperturbed character variety} $R(U_1,T_1)$ to be
the set of conjugacy classes of homomorphisms
$\pi_1(U_1 - T_1) \rightarrow SU(2)$ that map loops around $T_1$ to
traceless matrices.
We define $R^*(U_1,T_1)$ to be the irreducible locus of $R(U_1,T_1)$.
The set $R(U_1,T_1)$ has the structure of a smooth manifold with
boundary, as described by the following result:

\begin{figure}
  \centering
  \includegraphics[scale=0.65]{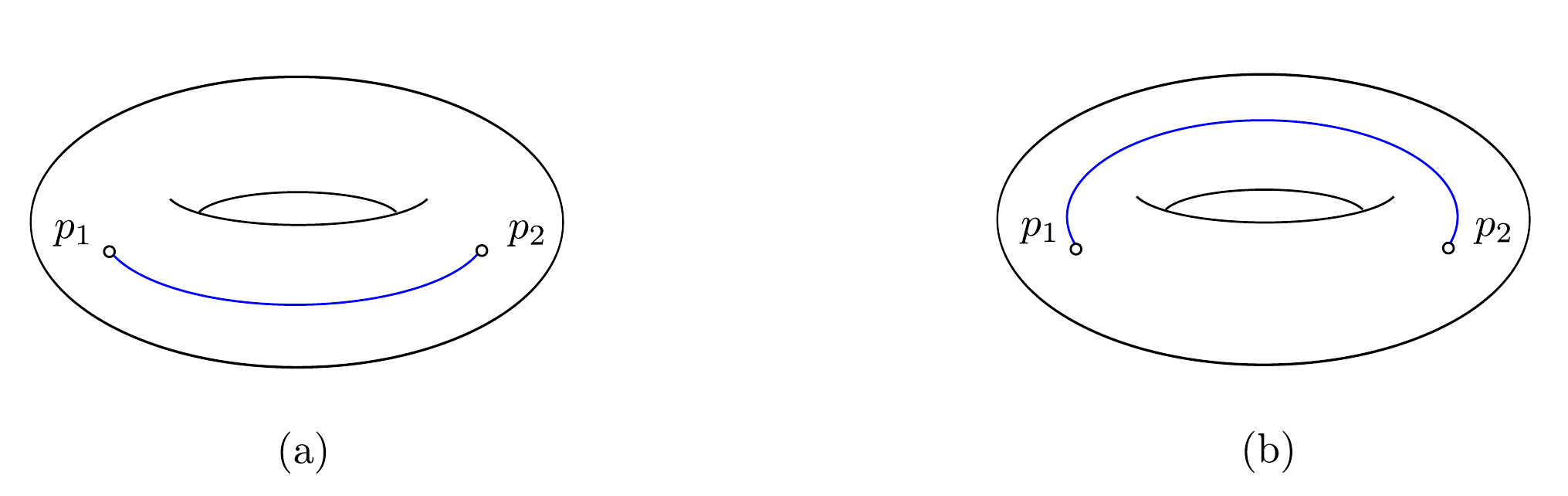}
  \caption{
    \label{fig:torus-arc}
    (a) Handlebody corresponding to the unperturbed Lagrangian $W_0$.
    (b) Handlebody corresponding to the unperturbed Lagrangian $W_1$.
  }
\end{figure}

\begin{theorem}(Boozer \cite[Theorem 3.11]{Boozer-2})
The unperturbed character variety $R(U_1,T_1)$ is diffeomorphic to the
closed unit disk $D^2$.
The reducible and irreducible loci of $R(U_1,T_1)$ are its boundary
and interior.
\end{theorem}

An explicit diffeomorphism
$R(U_1,T_1) \rightarrow D^2 =
\{(x,y) \in \Reals^2 \mid x^2 + y^2 \leq 1\}$ is given in
\cite{Boozer-2}.
We define coordinates
$(\chi,\psi) \in (0,\pi) \times (-\pi/2,\pi/2)$
on the irreducible locus $R^*(U_1,T_1)$ by
\begin{align*}
  \chi &=
  \pi - \cos^{-1} x \in (0,\pi), &
  \psi &=
  \sin^{-1}\left(\frac{y}{\sqrt{1-x^2}}\right) \in (-\pi/2,\pi/2).
\end{align*}

We define a map $R^*(U_1,T_1) \rightarrow R^*(T^2,2)$ by pulling back
representations along the inclusion
$T^2 - \{p_1, p_2\} \hookrightarrow U_1 - T_1$.
We define the \emph{unperturbed Lagrangian} $W_0$ to be the image of
this map.
We define $W_0(\chi,\psi)\ \in R^*(T^2,2)$ to be the image of the
point in $R^*(U_1,T_1)$ with coordinates $(\chi,\psi)$.

\begin{theorem}(Boozer \cite[Theorem 3.13]{Boozer-2})
\label{theorem:W0}
The map $R^*(U_1,T_1) \rightarrow R^*(T^2,2)$ is an injective
immersion.
The Lagrangian $W_0$ lies in $R_3^*(T^2,2)$ and is given by
\begin{align*}
  W_0 = \{B = \id\} = R_3^*(T^2,2) \cap \{\beta = 0\}.
\end{align*}
The point $W_0(\chi,\psi) = [A,B,a,b] \in R^*(T^2,2)$ is given by
\begin{align*}
  A &= \cos\chi + \sin\chi\,\qz, &
  B &= \id, &
  a &= b^{-1} = \cos\psi\,\qx + \sin\psi\,\qz.
\end{align*}
The point $W_0(\chi,\psi)$ lies in $R_3^*(T^2,2)$ and has coordinates
\begin{align*}
  (\alpha,\beta,s)(W_0(\chi,\psi)) = (\chi,0,\psi).
\end{align*}
\end{theorem}

The Lagrangian $W_0$ is an open disk, and hence noncompact.
We will often prefer to work with compact Lagrangians.
To obtain a compact Lagrangian, we first modify the unperturbed
character variety $R(U_1,T_1)$ by introducing additional structure.
We introduce an \emph{earring} and a \emph{holonomy perturbation} to
yield a \emph{perturbed character variety} $R_\pi^\natural(U_1,T_1)$.
These modifications are motivated by considerations from gauge theory
and are discussed in \cite{Boozer-2}; the precise details will not be
needed here.
The set $R_\pi^\natural(U_1,T_1)$ has the structure of a smooth
manifold described by the following result:

\begin{theorem}(Boozer \cite[Theorem 3.15]{Boozer-2})
The perturbed character variety $R_\pi^\natural(U_1,T_1)$ is
diffeomorphic to $S^2$.
All points of $R_\pi^\natural(U_1,T_1)$ are irreducible.
\end{theorem}

An explicit diffeomorphism
$R_\pi^\natural(U_1,T_1) \rightarrow S^2 =
\{(x,y,z) \in \Reals^3 \mid x^2+y^2+z^2 = 1\}$
is given in \cite{Boozer-2}.
We define spherical-polar coordinates
$(\phi,\theta) \in [0,\pi] \times [0,2\pi]$ on
$R_\pi^\natural(U_1,T_1)$ by
\begin{align*}
  x &= \sin\phi \cos\theta, &
  y &= \sin\phi \sin\theta, &
  z &= \cos\phi.
\end{align*}

We define a pullback map
$R_\pi^\natural(U_1,T_1) \rightarrow R^*(T^2,2)$ and define the
\emph{perturbed Lagrangian} $L_0$ to be its image.
We define $L_0(\phi,\theta) \in R^*(T^2,2)$ to be the image of the point
in $R_\pi^\natural(U_1,T_1)$ with coordinates $(\phi,\theta)$.
We define the \emph{double-point} $p_D \in R^*(T^2,2)$ to be the point
$[A,B,a,b] = [\qz,\id,\qx,-\qx]$.
The point $p_D$ lies in $R_3^*(T^2,2)$ and has
coordinates
\begin{align*}
  (\alpha,\beta,s)(p_D) = (\pi/2,0,0).
\end{align*}

\begin{theorem}(Boozer \cite[Theorem 3.17]{Boozer-2})
\label{theorem:L0}
The map $R_\pi^\natural(U_1,T_1) \rightarrow R^*(T^2,2)$ is an immersion
and is injective except at the north pole $(\phi=0)$ and south pole
$(\phi=\pi)$ of $R_\pi^\natural(U_1,T_1) = S^2$, which are
both mapped to the double-point $p_D$.
The point $L_0(\phi,\theta) = [A,B,a,b] \in R^*(T^2,2)$ is given by
\begin{align}
  \nonumber
  A &=
  (\cos^2\nu + \sin^2\nu\sin^2\theta)^{-1/2}
  (\cos \nu\,\qx + \sin\nu \sin\theta\,\qz)
  (\cos\phi + \sin\phi\,(\cos\theta\,\qx + \sin\theta\,\qy)), \\
  \nonumber
  B &=
  \cos \nu + \sin\nu\,(\cos\theta\,\qx + \sin\theta\,\qy), \\
  \nonumber
  a &= \qz, \\
  \nonumber
  b &=
  (\cos^2\nu + \sin^2\nu\sin^2\theta)^{-1}
  (\sin2\nu \sin\theta\,\qx - (\cos^2\nu - \sin^2\nu\sin^2\theta)\qz),
\end{align}
where
\begin{align*}
  \nu = \epsilon \sin \phi
\end{align*}
and $\epsilon > 0$ is a small control parameter that determines the
strength of the holonomy perturbation.
The point $L_0(\phi,\theta)$ lies in $R_3^*(T^2,2)$ if and only
if $\theta \in \{0, \pi\}$, and the coordinates of such points are
\begin{align*}
  (\alpha, \beta, s)(L_0(\phi,0)) &=
  (\phi + \pi/2, \nu, 0).
\end{align*}
\end{theorem}

In the limit  $\epsilon \rightarrow 0$ in which the
holonomy perturbation goes to zero, we can relate the perturbed
Lagrangian $L_0$ to the unperturbed Lagrangian $W_0$.
Recall that we have explicit diffeomorphisms
$R(U_1,T_1) \rightarrow D^2$ and
$R_\pi^\natural(U_1,T_1) \rightarrow S^2$.
If we view $D^2$ as the equatorial disk of $S^2$, we can define a map
$R_\pi^\natural(U_1,T_1) \rightarrow R(U_1,T_1)$ by orthogonally
projecting $S^2$ to $D^2$ along the $z$-axis.
Using the results of Theorems \ref{theorem:W0} and \ref{theorem:L0}, a
straightforward calculation gives the following result:

\begin{theorem}
\label{theorem:map-L0-W0}
In the limit $\epsilon \rightarrow 0$ in which the holonomy
perturbation goes to zero, we have a commutative diagram
\begin{eqnarray*}
\begin{tikzcd}
  R_\pi^\natural(U_1,T_1) \arrow{d} \arrow{dr} \\
  R(U_1,T_1) \arrow{r} &
  L_0 = W_0.
\end{tikzcd}
\end{eqnarray*}
\end{theorem}

We define an unperturbed Lagrangian $W_n$ and a perturbed Lagrangian
$L_n$ by acting on $W_0$ and $L_0$ with the braid group element
$\alpha_1^n$:
\begin{align*}
  W_n &= \alpha_1^n \cdot W_0, &
  L_n &= \alpha_1^n \cdot L_0.
\end{align*}
The Lagrangians $W_n$ and $L_n$ correspond to a trivial 1-tangle that
winds $n$ times around the solid torus.
We define maps
$R^*(U_1,T_1) \rightarrow W_n$ and
$R_\pi^\natural(U_1,T_1) \rightarrow L_n$ by post-composing
$R^*(U_1,T_1) \rightarrow W_0$ and
$R_\pi^\natural(U_1,T_1) \rightarrow L_0$ with $\alpha_1^n$.

\begin{theorem}
\label{theorem:W2}
The point $W_2(\chi,\psi)$ lies in $R_2^*(T^2,2)$ if and only if
$\psi=0$.
The coordinates of such points are
\begin{align*}
  (\alpha,\beta)(W_2(\chi,0)) &=
  (\chi, \pi - 2\chi).
\end{align*}
\end{theorem}

\begin{proof}
This follows from Theorem \ref{theorem:W0} for
$W_0$ and the fact that $\alpha_1^2$ fixes $R_2^*(T^2,2)$ as a set.
\end{proof}

\begin{theorem}
\label{theorem:L2}
The point $L_2(\phi,\theta)$ lies in $R_2^*(T^2,2)$ if and only if
$\theta \in \{0,\pi\}$.
The coordinates of such points are
\begin{align*}
  (\alpha,\beta)(L_2(\phi,0)) &=
  (\phi + \pi/2, \nu - 2\phi).
\end{align*}
\end{theorem}

\begin{proof}
This follows from Theorem \ref{theorem:L0} for
$L_0$ and the fact that $\alpha_1^2$ fixes $R_2^*(T^2,2)$ as a set.
\end{proof}

We depict the intersection of the Lagrangians $W_0$, $W_2$, $L_0$, and
$L_2$ with $R_2^*(T^2,2)$ in Figure \ref{fig:lagrangians-rectangle},
where $R_2^*(T^2,2)$ is depicted as a rectangle in the $\alpha\beta$
plane with edges identified, and in Figure \ref{fig:lagrangians},
where $R_2^*(T^2,2)$ is depicted as a plane with four punctures,
together with the point at infinity.

\begin{figure}
  \centering
  \includegraphics[scale=0.65]{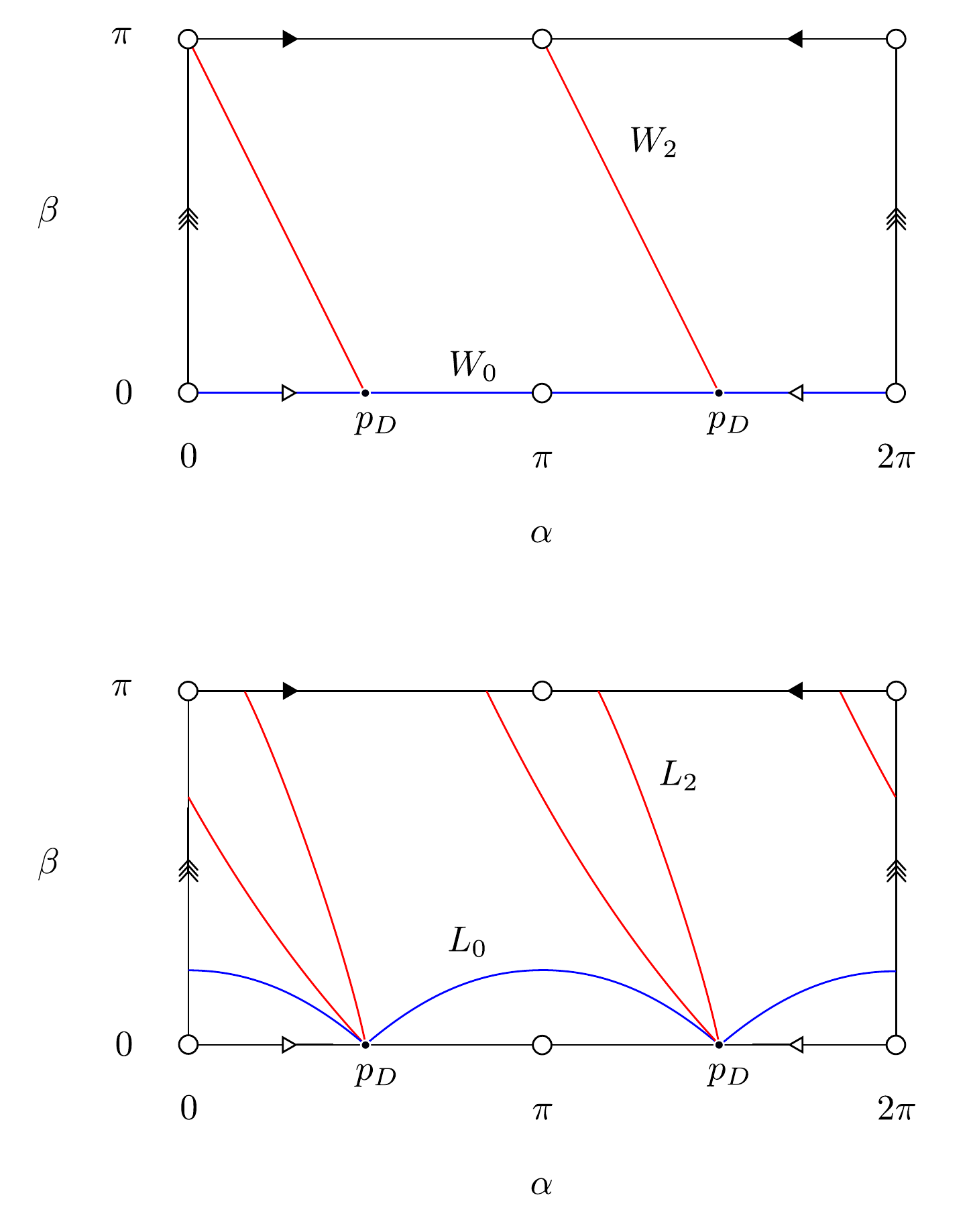}
  \caption{
    \label{fig:lagrangians-rectangle}
    Unperturbed Lagrangians $W_0$ and $W_2$, perturbed Lagrangians
    $L_0$ and $L_2$, and double-point $p_D$ in $R_2^*(T^2,2)$.
    The edges of the rectangle in the $\alpha\beta$ plane are
    identified as indicated to give $R_2^*(T^2,2)$.
    The circles indicate the four puncture points of
    $R_2^*(T^2,2)$.
  }
\end{figure}

\begin{figure}
  \centering
  \includegraphics[scale=0.65]{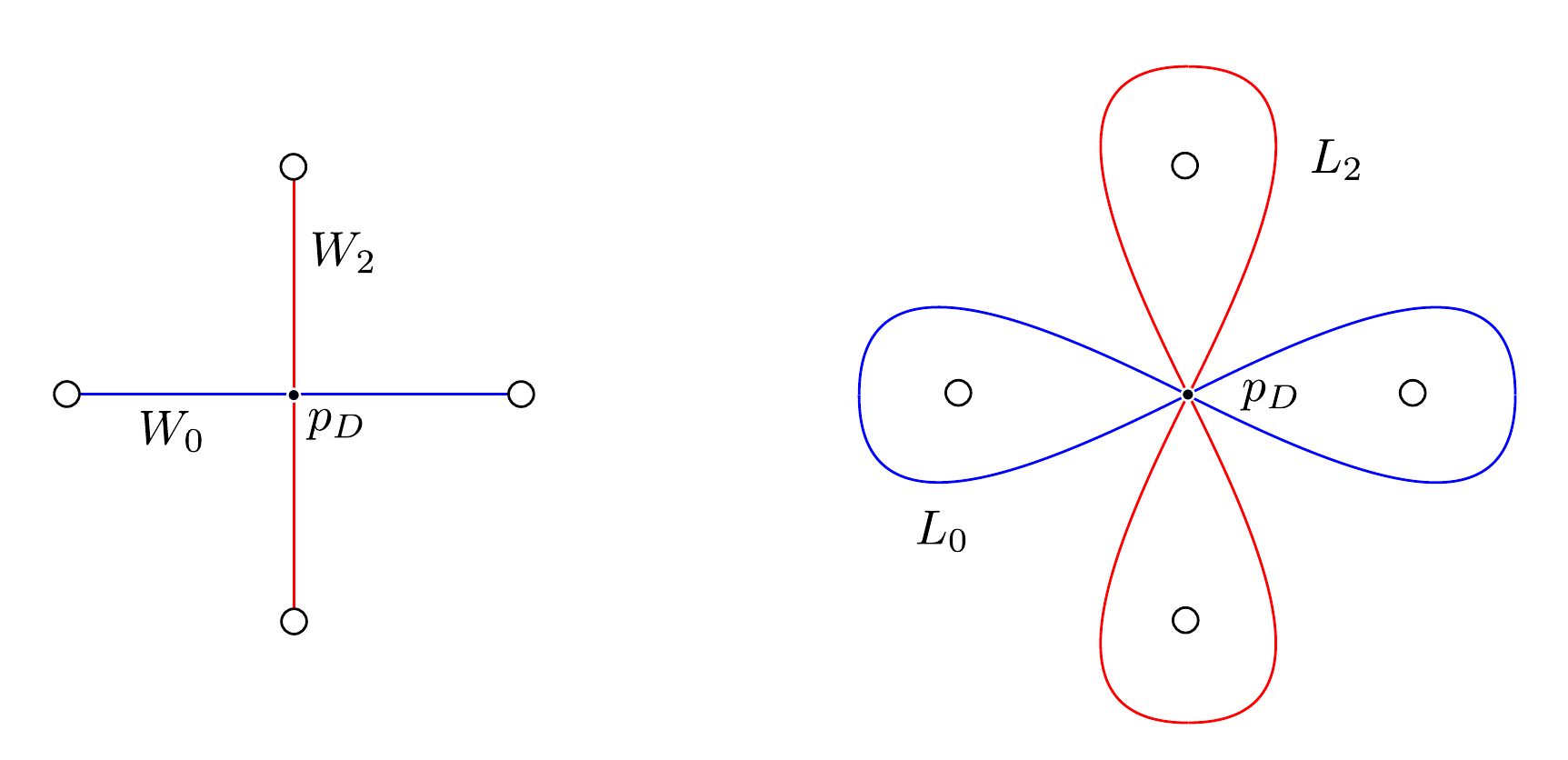}
  \caption{
    \label{fig:lagrangians}
    Unperturbed Lagrangians $W_0$ and $W_2$, perturbed Lagrangians
    $L_0$ and $L_2$, and double-point $p_D$ in $R_2^*(T^2,2)$.
    The four circles indicate the four puncture points of
    $R_2^*(T^2,2)$.
  }
\end{figure}

\subsection{The Fukaya category $\calL$}
\label{ssec:fuaka}

We define a subcategory $\calL$ of the Fukaya category of
$R^*(T^2,2)$ whose objects are the perturbed Lagrangians
$L_n$ in $R^*(T^2,2)$.
To complete the description of $\calL$, we need to describe the
morphism spaces and the $A_\infty$ operations.

We first consider the morphism spaces.
Given a Lagrangian $M$ of $R^*(T^2,2)$, let
$M^{(1)}, M^{(2)}, M^{(3)}, \cdots$ denote successive Hamiltonian
pushoffs of $M$ using the Hamiltonian described in Section
\ref{ssec:hamiltonian}.
Given transverse Lagrangians $M$ and $N$, we define $(M,N)$ to be the
$\F$-vector space freely generated by the intersection points of
$M^{(1)}$ and $N^{(0)}$:
\begin{align*}
  (M,N) = \langle M^{(1)} \cap N^{(0)} \rangle.
\end{align*}
The space of morphisms from $L_n$ to $L_m$ is given by
\begin{align*}
  \hom_{\calL}(L_n, L_m) = (L_n, L_m) =
  \langle L_n^{(1)} \cap L_m^{(0)} \rangle.
\end{align*}

Next we consider the $A_\infty$ operations.
Given $m+1$ immersed Lagrangian submanifolds $M_0, \cdots, M_m$, we
define
\begin{align*}
  \mu_\calL^m: (M_{m-1},M_m) \otimes (M_{m-2}, M_{m-1}) \otimes \cdots
  \otimes (M_0,M_1) \rightarrow (M_0,M_m)
\end{align*}
as follows.
Given generators $p_k \in M_{k-1} \cap M_k$ for $k=1, \cdots, m$, we
define
\begin{align*}
  \mu_\calL^m(p_m, \cdots, p_1) = q,
\end{align*}
where the coefficient of the point $q \in M_0^{(m)} \cap M_m^{(0)}$ is
given by a count of pseudo-holomorphic disks in $R^*(T,2)$ with
boundary on $M_0^{(m)} \cup M_1^{(m-1)} \cup \cdots \cup M_m^{(0)}$
and corners at $p_1, \cdots, p_m, q$.

We also need to define integer gradings on the morphism spaces.
We will return to this issue in Section \ref{ssec:ops-fukaya}; for
now, we will settle for defining a $\Ints_2$ grading, known as an
\emph{orientation grading}, as described in \cite{Auroux} Section 1.3.
Given a pair of oriented Lagrangians $M$ and $N$, consider a generator
$p$ of $(M,N)$, which is a point in $M^{(1)} \cap N^{(0)}$.
The tangent spaces $T_pM^{(1)}$ and $T_pN^{(0)}$ define Lagrangian
2-planes in the tangent space $T_p R^*(T^2,2)$.
One can define the notion of a \emph{canonical short path} from
$T_pM^{(1)}$ to $T_pN^{(0)}$ in the space of Lagrangian 2-planes of
$T_p R^*(T^2,2)$.
We assign the point $p$ grading plus or minus depending on whether the
canonical short path from $T_pM^{(1)}$ to $T_pN^{(0)}$ does or does
not map the given orientation of $T_pM^{(1)}$ to $T_pN^{(0)}$.
We will use the notation $p^{(+)}$ or $p^{(-)}$ to indicate that $p$
has orientation grading plus or minus.
The orientation grading of $\mu_{\calL}^m$ is plus for $m$ even and
minus for $m$ odd.

\subsection{Lagrangian intersections}

Our goal in this section is to describe the morphism spaces
$(M,N)$ for various pairs of Lagrangians $M$ and $N$ in $R^*(T^2,2)$.
We will focus our attention on the Lagrangians $W_0$, $L_0$, $W_2$,
and $L_2$.
For these Lagrangians, the intersection points that generate $(M,N)$
all lie in the two-dimensional submanifold $R_2^*(T^2,2)$ of
$R^*(T^2,2)$.
We can thus find the generators of $(M,N)$ by plotting the Lagrangians
$M^{(1)}$ and $N^{(0)}$ in $R_2^*(T^2,2)$ and identifying the
intersection points of the resulting curves.

To check that pairs of Lagrangians intersect transversely and to
compute the orientation gradings of the intersection points, we will
want to calculate tangent vectors to the Lagrangians.
We will express the tangent vectors in terms the following basis of
vector fields on $R_3^*(T^2,2) \cap \{\sin\alpha \neq 0\}$:
\begin{align*}
  v_1 &= \sin\alpha_1\,\partial_{\alpha_1}, &
  v_2 &= \sin\alpha_1\,\partial_{\beta_1}, &
  v_3 &= \sin\alpha_1\cos s_1\,\partial_{s_1}, &
  v_4 &= \cos s_1\,\partial_{t_1}.
\end{align*}
These vector fields are invariant under the coordinate identifications
given in equation (\ref{eqn:coords-ident}) and are thus well-defined.
A computation using the symplectic form given in Theorem
\ref{theorem:symplectic-form} shows that these vector fields do in
fact form a basis.

The Lagrangians $W_0$, $L_0$, $W_2$, and $L_2$ all contain the
double-point $p_D$, and for our purposes it will suffice to compute
tangent vectors at this point.
For $n \in \{0,2\}$, the point $p_D$ has a unique preimage under
$R^*(U_1,T_1) \rightarrow W_n$ whose coordinates are
$(\chi,\psi) = (\pi/2,0)$.
For $n \in \{0,2\}$, the point $p_D$ has two preimages under
$R_\pi^\natural(U_1,T_1) \rightarrow L_n$, namely the north pole
$(\phi = 0)$ and the south pole $(\phi = \pi)$.
We will orient $W_n$ so that $(\partial_\chi,\partial_\psi)$ is a
positively oriented basis.
We will orient $L_n$ so that $(\partial_\phi, \partial_\theta)$ is a
positively oriented basis.
For points along the great circle $y=0$ of
$R_\pi^\natural(U_1,T_1) = S^2$, it follows that
$(\partial_\phi, \partial_y)$ is an oriented basis, where
\begin{align*}
  \partial_\phi &= \cos\phi\,\partial_x - \sin\phi\,\partial_z, &
  \partial_y &= \csc\phi\,\partial_\theta.
\end{align*}

\begin{lemma}
\label{lemma:tangent-vectors}
The tangent vectors to $W_0$ at the point
$p_D$ are
\begin{align*}
  \partial_\chi &= v_1, &
  \partial_\psi &= v_3.
\end{align*}
The tangent vectors to $W_2$ at the point
$p_D$ are
\begin{align*}
  \partial_\chi &= v_1 - 2v_2, &
  \partial_\psi &= v_3 + 2v_4.
\end{align*}
The tangent vectors to $L_0$ at the point $p_D$ corresponding to
the north pole $(p_D = L_0(0,0)) $ and south pole
$(p_D = L_0(\pi,0))$ are
\begin{align*}
  &L_0(0,0): &
  &\partial_\phi =
  v_1 + \epsilon v_2, &
  &\partial_y =
  (1 + \epsilon) v_3 + \epsilon v_4, \\
  &L_0(\pi,0): &
  &\partial_\phi =
  -v_1 + \epsilon v_2, &
  &\partial_y =
  (1 - \epsilon) v_3 - \epsilon v_4.
\end{align*}
The tangent vectors to $L_2$ at the point $p_D$ corresponding to
the north pole $(p_D = L_2(0,0)) $ and south pole
$(p_D = L_2(\pi,0))$ are
\begin{align*}
  &L_2(0,0): &
  &\partial_\phi =
  v_1 - (2 - \epsilon) v_2, &
  &\partial_y =
  (1 + \epsilon) v_3 + (2 + 3\epsilon) v_4, \\
  &L_2(\pi,0): &
  &\partial_\phi =
  -v_1 + (2 + \epsilon) v_2, &
  &\partial_y =
  (1 - \epsilon) v_3 + (2 - 3\epsilon) v_4.
\end{align*}
\end{lemma}

\begin{proof}
We define functions $f^\mu:R(T^2,2) \rightarrow \Reals$ by
\begin{align*}
  f^1 &= -\frac{1}{2}(\tr A), &
  f^2 &= -\frac{1}{2}(\tr AB - \tr A), &
  f^3 &= -\frac{1}{2}(\tr Aa), &
  f^4 &= -\frac{1}{4}(\tr Aa + \tr Ab).
\end{align*}
We define derivatives $\partial_\mu$:
\begin{align*}
  \partial_1 &= \partial_{\alpha_1}, &
  \partial_2 &= \partial_{\beta_1}, &
  \partial_3 &= \partial_{s_1}, &
  \partial_4 &= \partial_{t_1}.
\end{align*}
A calculation shows that the derivatives of $f^\mu$ at the point $p_D$
are given by
\begin{align*}
  (\partial_\nu f^\mu)(p_D) = \delta_\mu^\nu,
\end{align*}
so $f^\mu$ define coordinates on a neighborhood of $p_D$.
The tangent vector to a curve $\gamma:\Reals \rightarrow R^*(T^2,2)$,
$t \mapsto \gamma(t)$ passing through the point
$p_D$ at $t=0$ is thus given by
\begin{align*}
  (f^1 \circ \gamma)'(0)\,v_1 +
  (f^2 \circ \gamma)'(0)\,v_2 +
  (f^3 \circ \gamma)'(0)\,v_3 +
  (f^4 \circ \gamma)'(0)\,v_4.
\end{align*}
The expressions for the tangent vectors are obtained via
straightforward calculations using Theorems \ref{theorem:W0} and
\ref{theorem:L0} for $W_0$ and $L_0$ and the action of $\alpha_1$.
\end{proof}

\begin{theorem}
\label{theorem:w0-l0-l2}
We have
\begin{align*}
  (W_0, L_0) &= \langle \alpha_0^{(-)},\, \beta_0^{(+)}\rangle, &
  (W_0, L_2) &= \langle \sigma_{2,0}^{(-)},\, \tau_{2,0}^{(+)}\rangle, &
  (W_2, L_0) &= \langle \sigma_{0,2}^{(-)},\, \tau_{0,2}^{(+)}\rangle,
\end{align*}
where the points
$\alpha_0, \beta_0, \sigma_{2,0}, \tau_{2,0}, \sigma_{0,2}, \tau_{0,2}
\in R_2^*(T^2,2)$ are as shown in Figure \ref{fig:w0-l0-w0-l2}.
\end{theorem}

\begin{figure}
  \centering
  \includegraphics[scale=0.65]{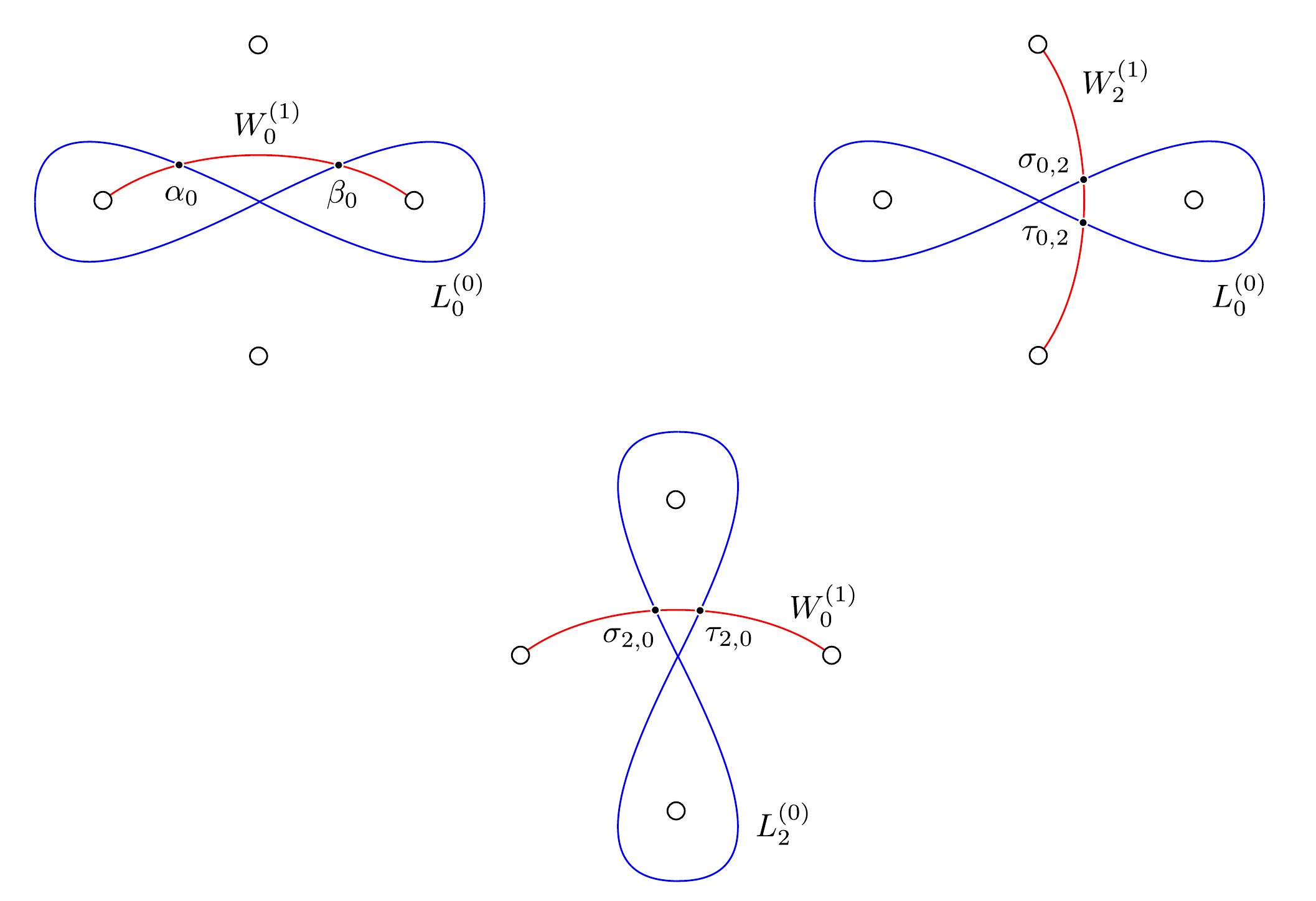}
  \caption{
    \label{fig:w0-l0-w0-l2}
    Generators $\alpha_0$ and $\beta_0$ of $(W_0, L_0)$,
    generators $\sigma_{0,2}$ and $\tau_{0,2}$ of $(W_2, L_0)$, and
    generators $\sigma_{2,0}$ and $\tau_{2,0}$ of $(W_0, L_2)$ in
    $R_2^*(T^2,2)$.
    The four circles indicate the four puncture points of
    $R_2^*(T^2,2)$.
  }
\end{figure}

\begin{proof}
First consider $(W_0,L_0)$.
To show that $\alpha_0$ and $\beta_0$ are the only points at which
$W_0^{(1)}$ and $L_0^{(0)}$ intersect, we consider the limit in which
the Hamiltonian perturbation goes to zero.
In this limit we have $W_0^{(1)} = W_0$ and $L_0^{(0)} = L_0$.
Recall from Theorem \ref{theorem:W0} that
$W_0 = R_3^*(U_1,T_1) \cap \{\beta = 0\}$.
From Theorem \ref{theorem:L0} for $L_0$ it follows that $W_0$ and
$L_0$ intersect in the unique point $p_D$.
Theorems \ref{theorem:W0} and \ref{theorem:L0} show that
$p_D$ has a unique preimage point under
$R(U_1,T_1) \rightarrow W_0$ and two preimage points under
$R_\pi^\natural(U_1,T_1) \rightarrow L_0$:
\begin{align*}
  p_D = W_0(\pi/2,0) = L_0(0,0) = L_0(\pi,0).
\end{align*}
So in the limit in which the Hamiltonian perturbation goes to zero we
have a 2-fold intersection
$W_0^{(1)} \cap L_0^{(0)} = W_0 \cap L_0 = \{p_D\}$.
In this limit the intersection of $W_0^{(1)}$ and $L_0^{(0)}$ is
transversal, as can be verified using the tangent vectors given in
Lemma \ref{lemma:tangent-vectors}.
If we turn on the Hamiltonian perturbation, the point $p_D$ splits
into two points $\alpha_0$ and $\beta_0$, and the
transversality of the intersection shows that for a sufficiently small
perturbation there are no additional intersection points.
The argument for $(W_0,L_2)$ and $(W_2,L_0)$ is similar.

To compute the orientation gradings, we note that
\begin{align*}
  &\alpha_0 = W_0^{(1)}(\pi/2,0) = L_0^{(0)}(\pi,0) = p_D, &
  &\beta_0 = W_0^{(1)}(\pi/2,0) = L_0^{(0)}(0,0) = p_D, \\
  &\sigma_{2,0} = W_0^{(1)}(\pi/2,0) = L_2^{(0)}(0,0) = p_D, &
  &\tau_{2,0} = W_0^{(1)}(\pi/2,0) = L_2^{(0)}(\pi,0) = p_D, \\
  &\sigma_{0,2} = W_2^{(1)}(\pi/2,0) = L_0^{(0)}(0,0) = p_D, &
  &\tau_{0,2} = W_2^{(1)}(\pi/2,0) = L_0^{(0)}(\pi,0) = p_D.
\end{align*}
The orientation gradings now follow from a computation using the
tangent vectors given in Lemma \ref{lemma:tangent-vectors}.
\end{proof}

\begin{theorem}
\label{theorem:gens-L0-L0}
We have
\begin{align*}
  (L_0, L_0) &=
  \langle a_0^{(+)},\, b_0^{(-)},\, c_0^{(-)},\, d_0^{(+)} \rangle,
\end{align*}
where the points $a_0, b_0, c_0, d_0 \in R_2^*(T^2,2)$ are as shown in
Figure \ref{fig:l0-l0}.
\end{theorem}

\begin{figure}
  \centering
  \includegraphics[scale=0.65]{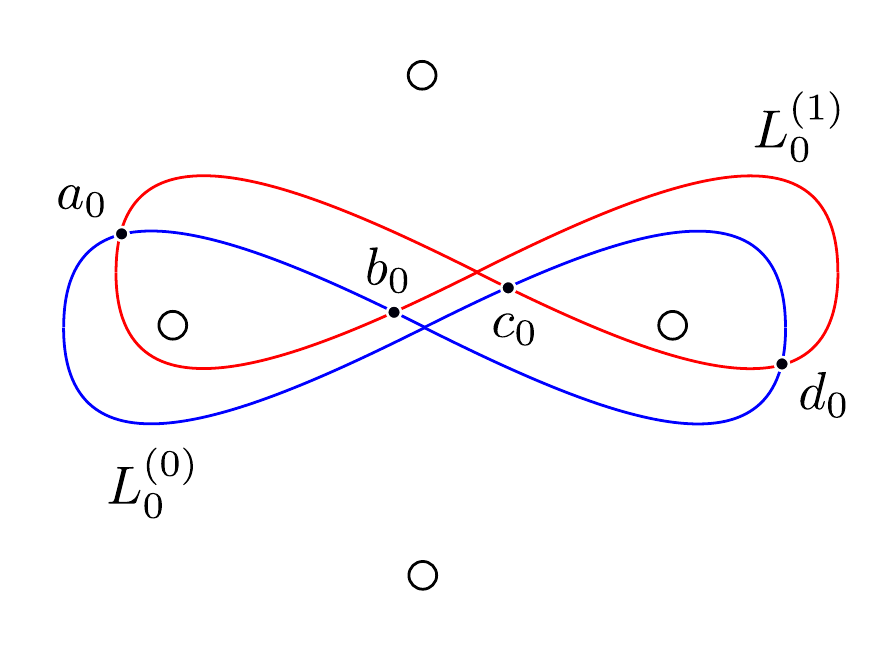}
  \caption{
    \label{fig:l0-l0}
    Generators $a_0$, $b_0$, $c_0$, and $d_0$ of $(W_0, L_0)$ in
    $R_2^*(T^2,2)$.
    The four circles indicate the four puncture points of
    $R_2^*(T^2,2)$.
  }
\end{figure}

\begin{proof}
To show that $a_0$, $b_0$, $c_0$, and $d_0$ are the only points at
which $L_0^{(1)}$ and $L_0^{(0)}$ intersect, we first take the limit
in which the Hamiltonian perturbation goes to zero.
In this limit $L_0^{(1)} = L_0^{(0)} = L_0$.
We next take the limit in which the holonomy perturbation of
$L_0^{(1)}$ goes to zero but the holonomy perturbation of $L_0^{(0)}$
remains nonzero.
In this limit $L_0^{(1)} = W_0$ and $L_0^{(0)} = L_0$.
Recall from Theorem \ref{theorem:W0} that
$W_0 = R_3^*(U_1,T_1) \cap \{\beta = 0\}$.
From Theorem \ref{theorem:L0} it follows that
$W_0$ and $L_0$ intersect in the unique point $p_D$.
From Theorems \ref{theorem:W0} and \ref{theorem:L0} it follows that
$p_D$ has two preimage points under
$R_\pi^\natural(U_1,T_1) \rightarrow R(U_1,T_1) \rightarrow W_0$:
\begin{align*}
  p_D &= W_0(\pi/2,0) = L_0^{(1)}(0,0) = L_0^{(1)}(\pi,0)
\end{align*}
and two preimage points under
$R_\pi^\natural(U_1,T_1) \rightarrow L_0$:
\begin{align*}
  p_D &= L_0^{(0)}(0,0) = L_0^{(0)}(\pi,0).
\end{align*}
So in the combined limit in which the Hamiltonian perturbation goes to
zero and the holonomy perturbation of $L_0^{(1)}$ goes to zero, we
have a four-fold intersection
$L_0^{(1)} \cap L_0^{(0)} = W_0 \cap L_0 = \{p_D\}$.
In the combined limit the intersection of $L_0^{(1)}$ and $L_0^{(0)}$
is transversal, as can be verified using the tangent vectors given
in Lemma \ref{lemma:tangent-vectors}.
If we now turn on the Hamiltonian and holonomy perturbations, the
point $p_D$ splits into the four points $a_0$, $b_0$, $c_0$, $d_0$
indicated in Figure \ref{fig:l0-l0}, and there are no additional
intersection points.

To compute the orientation gradings, we note that in the combined
limit we have
\begin{align*}
  b_0 &= L_0^{(1)}(0,0) = L_0^{(0)}(\pi,0) = p_D, &
  c_0 &= L_0^{(1)}(\pi,0) = L_0^{(0)}(0,0) = p_D.
\end{align*}
The orientation gradings now follow from a computation using the
tangent vectors given in Lemma \ref{lemma:tangent-vectors}.
The orientation gradings of $a_0$ and $d_0$ follow from the fact that
$L_0$ has self-intersection number 0, as can be seen by noting that
$L_0$ is homotopic to the closure of $W_0$ in $R(T^2,2)$, which is a
closed disk, and hence contractible in $R(T^2,2)$
\end{proof}

\begin{theorem}
\label{theorem:gens-L2-L0}
We have
\begin{align*}
  (L_0, L_2) &=
  \langle
  r_{2,0}^{(-)},\,\rbar_{2,0}^{(-)},\,s_{2,0}^{(+)},\,\sbar_{2,0}^{(+)}
  \rangle,
  &
  (L_2, L_0) &=
  \langle
  r_{0,2}^{(-)},\,\rbar_{0,2}^{(-)},\,s_{0,2}^{(+)},\,\sbar_{0,2}^{(+)}
  \rangle,
\end{align*}
where
$r_{2,0}, \rbar_{2,0}, s_{2,0}, \sbar_{2,0}, r_{0,2}, \rbar_{0,2}, s_{0,2},
\sbar_{0,2} \in R_2^*(T^2,2)$ are as
shown in Figure \ref{fig:l0-l2-l2-l0}.
\end{theorem}

\begin{figure}
  \centering
  \includegraphics[scale=0.65]{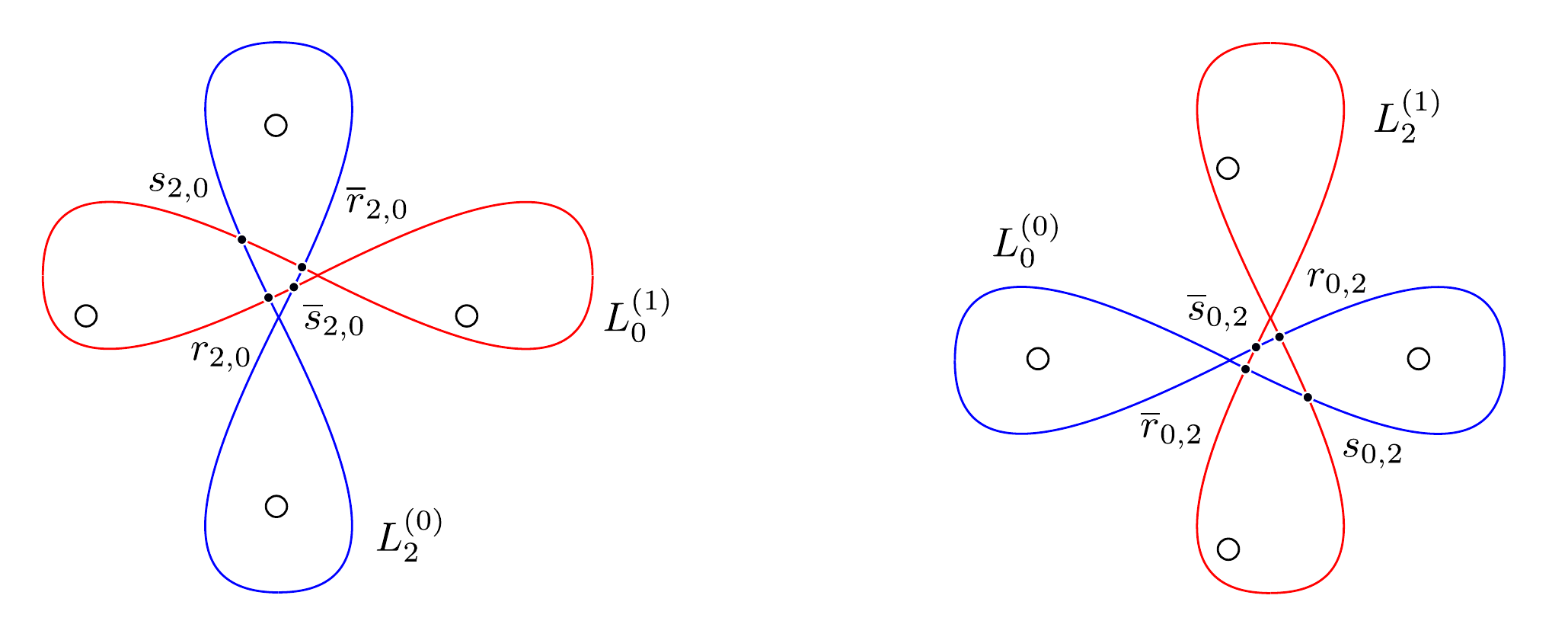}
  \caption{
    \label{fig:l0-l2-l2-l0}
    Generators $r_{2,0}$, $\rbar_{2,0}$, $s_{2,0}$, and $\sbar_{2,0}$
    of $(L_0,L_2)$ and
    generators $r_{0,2}$, $\rbar_{0,2}$, $s_{0,2}$, and $\sbar_{0,2}$
    of $(L_2,L_0)$ in $R_2^*(T^2,2)$.
    The four circles indicate the four puncture points of
    $R_2^*(T^2,2)$.
  }
\end{figure}

\begin{proof}
Similar arguments as those used in Theorems \ref{theorem:w0-l0-l2} and
\ref{theorem:gens-L0-L0} show that
$L_2^{(1)}$ and $L_0^{(0)}$ intersect transversely in the four points
$r_{0,2}$, $\rbar_{0,2}$, $s_{0,2}$, $\sbar_{0,2}$ shown in
Figure \ref{fig:l0-l2-l2-l0}.
In the limit that the Hamiltonian perturbation goes to zero, each of
these points approaches $p_D$ and we have
\begin{align*}
  & r_{0,2} =
  L_2^{(1)}(0,0) = L_0^{(0)}(0,0) = p_D, &
  & \rbar_{0,2} =
  L_2^{(1)}(\pi,0) = L_0^{(0)}(\pi,0) = p_D, \\
  & s_{0,2} =
  L_2^{(1)}(0,0) = L_0^{(0)}(\pi,0) = p_D, &
  & \sbar_{0,2} =
  L_2^{(1)}(\pi,0) = L_0^{(0)}(0,0) = p_D.
\end{align*}
The orientation gradings now follow from a computation using the
tangent vectors given in Lemma \ref{lemma:tangent-vectors}.
\end{proof}

By acting with suitable powers of $\alpha_1$, we can generalize
Theorems \ref{theorem:w0-l0-l2}, \ref{theorem:gens-L0-L0}, and
\ref{theorem:gens-L2-L0} as follows:

\begin{corollary}
\label{cor:gens-s2-s1}
We have
\begin{align*}
  (W_0, L_0) &= \langle \alpha_0^{(-)},\, \beta_0^{(+)}\rangle, &
  (W_0, L_{\pm 2}) &=
  \langle \sigma_{\pm 2,0}^{(-)},\, \tau_{\pm 2, 0}^{(+)}\rangle,
\end{align*}
where
\begin{align*}
  \sigma_{-2,0} &= \alpha_1^{-2} \cdot \sigma_{0,2}, &
  \tau_{-2,0} &= \alpha_1^{-2} \cdot \tau_{0,2}.
\end{align*}
\end{corollary}

\begin{corollary}
\label{cor:fukaya-gens}
We have
\begin{align*}
  (L_n, L_n) &=
  \langle a_n^{(+)},\, b_n^{(-)},\, c_n^{(-)},\, d_n^{(+)} \rangle, &
  (L_n, L_{n\pm 2}) &=
  \langle
  r_{n\pm 2,n}^{(-)},\,\rbar_{n\pm 2,n}^{(-)},\,
  s_{n\pm 2,n}^{(+)},\,\sbar_{n\pm 2,n}^{(+)}
  \rangle,
\end{align*}
where
\begin{align*}
  &a_n = \alpha_1^n \cdot a_0, &
  &b_n = \alpha_1^n \cdot b_0, &
  &c_n = \alpha_1^n \cdot c_0, &
  &d_n = \alpha_1^n \cdot d_0, \\
  &r_{n+2,n} = \alpha_1^n \cdot r_{2,0}, &
  &\rbar_{n+2,n} = \alpha_1^n \cdot \rbar_{2,0}, &
  &s_{n+2,n} = \alpha_1^n \cdot s_{2,0}, &
  &\sbar_{n+2,n} = \alpha_1^n \cdot \sbar_{2,0}, \\
  &r_{n-2,n} = \alpha_1^{n-2} \cdot r_{0,2}, &
  &\rbar_{n-2,n} = \alpha_1^{n-2} \cdot \rbar_{0,2}, &
  &s_{n-2,n} = \alpha_1^{n-2} \cdot s_{0,2}, &
  &\sbar_{n-2,n} = \alpha_1^{n-2} \cdot \sbar_{0,2}.
\end{align*}
\end{corollary}

\subsection{Lagrangians and Lagrangian intersections for links in
  $S^3$}
\label{sec:lagrangians-s3}

Given a link $L$ in $S^3$, we can split $S^3$ into two solid
tori, each containing a 1-tangle, that are glued together using the
mapping class group element $s$ defined in Section \ref{ssec:mcg}.
As we describe in Section \ref{sec:s3}, we will choose the splitting
in such a way that one of the 1-tangles is trivial.
The other 1-tangle is projected onto an annulus to yield a tangle
diagram, which is closed with the trivial tangle to yield a
description of the original link $L$.
We thus define unperturbed Lagrangians that can used for describing
links in $S^3$:
\begin{align*}
  \Wbar_n &= s \cdot W_n.
\end{align*}
We can also use the relation
$s \alpha_1 = \beta_1 s$ to express
$\Wbar_n$ as
\begin{align*}
  \Wbar_n &= s \cdot \Wbar_n = s\alpha_1^n \cdot W_0 =
  \beta_1^n s \cdot W_0 = \beta_1^n \cdot \Wbar_n.
\end{align*}
The following result describes the action of $s^2$ on the Lagrangians
$W_n$ and $\Wbar_n$:

\begin{lemma}
We have
\begin{align*}
  s^2 \cdot W_n &= W_{1-n}, &
  s^2 \cdot \Wbar_n &= \Wbar_{1-n}.
\end{align*}
\end{lemma}

\begin{proof}
First we show that $s^2 \cdot W_0 = W_1$.
Recall from Theorem \ref{theorem:W0} that
$W_0 = \{B = \id\}$.
A computation gives
$(\tr B)(\alpha_1^{-1} s^2 \cdot W_0) = 2$, so
$\alpha_1^{-1} s^2 \cdot W_0 = W_0$ and thus
$s^2 \cdot W_0 = \alpha_1 \cdot W_0 = W_1$.
We then have
\begin{align*}
  s^2 \cdot W_n &= s^2 \alpha_1^n \cdot W_0 =
  \alpha_1^{-n} s^2 \cdot W_0 = \alpha_1^{-n} \cdot W_1 =
  W_{1-n}, &
  s^2 \cdot \Wbar_n &= s^3 \cdot W_n =
  s\cdot W_{1-n} =
  \Wbar_{1-n}.
\end{align*}
\end{proof}

We can depict the action of $s$ on the unperturbed Lagrangians $W_n$
and $\Wbar_n$ as
\begin{eqnarray*}
\begin{tikzcd}
  W_n \arrow{r}{s} &
  \Wbar_n \arrow{d}{s} \\
  \Wbar_{1-n} \arrow{u}{s} &
  W_{1-n}. \arrow{l}{s}
\end{tikzcd}
\end{eqnarray*}

We can explicitly describe the Lagrangian $\Wbar_0$:

\begin{lemma}
\label{lemma:W0bar}
We have
\begin{align*}
  \Wbar_0 = \{A = \id\} = R^*_3(T^2,2) \cap \{\alpha = 0\}.
\end{align*}
\end{lemma}

\begin{proof}
The result follows from Theorem \ref{theorem:W0}, the definition
$\Wbar_0 = s \cdot W_0$, and the known action of $s$.
\end{proof}

The Lagrangians $W_0$ and $W_1$ correspond to the 1-tangles shown in
Figure \ref{fig:torus-arc}(a) and \ref{fig:torus-arc}(b), so the
Lagrangians $\Wbar_0$ and $\Wbar_1$ correspond to an
overpass arc $A_+$ and an underpass arc $A_-$ for a link in $S^3$.
Recall that the braid group elements $\alpha_1$ and $\beta_1$ drag the
point $p_1$ around fundamental cycles of $T^2$.
For $W_0$ and $W_1$, the element $\beta_1$ drags $p_1$ around a
meridian.
For $\Wbar_0$ and $\Wbar_1$, the element $\alpha_1$ drags $p_1$ around
a meridian.
As we would expect, the braid group element $\beta_1$ acts trivially
on $W_0$ and $W_1$ and the braid group element $\alpha_1$ acts
trivially on $\Wbar_0$ and $\Wbar_1$:

\begin{lemma}
We have
\begin{align*}
  \beta_1 \cdot W_0 &= W_0, &
  \beta_1 \cdot W_1 &= W_1, &
  \alpha_1 \cdot \Wbar_0 &= \Wbar_0, &
  \alpha_1 \cdot \Wbar_1 &= \Wbar_1.
\end{align*}
\end{lemma}

\begin{proof}
Recall from Theorem \ref{theorem:W0} that
$W_0 = \{B = \id\}$.
A computation gives $(\tr B)(\beta_1^{-1} \cdot W_0) = 2$, so
$\beta_1 \cdot W_0 = W_0$.
We have
\begin{align*}
  \beta_1 \cdot W_1 &=
  \beta_1 s^2 \cdot W_0 = s^2 \beta_1^{-1} \cdot W_0 =
  s^2 \cdot W_0 =
  W_1, &
  \alpha_1 \cdot \Wbar_n &=
  \alpha_1 s \cdot W_n = s \beta_1^{-1} \cdot W_n =
  s \cdot W_n =
  \Wbar_n
\end{align*}
for $n=0, 1$.
\end{proof}

\begin{theorem}
\label{theorem:Wbar0}
We have
\begin{align*}
  (\Wbar_0, L_n) = \langle \wbar_{n,0}^{(+)} \rangle,
\end{align*}
where the point $\wbar_{n,0}$ lies in $R_3^*(T^2,2)$ and has coordinates
\begin{align*}
  (\alpha,\beta,s)(\wbar_{n,0}) &=
  \left\{
  \begin{array}{ll}
    (0, \epsilon, 0) &
    \quad \mbox{for $n \equiv 0 \pmod 4$,} \\
    (0, \pi/2, \pi/2-\epsilon) &
    \quad \mbox{for $n \equiv 1 \pmod 4$,} \\
    (0, \pi - \epsilon, 0) &
    \quad \mbox{for $n \equiv 2 \pmod 4$,} \\
    (0, \pi/2, -\pi/2+\epsilon) &
    \quad \mbox{for $n \equiv 3 \pmod 4$.}
  \end{array}
  \right.
\end{align*}
\end{theorem}

\begin{proof}
Recall from Lemma \ref{lemma:W0bar} that
$\Wbar_0 = R_3^*(T^2,2) \cap \{\alpha = 0\}$.
From Lemma \ref{theorem:L0} for $L_0$ it follows that
$\Wbar_0$ and $L_0$ intersect in a unique point $\wbar_{0,0}$ in
$R_3^*(T^2,2)$ with coordinates
\begin{align*}
  (\alpha,\beta,s)(\wbar_{0,0}) = (0,\epsilon,0).
\end{align*}
We apply $\alpha_1^n$ to both sides of the equation
$\Wbar_0 \cap L_0 = w_0$ to obtain
\begin{align*}
  (\alpha_1^n \cdot \Wbar_0) \cap
  (\alpha_1^n \cdot L_0) =
  \Wbar_0 \cap L_n =
  \alpha_1^n \cdot \wbar_{0,0},
\end{align*}
where we substituted $L_n = \alpha_1^n \cdot L_0$ and used the fact
that $\Wbar_0 = \alpha_1^n \cdot \Wbar_0$.
So $\Wbar_0$ and $L_n$ intersect in a single point
$\wbar_{n,0} := \alpha_1^n \cdot \wbar_{0,0}$.
Using the coordinates of $\wbar_{0,0}$ and the known action of
$\alpha_1$, we find that
\begin{align*}
  \wbar_{4r,0} &= \Wbar_0(\pi/2,-\pi/2+\epsilon) =
  L_0(3\pi/2,0), &
  \wbar_{4r+1,0} &= \Wbar_0(\epsilon,0), \\
  \wbar_{4r+2,0} &= \Wbar_0(\pi/2,\pi/2-\epsilon), &
  \wbar_{4r+3,0} &= \Wbar_0(\pi - \epsilon,0),
\end{align*}
so in particular $\wbar_{n,0}$ lies in $R_3^*(T^2,2)$,
and $\wbar_{n,0}$ has the coordinates stated in the theorem.

To check that $\Wbar_0$ and $L_n$ intersect transversely and to
compute the orientation grading of $\wbar_{n,0}$, we compute the
tangent vectors to $\Wbar_0$ and $L_n$ at $\wbar_{n,0}$.
We define the following basis of vector fields on
$R_3^*(T^2,2) \cap \{\sin\beta \neq 0\}$:
\begin{align*}
  u_1 &= \sin\beta\,\partial_{\alpha_2}, &
  u_2 &= \sin\beta\,\partial_{\beta_2}, &
  u_3 &= \sin\beta\,\partial_{s_2}, &
  u_4 &= \cos s\,\partial_{t_2}.
\end{align*}
These vector fields are invariant under the coordinate identifications
given in equation (\ref{eqn:coords-ident}) and are thus well-defined.
A computation using the symplectic form given in Theorem
\ref{theorem:symplectic-form} shows that these vector fields do in
fact form a basis.
We define functions $h^\mu:R(T^2,2) \rightarrow \Reals$:
\begin{align*}
  h^1 &= -\frac{1}{2}(\tr AB - \tr B), &
  h^2 &= -\frac{1}{2}(\tr B), &
  h^3 &= -\frac{1}{2}(\tr Ba), &
  h^4 &= \frac{1}{4}(\tr Ba + \tr B b).
\end{align*}
We define derivatives $\partialbar_\mu$:
\begin{align*}
  \partialbar_1 &= \partial_{\alpha_2}, &
  \partialbar_2 &= \partial_{\beta_2}, &
  \partialbar_3 &= \partial_{s_2}, &
  \partialbar_4 &= \partial_{t_2}.
\end{align*}
A calculation shows that $\partialbar_\nu h^\mu$ is diagonal at each
intersection point $\wbar_{n,0}$, so the functions $h^\mu$ serve as
convenient coordinates these points.
A calculation using the coordinates $h^\mu$, the vector fields
$u_\mu$, and the tangent vectors for $W_0$ given in Lemma
\ref{lemma:tangent-vectors} shows that $\Wbar_0$ and $L_n$ intersect
transversely at $\wbar_{n,0}$ and $\wbar_{n,0}$ has positive
intersection grading.
\end{proof}

A similar argument proves:

\begin{theorem}
\label{theorem:Wbar1}
We have
\begin{align*}
  (\Wbar_1, L_n) = \langle \wbar_{n,1}^{(+)} \rangle,
\end{align*}
where the point $\beta_1^{-1} \cdot \wbar_{n,1}$ lies in
$R_3^*(T^2,2)$ and has coordinates
\begin{align*}
  (\alpha,\beta,\gamma)(\beta_1^{-1} \cdot \wbar_{n,1}) &=
  \left\{
  \begin{array}{ll}
    (0, \nu_0, 0) &
    \quad \mbox{for $n \equiv 0 \pmod 4$,} \\
    (0, \pi/2, \pi/2-\nu_0) &
    \quad \mbox{for $n \equiv 1 \pmod 4$,} \\    
    (0, \pi - \nu_0, 0) &
    \quad \mbox{for $n \equiv 2 \pmod 4$,} \\    
    (0, \pi/2, -\pi/2+\nu_0) &
    \quad \mbox{for $n \equiv 3 \pmod 4$,}
  \end{array}
  \right.
\end{align*}
where $\nu_0 = \epsilon \sin\phi_0$ and $\phi_0$ is defined such that
\begin{align*}
  \phi_0 + \epsilon \sin \phi_0 = \pi/2.
\end{align*}
\end{theorem}

\subsection{$A_\infty$ operations for $\calL$}
\label{ssec:ops-fukaya}

We would like to determine the $A_\infty$ operations involving the
Lagrangians $L_0$ and $L_2$, but a direct calculation of these
operations does not appear feasible.
Instead, we conjecture the operations we will need based on certain
considerations.

Let $E \rightarrow R_2^*(T^2,2)$ denote the restriction of the tangent
bundle of $R^*(T^2,2)$ to the 2-dimensional submanifold $R_2^*(T^2,2)$.
The bundle $E$ has the structure of a symplectic vector bundle with
symplectic form given by Theorem \ref{theorem:symplectic-form}.
We can split $(E,\omega)$ as $(E_1,\omega_1) \oplus (E_2,\omega_2)$,
where
\begin{align*}
  \omega_1 &=
  d\alpha_1 \wedge d\beta_1 = d\alpha_1 \wedge d\beta_1, &
  \omega_2 &=
  \sin\alpha_1\,ds_1 \wedge dt_1 = -\sin\beta_2\,ds_2 \wedge dt_2,
\end{align*}
and $E_1$ and $E_2$ are the kernel of
$v \mapsto \omega_2(v,-)$ and $v \mapsto \omega_1(v,-)$, respectively.
We can choose compatible almost complex structures $J_k$ on
$(E_k, \omega_k)$ and define an almost complex structure
$J = J_1 \oplus J_2$ on $E$.
We can extend $J$ to an almost-complex structure on $R^*(T^2,2)$, and
with this choice of complex structure pseudo-holomorphic disks in
$R_2^*(T^2,2)$ are also pseudo-holomorphic disks in $R^*(T^2,2)$.

To conjecture the $A_\infty$ operations, we plot the intersections of
the relevant Lagrangians with $R_2^*(T^2,2)$ and look for topological
disks within that 2-dimensional submanifold of $R^*(T^2,2)$.
For example, to conjecture $\mu_\calL^1$ operations for generators of
$(L_0,L_0)$, we look for topological disks in Figure \ref{fig:l0-l0}.
The only bigons we find are two canceling pairs of bigons
$a_0 \rightarrow d_0$, so we conjecture that $\mu_\calL^1 = 0$ for
$(L_0,L_0)$.
Based on Figure \ref{fig:l0-l2-l2-l0}, we conjecture that
$\mu_\calL^1 = 0$ for $(L_0,L_2)$ and $(L_2,L_0)$.

Next we consider product operations.
Based on counts of topological disks in $R_2^*(T^2,2)$, we conjecture
the following product operations:

\begin{enumerate}
\item
We have the following operations
$\mu_\calL^2:(L_0, L_0) \otimes (L_0, L_0) \rightarrow (L_0, L_0)$:
\begin{align*}
  & \mu_\calL^2(b_0,c_0) = \mu_\calL^2(c_0, b_0) = d_0, \\
  & \mu_\calL^2(a_0, a_0) = a_0, &
  & \mu_\calL^2(a_0, b_0) = \mu_\calL^2(b_0, a_0) = b_0, \\
  & \mu_\calL^2(a_0, c_0) = \mu_\calL^2(c_0, a_0) = c_0, &
  & \mu_\calL^2(a_0, d_0) = \mu_\calL^2(d_0, a_0) = d_0.
\end{align*}

\item
We have the following operations
$\mu_\calL^2:(L_2, L_0) \otimes (L_0, L_2) \rightarrow (L_0, L_0)$:
\begin{align*}
  & \mu_\calL^2(\sbar_{0,2} \cdot \rbar_{2,0}) =
  \mu_\calL^2(r_{0,2}, s_{2,0}) = c_0, \\
  & \mu_\calL^2(r_{0,2}, r_{2,0}) =
  \mu_\calL^2(\rbar_{0,2}, \rbar_{2,0}) =
  \mu_\calL^2(s_{0,2}, s_{2,0}) =
  \mu_\calL^2(\sbar_{0,2}, \sbar_{2,0}) = d_0.
\end{align*}

\item
We have the following operations
$\mu_\calL^2:(L_0, L_2) \otimes (L_2, L_0) \rightarrow (L_2, L_2)$:
\begin{align*}
  & \mu_\calL^2(r_{2,0}, \sbar_{0,2}) = 
  \mu_\calL^2(s_{2,0}, \rbar_{0,2}) = c_2, \\
  & \mu_\calL^2(r_{2,0}, r_{0,2}) =
  \mu_\calL^2(\rbar_{2,0}, \rbar_{0,2}) =
  \mu_\calL^2(s_{2,0}, s_{0,2}) =
  \mu_\calL^2(\sbar_{2,0}, \sbar_{0,2}) = d_2.
\end{align*}

\item
We have the following operations
$\mu_\calL^2:(L_0, L_2) \otimes (L_0, L_0) \rightarrow (L_0, L_2)$:
\begin{align*}
  & \mu_\calL^2(\rbar_{2,0}, b_0) = \sbar_{2,0}, &
  & \mu_\calL^2(s_{2,0}, b_0) = r_{2,0}, \\
  & \mu_\calL^2(r_{2,0}, a_0) = r_{2,0}, &
  & \mu_\calL^2(\rbar_{2,0}, a_0) = \rbar_{2,0}, &
  & \mu_\calL^2(s_{2,0}, a_0) = s_{2,0}, &
  & \mu_\calL^2(\sbar_{2,0}, a_0) = \sbar_{2,0}.
\end{align*}

\item
We have the following operations
$\mu_\calL^2:(L_2, L_0) \otimes (L_2, L_2) \rightarrow (L_2, L_0)$:
\begin{align*}
  &\mu_\calL^2(\sbar_{0,2}, b_2) = r_{0,2}, &
  &\mu_\calL^2(\rbar_{0,2}, b_2) = s_{0,2}, \\
  &\mu_\calL^2(r_{0,2}, a_2) = r_{0,2}, &
  &\mu_\calL^2(\rbar_{0,2}, a_2) = \rbar_{0,2}, &
  &\mu_\calL^2(s_{0,2}, a_2) = s_{0,2}, &
  &\mu_\calL^2(\sbar_{0,2}, a_2) = \sbar_{0,2}.
\end{align*}

\item
We have the following operations
$\mu_\calL^2:(L_2, L_2) \otimes (L_0, L_2) \rightarrow (L_0, L_2)$:
\begin{align*}
  & \mu_\calL^2(b_2, r_{2,0}) = \sbar_{2,0}, &
  & \mu_\calL^2(b_2, s_{2,0}) = \rbar_{2,0}, \\
  & \mu_\calL^2(a_2, r_{2,0}) = r_{2,0}, &
  & \mu_\calL^2(a_2, \rbar_{2,0}) = \rbar_{2,0}, &
  & \mu_\calL^2(a_2, s_{2,0}) = s_{2,0}, &
  & \mu_\calL^2(a_2, \sbar_{2,0}) = \sbar_{2,0}.
\end{align*}

\item
We have the following operations
$\mu_\calL^2:(L_0, L_0) \otimes (L_2, L_0) \rightarrow (L_2, L_0)$:
\begin{align*}
  & \mu_\calL^2(b_0, \sbar_{0,2}) = \rbar_{0,2}, &
  & \mu_\calL^2(b_0, r_{0,2}) = s_{0,2}, \\
  & \mu_\calL^2(a_0, r_{0,2}) = r_{0,2}, &
  & \mu_\calL^2(a_0, \rbar_{0,2}) = \rbar_{0,2}, &
  & \mu_\calL^2(a_0, s_{0,2}) = s_{0,2}, &
  & \mu_\calL^2(a_0, \sbar_{0,2}) = \sbar_{0,2}.
\end{align*}

\end{enumerate}

There are several consistency checks we can perform on our conjectured
operations.
First, the product operations must respect orientation gradings: the
orientation grading of a product of generators must be the product of
their orientation gradings.
Second, if $\mu_\calL^1 = 0$, then the relations for the $A_\infty$
operations imply that the product operations must be associative.
Third, the product operations must satisfy the following invariance
result:

\begin{theorem}
\label{theorem:invariance}
The product operations
$(L_{n_1}, L_{n_2}) \otimes (L_{n_0}, L_{n_1}) \rightarrow
(L_{n_0}, L_{n_2})$
for $n_1, n_2, n_3 \in \{0, 2\}$
are invariant under the substitutions
\begin{align*}
  &a_0 \leftrightarrow a_2, &
  &b_0 \leftrightarrow b_2, &
  &c_0 \leftrightarrow c_2^, &
  &d_0 \leftrightarrow d_2, \\
  &r_{2,0} \leftrightarrow r_{0,2}, &
  &\rbar_{2,0} \leftrightarrow \rbar_{0,2}, &
  &s_{2,0} \leftrightarrow \sbar_{0,2}, &
  &\sbar_{2,0} \leftrightarrow s_{0,2}.
\end{align*}
\end{theorem}

\begin{proof}
A symplectic automorphism of $R^*(T^2,2)$ must preserve the product
structure.
Invariance under the stated substitutions follows from considering the
action of the symplectic automorphism $\alpha_1 s^2$ of
$R^*(T^2,2)$ on the generators.
To understand this action, it is helpful to consider the limit in
which both the Hamiltonian and holonomy perturbations of the
Lagrangians go to zero.
Consider, for example, the generators of $(L_0,L_2)$.
In the limit that both perturbations go to zero we have
$L_0^{(1)} \rightarrow W_0$ and
$L_2^{(0)} \rightarrow W_2$.
The automorphism $\alpha_1 s^2$ maps $W_n$ to $W_{2-n}$.
Recall from the proof of Theorem \ref{theorem:gens-L2-L0} that in the
limit in which the Hamiltonian perturbation goes to zero we have
\begin{align*}
  & r_{0,2} =
  L_2^{(1)}(0,0) = L_0^{(0)}(0,0) = p_D, &
  & \rbar_{0,2} =
  L_2^{(1)}(\pi,0) = L_0^{(0)}(\pi,0) = p_D, \\
  & s_{0,2} =
  L_2^{(1)}(0,0) = L_0^{(0)}(\pi,0) = p_D, &
  & \sbar_{0,2} =
  L_2^{(1)}(\pi,0) = L_0^{(0)}(0,0) = p_D,
\end{align*}
so the action of $\alpha_1 s^2$ on the generators of $(L_0,L_2)$ is
as stated.
The argument for the generators of $(L_0,L_0)$ is similar.
\end{proof}

It is straightforward to check that the conjectured product operations
do in fact satisfy all three conditions.
By acting with suitable powers of $\alpha_1$, we obtain the following
generalized conjecture:

\begin{conjecture}
\label{conj:prod-perturbed}
The following are the only nonzero product operations
$(L_{n_1}, L_{n_2}) \otimes (L_{n_0}, L_{n_1}) \rightarrow
(L_{n_0}, L_{n_2})$
for $n_i - n_j \in \{0, \pm 2\}$.
We have
\begin{align*}
  \mu_\calL^2(a_n,x) &= x, &
  \mu_\calL^2(y,a_n) &= y
\end{align*}
whenever these operations are defined.
We have
\begin{align*}
  \mu_\calL^2(b_n, c_n) = \mu_\calL^2(c_n, b_n) = d_n.
\end{align*}
\begin{align*}
  \mu_\calL^2(r_{n,n+2}, s_{n+2,n}) =
  \mu_\calL^2(\rbar_{n,n+2}, \sbar_{n+2,n}) =
  \mu_\calL^2(s_{n,n-2}, \rbar_{n-2,n}) =
  \mu_\calL^2(\sbar_{n,n-2}, r_{n-2,n}) = c_n, \\
  \mu_\calL^2(r_{n,n \pm 2}, r_{n \pm 2, n}) =
  \mu_\calL^2(\rbar_{n,n \pm 2}, \rbar_{n \pm 2, n}) =
  \mu_\calL^2(s_{n,n \pm 2}, s_{n \pm 2, n}) =
  \mu_\calL^2(\sbar_{n,n \pm 2}, \sbar_{n \pm 2, n}) = d_n.
\end{align*}
\begin{align*}
  &\mu_\calL^2(b_{n+2}, r_{n+2,n}) =
  \mu_\calL^2(\rbar_{n+2, n}, b_n) = \sbar_{n+2,n}, &
  &\mu_\calL^2(b_{n-2}, r_{n-2,n}) =
  \mu_\calL^2(\rbar_{n-2, n}, b_n) = s_{n-2,n}, \\
  &\mu_\calL^2(s_{n+2,n}, b_n) = r_{n+2, n}, &
  &\mu_\calL^2(b_{n+2}, s_{n+2,n}) = \rbar_{n+2,n}, \\
  &\mu_\calL^2(\sbar_{n-2,n}, b_n) = r_{n-2, n}, &
  &\mu_\calL^2(b_{n-2}, \sbar_{n-2,n}) = \rbar_{n-2,n}.
\end{align*}
\end{conjecture}

As described in Section \ref{ssec:a-infinity}, part of the data of
an $A_\infty$ category is an integer grading on the morphism spaces.
In the case of the Fukaya category of a symplectic manifold $M$, one
would like this integer grading to be given by the Maslov grading of
the Lagrangian intersection points that generate these spaces, but
this is possible only if two conditions are met.
First, the first Chern class of $M$ must be 2-torsion.
Second, the Maslov classes of the Lagrangians must vanish.
If these conditions are met, then the integer grading depends on the
additional data of a choice of \emph{graded lifts} of the Lagrangians.
We do not know if these conditions are met for our category $\calL$,
so rather than use Maslov gradings we simply assign integer gradings
to the generators of $(L_n,L_n)$ and $(L_n,L_{n\pm 2})$ as follows:
\begin{align*}
  (L_n, L_n) &=
  \langle a_n^{(0)},\, b_n^{(0)},\, c_n^{(-2)},\, d_n^{(-2)} \rangle, &
  (L_n, L_{n\pm 2}) &=
  \langle
  r_{n \pm 2, n}^{(-1)},\, \rbar_{n \pm 2,0}^{(-1)},\,
  s_{n \pm 2, n}^{(-1)},\, \sbar_{n \pm 2,0}^{(-1)} \rangle.
\end{align*}

Given a pair of Lagrangians $N_1$ and $N_2$, one defines the
\emph{Floer cohomology $HF(N_1,N_2)$} to be the cohomology of the
vector space $(N_1,N_2)$ with respect to the differential
$\mu^1:(N_1,N_2) \rightarrow (N_1,N_2)$.
If conditions described above for $(N_1,N_2)$ to carry an integer
grading are met, then $HF(N_1,N_2)$ is integer-graded and we have
a \emph{Poincar\'{e} duality} isomorphism
\begin{align*}
  HF^*(N_1,N_2) \rightarrow HF^{n-*}(N_2,N_1),
\end{align*}
where $2n$ is the real dimension of the symplectic manifold $M$.
We note that integer gradings we have defined on $(L_n,L_n)$ and
$(L_n, L_{n \pm 2})$ are not consistent with Poincar\'{e} duality,
though they are consistent with gradings collapsed to $\Ints_4$.

\subsection{$A_\infty$ operations for unperturbed Lagrangians}

We would also like the know certain product operations involving
the unperturbed Lagrangians $W_n$.
Strictly speaking, these Lagrangians are not objects of the category
$\calL$; rather, they will be used to define $A_\infty$ functors
$\calG_{W_n}:\calL \rightarrow \Ch$ as described in
Section \ref{ssec:a-infinity}.
We proceed as in Section \ref{ssec:ops-fukaya} and count topological
disks in $R_2^*(T^2,2)$.
Based on these counts, we conjecture that $\mu_1 = 0$ for
$(W_0, L_0)$, $(W_0, L_2)$, $(W_2, L_0)$, and $(W_2,L_2)$.
We also conjecture the following product operations:

\begin{enumerate}
\item
We have the following operations
$\mu^2:(L_0, L_0) \otimes (W_0, L_0) \rightarrow (W_0, L_0)$:
\begin{align*}
  \mu^2(a_0, \alpha_0) &= \alpha_0, &
  \mu^2(a_0, \beta_0) &= \beta_0, &
  \mu^2(c_0, \alpha_0) &= \beta_0.
\end{align*}

\item
We have the following operations
$\mu^2:(L_2, L_0) \otimes (W_0, L_2) \rightarrow (W_0, L_0)$:
\begin{align*}
  \mu^2(\sbar_{0,2}, \tau_{2,0}) =
  \mu^2(r_{0,2}, \sigma_{2,0}) = \beta_0.
\end{align*}

\item
We have the following operations
$\mu^2:(L_0, L_2) \otimes (W_2, L_0) \rightarrow (W_2, L_2)$:
\begin{align*}
  \mu^2(r_{2,0}, \sigma_{0,2}) &= \beta_2, &
  \mu^2(s_{2,0}, \tau_{0,2}) &= \beta_2.
\end{align*}

\item
We have the following operations
$\mu^2:(L_0, L_2) \otimes (W_0, L_0) \rightarrow (W_0, L_2)$:
\begin{align*}
  & \mu^2(\rbar_{2,0}, \alpha_0) = \tau_{2,0}, &
  & \mu^2(s_{2,0}, \alpha_0) = \sigma_{2,0}.
\end{align*}

\item
We have the following operations
$\mu^2:(L_2, L_0) \otimes (W_2, L_2) \rightarrow (W_2, L_0)$:
\begin{align*}
  \mu^2(\sbar_{0,2}, \alpha_2) &= \sigma_{0,2}, &
  \mu^2(\rbar_{0,2}, \alpha_2) &= \tau_{0,2}.
\end{align*}

\item
We have the following operations
$\mu^2:(L_2, L_2) \otimes (W_0, L_2) \rightarrow (W_0, L_2)$:
\begin{align*}
  \mu^2(a_2, \sigma_{2,0}) &= \sigma_{2,0}, &
  \mu^2(a_2, \tau_{2,0}) &= \tau_{2,0}, &
  \mu^2(b_2, \sigma_{2,0}) &= \tau_{2,0}.
\end{align*}

\item
We have the following operations
$\mu^2:(L_0, L_0) \otimes (W_2, L_0) \rightarrow (W_2, L_0)$:
\begin{align*}
  \mu^2(a_0, \sigma_{0,2}) &= \sigma_{0,2}, &
  \mu^2(a_0, \tau_{0,2}) &= \tau_{0,2}, &
  \mu^2(b_0, \sigma_{0,2}) &= \tau_{0,2}.
\end{align*}

\end{enumerate}

It is straightforward to check that the conjectured product operations
are consistent with orientation gradings and are associative when
combined with the product operations for unperturbed Lagrangians
conjectured in Section \ref{ssec:ops-fukaya}.
By acting with suitable powers of $\alpha_1$, we obtain the following
generalized conjecture:

\begin{conjecture}
\label{conj:prod-unperturbed}
The following are the only nonzero product operations
$(L_{n_1}, L_{n_2}) \otimes (W_0, L_{n_1}) \rightarrow (W_0, L_{n_2})$
for $n_1, n_2, n_1 - n_2 \in \{0, \pm 2\}$.
We have
\begin{align*}
  \mu^2(a_n,x) &= x, &
  \mu^2(y,a_n) &= y
\end{align*}
whenever these operations are defined.
We have
\begin{align*}
  \mu^2(c_0,\alpha_0) = 
  \mu^2(\sbar_{0,2}, \tau_{2,0}) =
  \mu^2(s_{-2,0}, \tau_{-2,0}) =
  \mu^2(r_{0,2}, \sigma_{2,0}) =
  \mu^2(r_{-2,0}, \sigma_{-2,0}) = \beta_0.
\end{align*}
\begin{align*}
  &\mu^2(s_{2,0}, \alpha_0) = \sigma_{2,0}, &
  &\mu^2(\rbar_{2,0}, \alpha_0) =
  \mu^2(b_2,\sigma_{2,0}) = \tau_{2,0}, \\
  &\mu^2(\sbar_{-2,0}, \alpha_0) = \sigma_{-2,0}, &
  &\mu^2(\rbar_{-2,0}, \alpha_0) =
  \mu^2(b_{-2},\sigma_{-2,0}) = \tau_{-2,0}.
\end{align*}
\end{conjecture}

We can depict the conjectured product operations as follows:
\begin{eqnarray*}
\begin{tikzcd}
  \sigma_{-2,0}
  \arrow{d}[swap]{b_{-2}}
  \arrow{dr}[swap]{r_{0,-2}} &[2 em]
  \alpha_0
  \arrow{d}{c_0}
  \arrow{l}[swap]{\sbar_{-2,0}}
  \arrow{dr}{\rbar_{2,0}} &[2 em]
  \sigma_{2,0}
  \arrow{d}{b_2} \\
  \tau_{-2,0} &[2 em]
  \beta_0 &[2 em]
  \tau_{2,0} \arrow{l}[swap]{\sbar_{0,2}}
\end{tikzcd}
&&
\begin{tikzcd}
  \sigma_{-2,0}
  \arrow{d}[swap]{b_{-2}} &[2 em]
  \alpha_0
  \arrow{d}{c_0}
  \arrow{dl}[swap]{\rbar_{-2,0}}
  \arrow{r}{s_{2,0}} &[2 em]
  \sigma_{2,0} \arrow{dl}{r_{0,2}}
  \arrow{d}{b_2} \\
  \tau_{-2,0}
  \arrow{r}{s_{0,-2}} &[2 em]
  \beta_0 &[2 em]
  \tau_{2,0}.
\end{tikzcd}
\end{eqnarray*}

We assign integer quantum gradings to the generators of $(W_0,L_0)$
and $(W_0, L_{\pm 2})$ as follows:
\begin{align*}
  (W_0, L_0) &=
  \langle \alpha_0^{(0)},\, \beta_0^{(-2)} \rangle, &
  (W_0, L_{\pm 2}) &=
  \langle \sigma_{\pm 2, 0}^{(-1)},\, \tau_{\pm 2, 0}^{(-1)} \rangle.
\end{align*}
It is straightforward to check that these gradings are consistent with
the product operations in Conjecture \ref{conj:prod-unperturbed}.

\section{Twisted complexes corresponding to tangle diagrams}
\label{sec:complexes}

In Section \ref{ssec:fuaka} we defined a Fukaya category $\calL$ for
$R^*(T^2,2)$ whose objects are the perturbed Lagrangians $L_n$ and
whose $A_\infty$ operations are given by counting pseudo-holomorphic
disks.
As described in Section \ref{ssec:a-infinity}, we can define an
$A_\infty$ category $\Tw \calL$ of twisted complexes over $\calL$.
Our goal in this section is to construct a twisted complex
$(X,\delta)$ in $\Tw \calL$ from a 1-tangle diagram $T$ in the
annulus.

\subsection{Conjectures regarding $\calL$}
\label{sec:conj-fukaya}

To construct the twisted object $(X,\delta)$, we need certain
$A_\infty$ operations of $\calL$ that we are currently unable to
compute.
We make the following conjecture regarding the operations we will
need:

\begin{conjecture}
\label{conj:fukaya}
We have identity elements $a_n \in (L_n, L_n)$ that satisfy
\begin{align*}
  &\mu_\calL^2(a_n,x) = x, &
  &\mu_\calL^2(y,a_n) = y
\end{align*}
whenever these operations are defined.
We have elements $c_n \in (L_n, L_n)$,
$p_{n+2,n} \in (L_n, L_{n+2})$, and
$q_{n-2,n} \in (L_n, L_{n-2})$ that satisfy
\begin{align*}
  &\mu_\calL^2(q_{n,n+2},p_{n+2,n}) =
  \mu_\calL^2(p_{n,n-2},q_{n-2,n}) = c_n, &
  &\mu_\calL^2(c_n,c_n) = 0, \\
  &\mu_\calL^2(c_{n+2}, p_{n+2,n}) = \mu_\calL^2(p_{n+2,n}, c_n) = 0, &
  &\mu_\calL^2(c_{n-2}, q_{n-2,n}) = \mu_\calL^2(q_{n-2,n}, c_n) = 0.
\end{align*}
All operations of the form
$\mu_\calL^m(x_m, \cdots, x_1)$
for $m \neq 2$ and
$x_m, \cdots, x_1 \in
\{a_n,\, c_n,\, p_{n+2,n},\, q_{n-2,n} \mid n \in \Ints \}$ are zero.
\end{conjecture}

We do not conjecture values for the products
$\mu_\calL^2(p_{n+4,n+2},p_{n+2,n})$ and $\mu_\calL^2(q_{n-4,n-2},q_{n-2,n})$.

Conjecture \ref{conj:fukaya} is motivated by several considerations.
First, it is consistent with the conjectured $A_\infty$ relations
described in Section \ref{ssec:ops-fukaya}.
The elements $a_n$ and $c_n$ are the generators described in Corollary
\ref{cor:fukaya-gens}.
The elements $p_{n+2,n}$ and $q_{n-2,n}$ are defined in terms of the
generators described in Corollary \ref{cor:fukaya-gens} in one of two
ways.
We could either define
\begin{align*}
  p_{n+2,n}^{(-)} &= r_{n+2,n}^{(-)} + \rbar_{n+2,n}^{(-)}, &
  q_{n-2,n}^{(+)} &= s_{n-2,n}^{(+)} + \sbar_{n-2,n}^{(+)},
\end{align*}
or
\begin{align*}
  p_{n+2,n}^{(+)} &= s_{n+2,n}^{(+)} + \sbar_{n+2,n}^{(+)}, &
  q_{n-2,n}^{(-)} &= r_{n-2,n}^{(-)} + \rbar_{n-2,n}^{(-)},
\end{align*}
where we have indicated orientation gradings using superscripts.
For some purposes it does not matter which of the two definitions is
used, but when it does matter we will indicate the definition we have
in mind by specifying the orientation gradings, which are opposite for
the two definitions.
It is straightforward to check that for both definitions Conjecture
\ref{conj:fukaya} is consistent with the product operations in
Conjecture \ref{conj:prod-perturbed}.

A second motivation for Conjecture \ref{conj:fukaya} is that it
is a natural generalization of a corresponding statement regarding the
$A_\infty$ operations for the Fukaya category of $R^*(S^2,4)$, which
as we discuss in Appendix  \ref{sec:pillowcase} appears to be closely
related to the Fukaya category of $R^*(T^2,2)$.
Indeed, Conjecture \ref{conj:fukaya} is in fact true if we restrict
$n$ to $n \in \{\pm 1\}$ and reinterpret it as a statement about
$R^*(S^2,4)$.

Recall that part of the data of an $A_\infty$ category is an integer
grading on the morphism spaces.
We will refer to this grading as a \emph{quantum grading} $q$.
We will also take the morphism spaces to carry a second integer
grading that we refer to as a \emph{homological grading} $h$.
We assign bigradings $(h,q)$ to $a_n$, $c_n$, $p_{n+2,n}$, and
$q_{n-2,n}$ as follows, as indicted by superscripts:
\begin{align}
  \label{eqn:bigradings}
  a_n^{(0,0)} &&
  c_n^{(0,-2)}, &&
  p_{n+2,n}^{(0,-1)}, &&
  q_{n-2,n}^{(0,-1)}.
\end{align}
The quantum gradings of these generators are consistent with the
grading assignments described at the end of Section
\ref{ssec:ops-fukaya}.
We say that $A_\infty$ operations \emph{respect bigradings} if
whenever
\begin{align*}
  \mu_\calL^m(x_m^{(h_m,q_m)}, \cdots, x_1^{(h_1,q_1)}) = y
\end{align*}
is nonzero for homogeneous vectors
$x_m^{(h_m,q_m)}, \cdots, x_1^{(h_1,q_1)}$, the bigrading
$(h,q)$ of $y$ is given by
\begin{align*}
  h &= h_m + \cdots + h_1 + 2 - m, &
  q &= q_m + \cdots + q_1 + 2 - m.
\end{align*}
As discussed in Section \ref{ssec:a-infinity}, the operations of an
$A_\infty$ category are required to respect quantum gradings.
Given the bigrading assignments in equation (\ref{eqn:bigradings}), 
the $A_\infty$ operations described in Conjecture \ref{conj:fukaya}
respect bigradings.

\subsection{Vector spaces and linear maps}
\label{ssec:linear-maps}

As described in Section \ref{ssec:a-infinity}, given an $A_\infty$
category $\calA$ one can construct a corresponding $A_\infty$ category
$\Sigma \calA$ called the additive enlargement of $\calA$.
The objects of
$\Sigma \calA$ are constructed from objects of $\calA$ and vector
spaces, and the morphisms of $\Sigma \calA$ are constructed from
morphisms of $\calA$ and linear maps.
We define here the vector spaces and linear maps that we will use to
construct objects and morphisms of $\Sigma \calL$.

The vector spaces that are used to construct objects of $\Sigma \calL$
are required to carry an integer grading.
We will refer to this grading as a \emph{quantum grading} $q$.
We will take these vector spaces to carry a second integer grading
that we refer to as a \emph{homological grading} $h$.
Given a bigraded vector space $V$, we use the notation $v^{(h,q)}$ to
indicate that a homogeneous vector $v \in V$ has bigrading $(h,q)$.
We define the vector space $V[h,q]$ to be $V$ with gradings shifted
upwards by $(h,q)$, so if $v \in V$ is homogeneous with bigrading
$(h_v, q_v)$ then the corresponding vector $v \in V[h,q]$ is
homogeneous with bigrading $(h_v + h, q_v + q)$.
We define $\F$ to be the field of two elements, where $1 \in \F$ is
assigned bigrading $(0,0)$.
We define a bigraded $\F$-vector space
\begin{align*}
  A = \langle e^{(0,1)},\, x^{(0,-1)} \rangle.
\end{align*}

We define the following $\F$-linear maps.
We define \emph{unit maps} $\eta^{(0,1)}$ and $\etadot^{(0,-1)}$:
\begin{align*}
  &\eta:\F \rightarrow A, &
  &\eta(1) = e, \\
  &\etadot:\F \rightarrow A, &
  &\etadot(1) = x.
\end{align*}
We define \emph{counit maps} $\epsilon^{(0,1)}$ and
$\epsilondot^{(0,-1)}$:
\begin{align*}
  &\epsilon:A \rightarrow \F, &
  &\epsilon(e) = 0, \qquad
  \epsilon(x) = 1, \\
  &\epsilondot:A \rightarrow \F, &
  &\epsilondot(e) = 1, \qquad
  \epsilondot(x) = 0.
\end{align*}
We define a \emph{raising map} $\id_{ex}^{(0,2)}$ and a
\emph{lowering map} $\id_{xe}^{(0,-2)}$:
\begin{align*}
  &\id_{ex}:A \rightarrow A, &
  &\id_{ex}(e) = 0, \qquad
  \id_{ex}(x) = e, \\
  &\id_{xe}:A \rightarrow A, &
  &\id_{xe}(e) = x, \qquad
  \id_{xe}(x) = 0.
\end{align*}
We define a \emph{multiplication map} $m^{(0,-1)}$:
\begin{align*}
  &m:A \otimes A \rightarrow A, &
  &m(e \otimes e) = e, \qquad
  m(e \otimes x) = m(x \otimes e) = x, \qquad
  m(x \otimes x) = 0.
\end{align*}
We define a \emph{comultiplication map} $\Delta^{(0,-1)}$:
\begin{align*}
  &\Delta:A \rightarrow A \otimes A, &
  &\Delta(e) = e \otimes x + x \otimes e, \qquad
  \Delta(x) = x \otimes x.
\end{align*}
We define an \emph{identity map} $I_r^{(0,0)}$:
\begin{align*}
  & I_r:A^{\otimes r} \rightarrow A^{\otimes r}, &
  I_r &= \id^{\otimes r}.
\end{align*}
We define a \emph{raising map} $\Sigma_r^{(0,2)}$:
\begin{align*}
  & \Sigma_r:A^{\otimes r} \rightarrow A^{\otimes r}, &
  \Sigma_r &=
  \sum_{s=0}^{r-1} \id^{\otimes s} \otimes \id_{ex} \otimes
  \id^{\otimes(r-1-s)}.
\end{align*}
Note that for $r=0$ we have $A^{\otimes r} = \F$, so
$I_0 = \id_{\F}$ and $\Sigma_0 = 0$.
For simplicity we will often omit tensor product symbols; for example
we write $\etadot \otimes I_r$ as $\etadot I_r$.
We record here several useful identities involving $\Sigma_r$:
\begin{align}
  \label{eqn:identities-sigma}
  &\Sigma_{a+b} = \Sigma_a I_b + I_a \Sigma_b, &
  &\Sigma_2 \circ \Delta = \Delta \circ \Sigma_1, &
  &\Sigma_1 \circ m = m \circ \Sigma_2.
\end{align}

\subsection{Planar tangles and saddles}
\label{ssec:saddles}

We will assign objects of $\Sigma \calL$ to planar 1-tangles in the
annulus and morphisms of $\Sigma \calL$ to saddles between pairs of
such tangles.
Here we describe these assignments.

Fix points $p_1$ and $p_2$ on the outer and inner bounding circle of
the annulus $S^1 \times [0,1]$.
A \emph{planar 1-tangle} is a compact 1-dimensional submanifold of the
annulus with boundary $\{p_1, p_2\}$.
A planar 1-tangle consists of an \emph{arc component} connecting $p_1$
to $p_2$, together with a finite number of \emph{circle components}.
We orient the arc component in the direction from the outer boundary
point $p_1$ to the inner boundary point $p_2$.

For each integer $n$ we define a planar tangle $P_n$ for which the
arc component winds $n$ times clockwise around the annulus, as shown
in Figure \ref{fig:planar-tangles}.
We define $P_n(r)$ to be the set of planar tangles whose arc component
is isotopic to $P_n$ and which contains $r$ circle components.
We also let $P_n(r)$ denote any specific tangle in the set $P_n(r)$.
We say that $n$ is the \emph{winding number} and $r$ is the
\emph{circle number} of the tangle $P_n(r)$.
To the planar tangle $P_n(r)$ we assign the following object of
$\Sigma \calL$:
\begin{align*}
  L_n \otimes A^{\otimes r},
\end{align*}
where $A$ is the 2-dimensional bigraded vector space defined in
Section \ref{ssec:linear-maps}.

\begin{figure}
  \centering
  \includegraphics[scale=0.45]{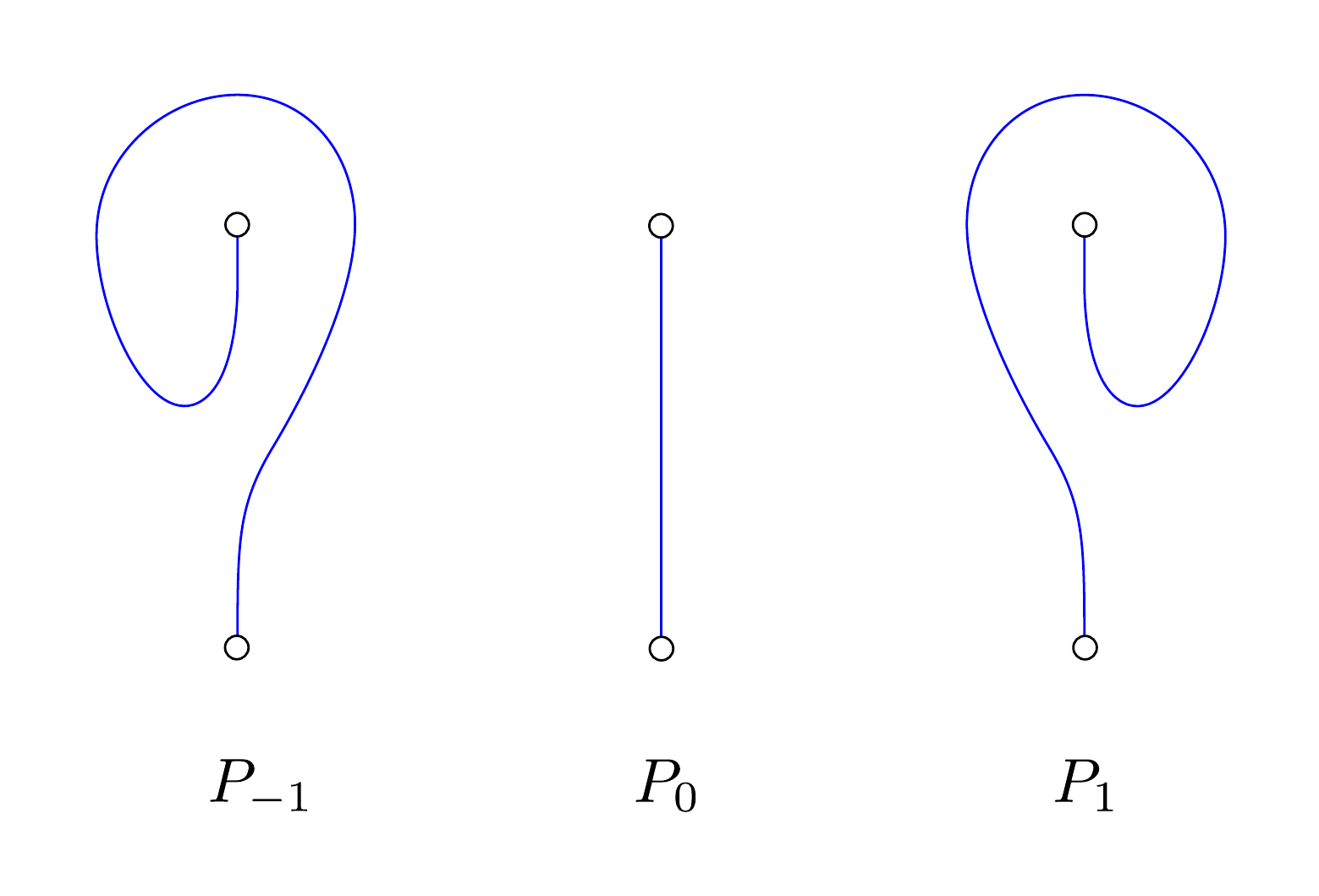}
  \caption{
    \label{fig:planar-tangles}
    Planar 1-tangles in the annulus.
  }
\end{figure}

Given a planar tangle $P_{n_1}(r_1)$, we can make a local replacement
of the form shown in Figure \ref{fig:saddle} to obtain a new
planar tangle $P_{n_2}(r_2)$.
We say that the local replacement defines a
\emph{saddle} $P_{n_1}(r_1) \rightarrow P_{n_2}(r_2)$.
The possible types of saddles are shown in Figure
\ref{fig:saddle-types}.
For saddles $P_n(r) \rightarrow P_n(r+1)$ or
$P_n(r+1) \rightarrow P_n$ that split or merge a circle component $C$
from the arc component, we say that circle components of $P_n(r+1)$
that lie inside (outside) $C$, and circle components of $P_n(r)$ whose
counterparts in $P_n(r+1)$ lie inside (outside) $C$, are
\emph{enclosed (non-enclosed) circles}.
To each saddle $P_{n_1}(r_1) \rightarrow P_{n_2}(r_2)$ we assign a
corresponding \emph{differential}, which is a morphism
\begin{align*}
  L_{n_1} \otimes A^{\otimes r_1} \rightarrow
  L_{n_2} \otimes A^{\otimes r_n}
\end{align*}
in $\Sigma \calL$ with bigrading $(h,q) = (0,-1)$.
The assignments are as follows:

\begin{figure}
  \centering
  \includegraphics[scale=0.50]{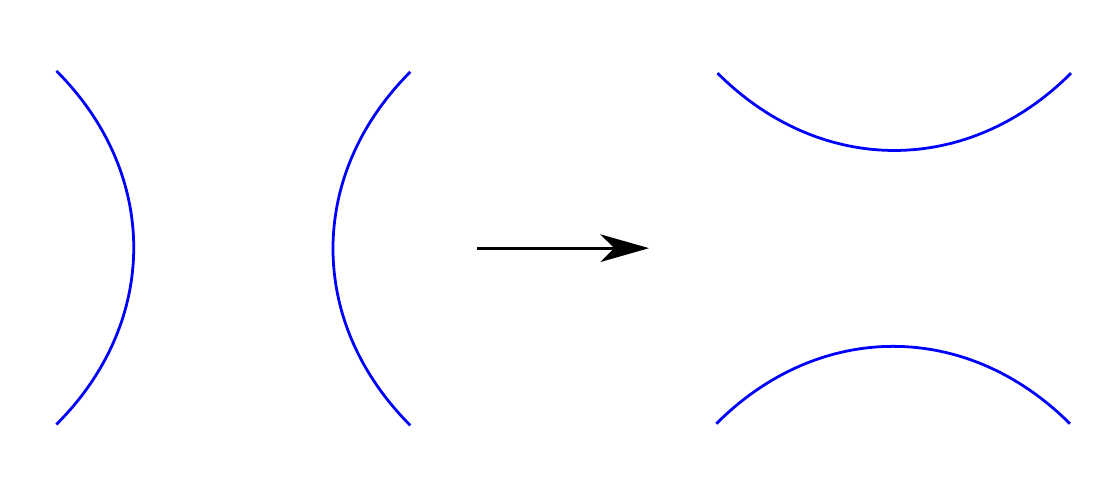}
  \caption{
    \label{fig:saddle}
    Local replacement for a saddle.
  }
\end{figure}

\begin{figure}
  \centering
  \includegraphics[scale=0.50]{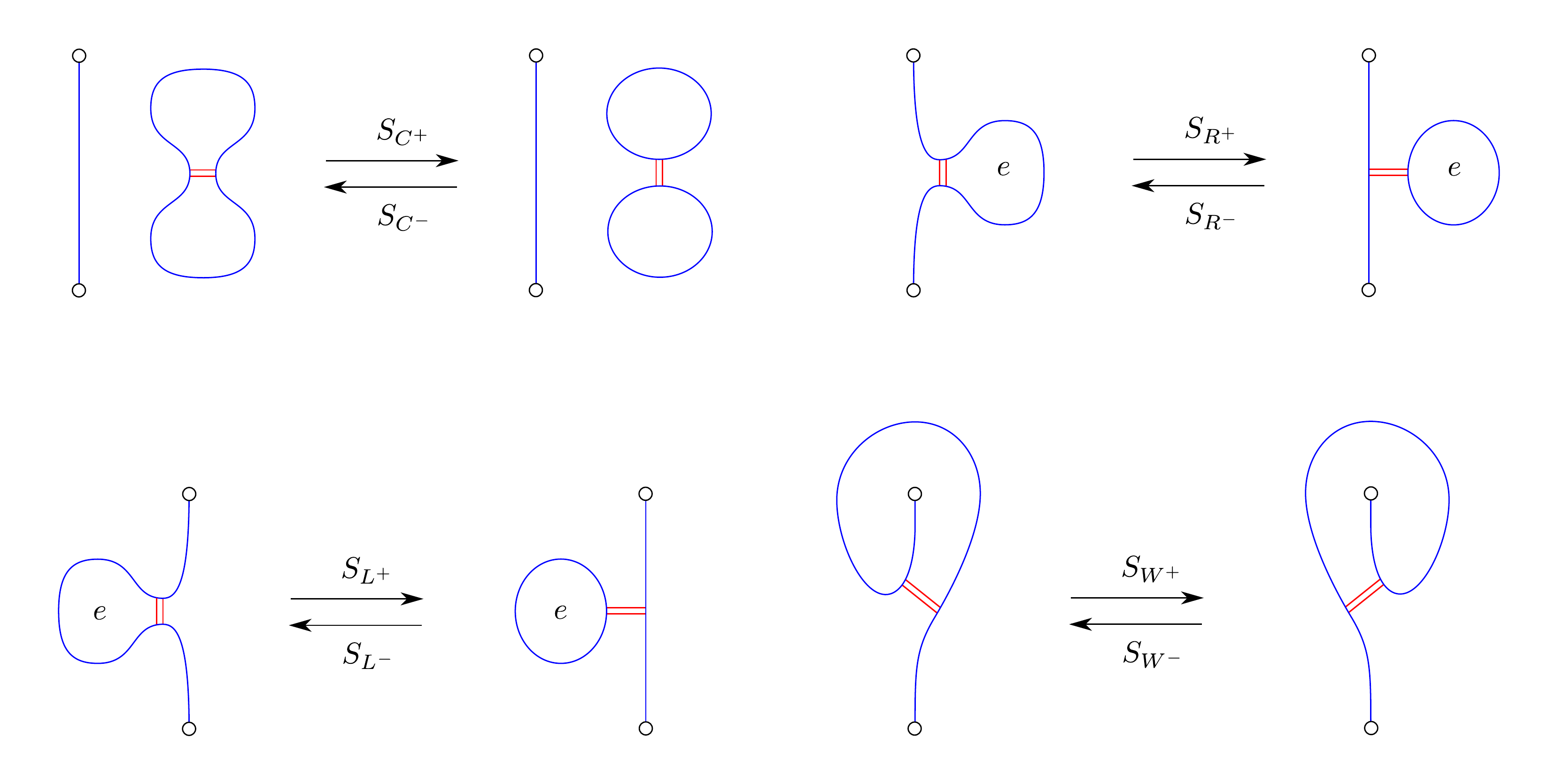}
  \caption{
    \label{fig:saddle-types}
    Types of saddles for planar 1-tangles in the annulus.
    Circle components in regions marked $e$ are said to be
    \emph{enclosed circles}.
  }
\end{figure}

\begin{enumerate}
\item
For a saddle
$S_{C^+}:P_n(r+1) \rightarrow P_n(r+2)$
that splits one circle into two circles:
\begin{align*}
  & d_{C^+}:L_n \otimes A^{\otimes(r+1)} \rightarrow
  L_n \otimes A^{\otimes(r+2)}, &
  & d_{C^+} = a_n \Delta I_r.
\end{align*}
For a saddle
$S_{C^-}:P_n(r+2) \rightarrow P_n(r+1)$
that merges two circles into one circle:
\begin{align*}
  & d_{C^-}:L_n \otimes A^{\otimes(r+2)} \rightarrow
  L_n \otimes A^{\otimes(r+1)}, &
  & d_{C^-} = a_n m I_r.
\end{align*}

\item
For a saddle
$S_{R^+}:P_n(r) \rightarrow P_n(r+1)$
that splits a circle from the right side of the arc component:
\begin{align*}
  & d_{R^+}:L_n \otimes A^{\otimes r} \rightarrow
  L_n \otimes A^{\otimes(r+1)}, &
  & d_{R^+} =
  (a_n \etadot + c_n \eta) I_r +
  c_n \etadot I_{r_n} \Sigma_{r_e}.
\end{align*}
For a saddle
$S_{R^-}:P_n(r+1) \rightarrow P_n(r)$
that merges a circle with the right side of the arc component:
\begin{align*}
  & d_{R^-}:L_n \otimes A^{\otimes(r+1)} \rightarrow
  L_n \otimes A^{\otimes r}, &
  & d_{R^-} =
  (a_n \epsilondot + c_n \epsilon)I_r +
  c_n\epsilondot I_{r_n} \Sigma_{r_e}.
\end{align*}
Here $r = r_e + r_n$, where $r_e$ is the number of enclosed
circles and $r_n$ is the number of non-enclosed circles.

\item
For a saddle
$S_{L^+}:P_n(r) \rightarrow P_n(r+1)$
that splits a circle from the left side of the arc component:
\begin{align*}
  & d_{L^+}:L_n \otimes A^{\otimes r} \rightarrow
  L_n \otimes A^{\otimes(r+1)}, &
  & d_{L^+} =
  a_n \etadot I_r +
  c_n\etadot I_{r_n} \Sigma_{r_e}.
\end{align*}
For a saddle
$S_{L^-}:P_n(r+1) \rightarrow P_n(r)$
that merges a circle with the left side of the arc component:
\begin{align*}
  & d_{L^-}:L_n \otimes A^{\otimes(r+1)} \rightarrow
  L_n \otimes A^{\otimes r}, &
  & d_{L^-} =
  a_n \epsilondot I_r +
  c_n\epsilondot I_{r_n} \Sigma_{r_e}.
\end{align*}
Here $r = r_e + r_n$, where $r_e$ is the number of enclosed circles
and $r_n$ is the number of non-enclosed circles.

\item
For a saddle
$S_{W^+}:P_n(r) \rightarrow P_{n+2}(r)$
that increases winding number of the arc component by two:
\begin{align*}
  & d_{W^+}:L_n \otimes A^{\otimes r} \rightarrow
  L_{n+2} \otimes A^{\otimes r}, &
  & d_{W^+} = p_{n+2,n} I_r.
\end{align*}
For a saddle
$S_{W^-}:P_n(r) \rightarrow P_{n-2}(r)$
that decreases the winding number of the arc component by two:
\begin{align*}
  & d_{W^-}:L_n \otimes A^{\otimes r} \rightarrow
  L_{n-2} \otimes A^{\otimes r}, &
  & d_{W^-} = q_{n-2,n} I_r.
\end{align*}

\end{enumerate}

The linear maps in the expressions for the differentials are as defined
in Section \ref{ssec:linear-maps}.
The expressions for the differentials depend on a specific choice of
ordering of the factors of $A$ corresponding to the circle components
of the planar tangles.
For a different ordering of the circle factors, we would need to
modify the expressions accordingly.
The differentials $d_{R^\pm}$ and $d_{L^\pm}$ corresponding
to saddles that split or merge a circle from the arc component contain
terms involving $\Sigma_{r_e}$ and are thus \emph{nonlocal}, in the
sense that they act nontrivially on the factors of $A$ corresponding
to enclosed circles.
As we will show in Section \ref{ssec:diff-cc-commute}, these terms are
needed to ensure that differentials corresponding to distinct saddles
commute.

Given a planar tangle $P_n(r+1)$ and a choice of a circle component
$C$, we define $P_n(r)$ to be the planar tangle obtained by removing
$C$ from $P_n(r+1)$.
We say that adding $C$ to $P_n(r)$ defines a \emph{cup}
$P_n(r) \rightarrow P_n(r+1)$ and removing $C$ from $P_n(r+1)$
defines a \emph{cap} $P_n(r+1) \rightarrow P_n(r)$.
To a cup $P_n(r) \rightarrow P_n(r+1)$ we assign a corresponding
\emph{cup map} with bigrading $(h,q) = (0,1)$:
\begin{align*}
  & \cup:L_n \otimes A^{\otimes r} \rightarrow
  L_n \otimes A^{\otimes(r+1)}, &
  & \cup = a_n \eta I_r.
\end{align*}
To a cap $P_n(r+1) \rightarrow P_n(r)$ we assign a corresponding
\emph{cap map} with bigrading $(h,q) = (0,1)$:
\begin{align*}
  & \cap:L_n \otimes A^{\otimes(r+1)} \rightarrow
  L_n \otimes A^{\otimes r}, &
  & \cap = a_n \epsilon I_r.
\end{align*}
We will use the same notation for cups and caps as for
their corresponding maps.

For simplicity, in what follows we will often denote
$\mu_{\Sigma \calL}^2(x,y)$ as $x \circ y$.

\subsection{Differentials, cup maps, and cap maps commute}
\label{ssec:diff-cc-commute}

Consider a planar tangle $P_{n_1}(r_1)$ and distinct saddles
$S_1:P_{n_1}(r_1) \rightarrow P_{n_2}(r_2)$ and
$S_2:P_{n_1}(r_1) \rightarrow P_{n_3}(r_3)$.
We can apply the local replacements for both $S_1$ and $S_2$ to
$P_{n_1}(r_1)$ to obtain a planar tangle $P_{n_4}(r_4)$.
We have induced saddles
$S_2':P_{n_2}(r_2) \rightarrow P_{n_4}(r_4)$ and
$S_1':P_{n_3}(r_3) \rightarrow P_{n_4}(r_4)$ obtained by applying
the local replacement for $S_2$ to $P_{n_2}(r_2)$ and the local
replacement for $S_1$ to $P_{n_3}(r_3)$:

\begin{eqnarray}
\label{diag:saddles}
\begin{tikzcd}
P_{n_1}(r_1) \arrow{r}{S_1} \arrow{d}{S_2} &
P_{n_2}(r_2) \arrow{d}{S_2'} \\
P_{n_3}(r_3) \arrow{r}{S_1'} &
P_{n_4}(r_4).
\end{tikzcd}
\end{eqnarray}

\begin{theorem}
\label{theorem:d-d-commute}
Assuming the $A_\infty$ operations of $\calL$ are as described in
Conjecture \ref{conj:fukaya}, the differentials corresponding to the
saddles in diagram (\ref{diag:saddles}) commute.
\end{theorem}

\begin{proof}
We prove this result by explicit computation for all possible pairs of
saddles.
The computations are straightforward, though somewhat tedious, so we
describe only a few representative examples.

Consider the pair of interleaved saddles shown in Figure
\ref{fig:saddles-interleaved}.
We let $a$, $b$, and $c$ denote the number of circle components in the
regions labeled $a$, $b$, and $c$ in Figure
\ref{fig:saddles-interleaved}.
The diagram corresponding to Figure \ref{fig:saddles-interleaved} is
\begin{eqnarray}
\label{eqn:diag-interleaved}
\begin{tikzcd}
  L_n \otimes A^{\otimes(a+b+c)}
  \arrow{r}{d_1}
  \arrow{d}{d_2} &
  L_n \otimes A^{\otimes(a+b+c+1)} \arrow{d}{d_2'} \\
  L_n \otimes A^{\otimes(a+b+c+1)} \arrow{r}{d_1'} &
  L_n \otimes A^{\otimes(a+b+c)},
\end{tikzcd}
\end{eqnarray}
where
\begin{align*}
  & d_1 = d_{L^+} =
  a_n\etadot I_{a+b+c} +
  c_n\etadot I_a \Sigma_b I_c, &
  & d_2' = d_{L^-} =
  a_n\epsilondot I_{a+b+c} +
  c_n\epsilondot I_a \Sigma_b I_c, \\
  & d_2 = d_{R^+} =
  (a_n \etadot + c_n \eta) I_{a+b+c} +
  c_n \etadot \Sigma_a I_{b+c}, &
  & d_1' = d_{R^-} =
  (a_n \epsilondot + c_n \epsilon) I_{a+b+c} +
  c_n \etadot \Sigma_a I_{b+c}.
\end{align*}
A calculation shows that
\begin{align*}
  d_2' \circ d_1 = d_1' \circ d_2 = 0,
\end{align*}
so diagram (\ref{eqn:diag-interleaved}) commutes.

\begin{figure}
  \centering
  \includegraphics[scale=0.60]{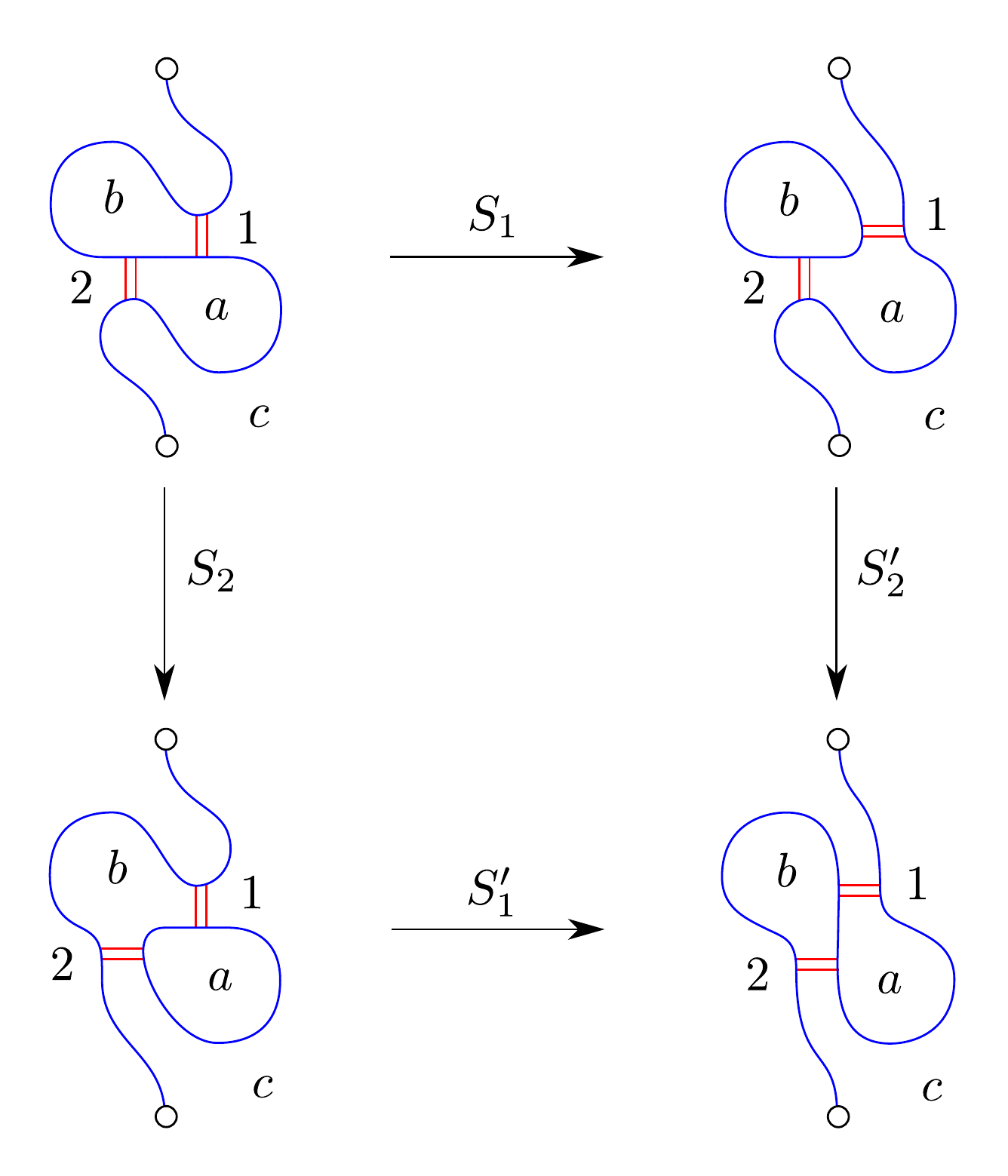}
  \caption{
    \label{fig:saddles-interleaved}
    Interleaved saddles.
  }
\end{figure}

Consider the pair of nested saddles shown in Figure
\ref{fig:saddles-nested}.
We let $a$, $b$, and $c$ denote the number of circle components in the
regions labeled $a$, $b$, and $c$ in Figure
\ref{fig:saddles-nested}.
The diagram corresponding to Figure \ref{fig:saddles-nested} is

\begin{eqnarray}
\label{eqn:diag-nested}
\begin{tikzcd}
L_n \otimes A^{\otimes(a+b+c)}
\arrow{r}{d_1} 
\arrow{d}{d_2} &
L_n \otimes A^{\otimes(a+b+c+1)}
\arrow{d}{d_2'} \\
L_n \otimes A^{\otimes(a+b+c+1)}
\arrow{r}{d_1'} &
L_n \otimes A^{\otimes(a+b+c+2)},
\end{tikzcd}
\end{eqnarray}
where
\begin{align*}
  & d_1 = d_{L^+} =
  a_n \etadot I_{a+b+c} +
  c_n\etadot I_a \Sigma_b I_c, &
  & d_2' = d_{R^+} =
  (a_n \etadot + c_n \eta)I_{a+b+c+1} + c_n \etadot
  \Sigma_{a+b+1}I_c, \\
  & d_2 = d_{R^+} =
  (a_n \etadot + c_n \eta)I_{a+b+c} +
  c_n\etadot\Sigma_a I_b, &
  & d_1' = d_{C^+} =
  a_n\Delta I_{a+b+c}.
\end{align*}
A calculation shows that
\begin{align*}
  d_2' \circ d_1 = d_1' \circ d_2 =
  (a_n\etadot\etadot +
  c_n (\etadot\eta + \eta\etadot))I_{a+b+c} +
  c_n\etadot\etadot \Sigma_a I_{b+c},
\end{align*}
where we have used the identity
\begin{align*}
  \Sigma_{a+b+1} \circ \etadot I_{a+b} =
  (\Sigma_1 I_{a+b} + I_1 \Sigma_{a+b}) \circ \etadot I_{a+b} =
  \eta I_{a+b} + \etadot \Sigma_{a+b}.
\end{align*}
So diagram (\ref{eqn:diag-nested}) commutes.
We note that without the nonlocal terms involving $\Sigma_r$, diagram
(\ref{eqn:diag-nested}) would not commute.

\begin{figure}
  \centering
  \includegraphics[scale=0.60]{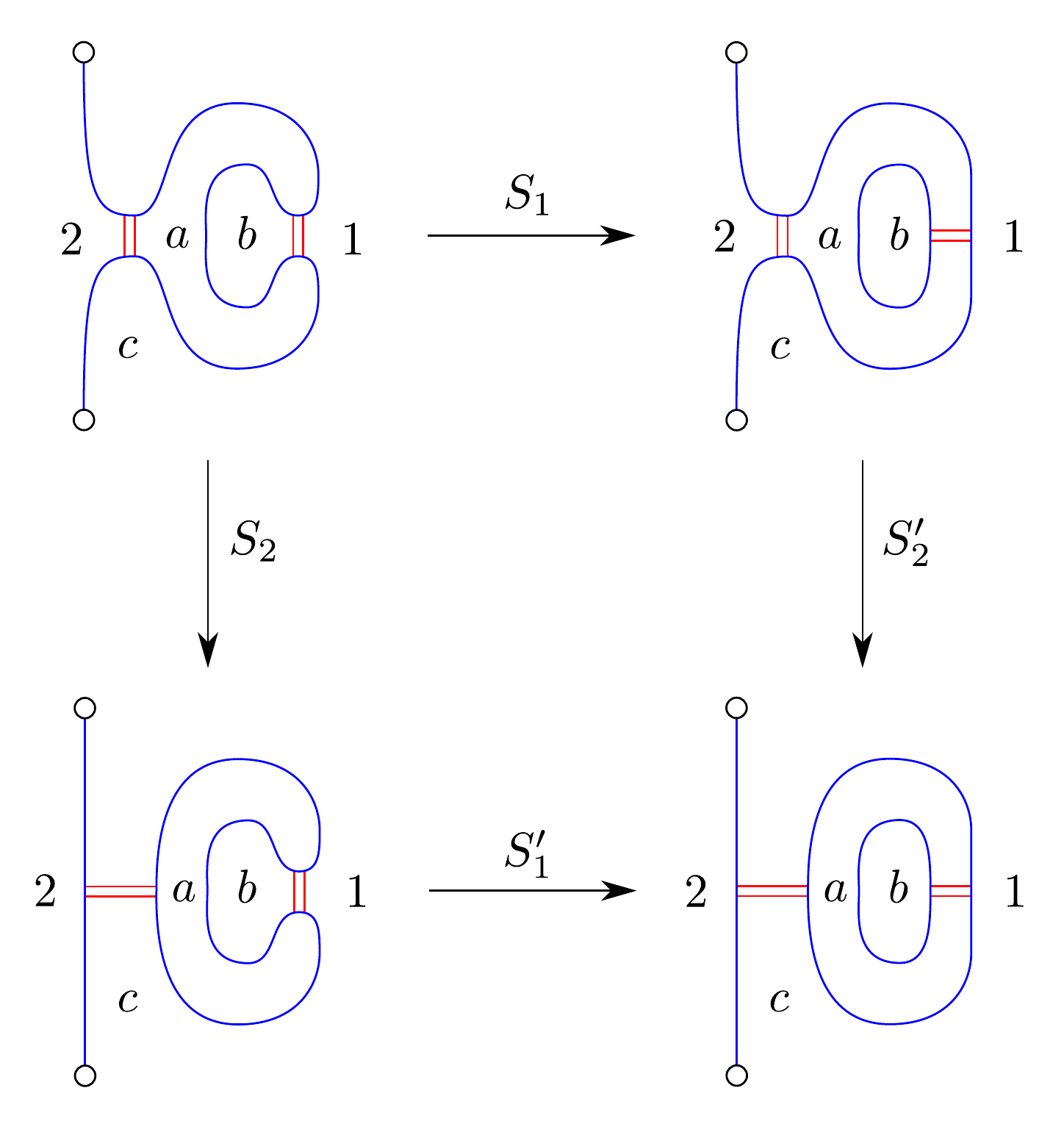}
  \caption{
    \label{fig:saddles-nested}
    Nested saddles.
  }
\end{figure}

The fact that differentials of the form $d_{C^\pm}$ commute with
differentials of the form $d_{R^\pm}$ and $d_{L^\pm}$ is due to the
identities in equation (\ref{eqn:identities-sigma}).

Similar calculations for the remaining pairs of saddles prove the
result.
\end{proof}

Consider a saddle
$S_{21}:P_{n_1}(r_1 + 1) \rightarrow P_{n_2}(r_2 + 1)$.
Choose a circle component component $C_1$ in $P_{n_1}(r_1 + 1)$ that
is unchanged under $S_{21}$, so there is a corresponding circle
component $C_2$ in $P_{n_2}(r_2 + 1)$.
Remove the circle components $C_1$ and $C_2$ from
$P_{n_1}(r_1 + 1)$ and
$P_{n_2}(r_2 + 1)$ to obtain planar tangles $P_{n_1}(r_1)$ and
$P_{n_2}(r_2)$, and
apply the local replacement for $S_{21}$ to $P_{n_1}(r_1)$ to
obtain an induced saddle
$S_{21}':P_{n_1}(r_1) \rightarrow P_{n_2}(r_2)$.
We have cups $\cup_k:P_{n_k}(r_k) \rightarrow P_{n_k}(r_k + 1)$ and
caps $\cap_k:P_{n_k}(r_k + 1) \rightarrow P_{n_k}(r_k)$ that add and
remove the circle components $C_k$.
We have the following diagram:

\begin{eqnarray}
\label{diag:cc}
\begin{tikzcd}
P_{n_1}(r_1 + 1)
\arrow{r}{S_{21}}
\arrow[shift left=2]{d}{\cap_1} &
P_{n_2}(r_2 + 1)
\arrow[shift left=2]{d}{\cap_2} \\
P_{n_1}(r_1) \arrow{r}{S_{21}'}
\arrow[shift left=2]{u}{\cup_1} &
P_{n_2}(r_2).
\arrow[shift left=2]{u}{\cup_2}
\end{tikzcd}
\end{eqnarray}

\begin{theorem}
\label{theorem:d-cc-commute}
Assuming the $A_\infty$ operations for $\calL$ are as described in
Conjecture \ref{conj:fukaya}, the differentials corresponding to the
saddles in diagram (\ref{diag:cc}) commute with the maps corresponding
to the cups and caps.
\end{theorem}

\begin{proof}
If the circles $C_1$ and $C_2$ are not enclosed circles, then the
differentials corresponding to $S_{21}$ and $S_{21}'$ act as the
identity on the factor of $A$ corresponding to $C_1$, so the result is
clear.
If $C_1$ and $C_2$ are enclosed circles, the result follows from the
identities
\begin{align*}
  &\Sigma_{r + 1} \circ \eta I_r =
  (\Sigma_1 I_r + I_1 \Sigma_r) \circ \eta I_r =
  \eta \Sigma_r = \eta I_r \circ \Sigma_r, \\
  &\epsilon I_r \circ \Sigma_{r + 1} =
  \epsilon I_r \circ (\Sigma_1 I_r + I_1 \Sigma_r) =
  \epsilon \Sigma_r = \Sigma_r \circ \epsilon I_r,
\end{align*}
where we have used the fact that
$\Sigma_1 \circ \eta = \epsilon \circ \Sigma_1 = 0$.
\end{proof}

\subsection{Construction of the twisted complex}
\label{ssec:twisted-complex}

Given an oriented 1-tangle diagram $T$, we will construct a
corresponding object $(X,\delta)$ of $\Tw \calL$.
Recall that $(X,\delta)$ consists of an object $X$ of $\Sigma \calL$
and a differential $\delta:X \rightarrow X$ with bigrading
$(h,q) = (1,1)$.

We construct the object $X$ of $\Sigma \calL$ from a cube of
resolutions of the tangle diagram $T$.
Let $m_+$ and $m_-$ denote the number of positive an negative
crossings of $T$, and let $m = m_+ + m_-$ denote the total number of
crossings.
Fix an arbitrary ordering of the crossings.
Define the $0$-resolution, respectively 1-resolution, of a crossing
such that the overpass turns left, respectively right.
We can specify a planar resolution of $T$ in terms of a binary
string $i$ of length $m$, where the $k$-th bit of $i$ tells us how to
resolve the $k$-th crossing of $T$.
Define $I = \{0,1\}^m$ to be the set of binary strings of
length $m$.
For each binary string $i \in I$, define $T_i$ to be the planar
resolution of $T$ specified by $i$.
We have $T_i = P_{n_i}(c_i)$, where $n_i$ is the winding number of
$T_i$ and $c_i$ is the number of circle components of $T_i$.
Define the \emph{resolution degree} $r(i) \in \{0, 1, \cdots, m\}$ to
be the number of 1's in the binary string $i \in I$.
We define
\begin{align}
  \label{eqn:obj-X}
  X = \bigoplus_{i \in  I}
  (L_{n_i} \otimes A^{\otimes c_i}[r(i)+h_T, 2r(i)+q_T]),
\end{align}
where $[h_T,q_T]$ is a bigrading shift due to the oriented crossings
of $T$ that is given by
\begin{align*}
  [h_T,q_T] = [-m_-, m_+ - 3m_-].
\end{align*}
The quantum grading shift $q_T = m_+ - 3m_-$ differs from the usual
shift of $m_+ - 2m_-$ for Khovanov homology, since the quantum grading
of $\delta$ is 1 in our convention but 0 in the usual convention.

We construct the differential $\delta:X \rightarrow X$ by summing over
maps corresponding to saddles that relate different resolutions of the
tangle diagram $T$.
Consider two binary strings $i,j \in I$ that are identical except for
a single bit that is 0 for $i$ and 1 for $j$, so $r(j) = r(i) + 1$.
The corresponding planar tangles $T_i$ and $T_{j}$ are related by a
saddle $S_{ji}:T_i \rightarrow T_j$ that changes a 0-resolution in
$T_i$ to a 1-resolution in $T_j$.
We define
\begin{align*}
    d_{ji}:
    L_{n_i} \otimes A^{\otimes c_i}[r(i)+h_T, 2r(i)+q_T] \rightarrow
    L_{n_j} \otimes A^{\otimes c_j}[r(j)+h_T, 2r(j)+q_T]
\end{align*}
to be the differential corresponding to the saddle $S_{ji}$ shifted
in bigrading by $[1,2]$, so $d_{ji}$ has bigrading
$(h,q)=(1,1)$.
We define $\delta$ to be the sum of the differentials
$d_{ji}$ over all pairs of binary strings $i,j \in I$ that are
identical except for a single bit that is 0 for $i$ and 1 for $j$.

Recall from Section \ref{sec:conj-fukaya} that we considered two
different ways of defining the elements $p_{n+2,n}$ and $q_{n-2,n}$
that appear in the differentials $d_{W_+}$ and $d_{W^-}$.
The two possible definitions of $p_{n+2,n}$ and $q_{n-2,n}$ correspond
to two possible definitions of $\delta$, which we denote $\delta_+$
and $\delta_-$.
For $\delta_+$, we define $p_{n+2,n}$ and $q_{n-2,n}$ as
\begin{align*}
  p_{n+2,n}^{(-)} &= r_{n+2,n}^{(-)} + \rbar_{n+2,n}^{(-)}, &
  q_{n-2,n}^{(+)} &= s_{n-2,n}^{(+)} + \sbar_{n-2,n}^{(+)}.
\end{align*}
For $\delta_-$, we define $p_{n+2,n}$ and $q_{n-2,n}$ as
\begin{align*}
  p_{n+2,n}^{(+)} &= s_{n+2,n}^{(+)} + \sbar_{n+2,n}^{(+)}, &
  q_{n-2,n}^{(-)} &= r_{n-2,n}^{(-)} + \rbar_{n-2,n}^{(-)}.
\end{align*}
For some purposes it does not matter which definition is used, in
which case we use the notation $\delta$ with no subscript.

\begin{theorem}
Assuming the $A_\infty$ operations for $\calL$ are as described in
Conjecture \ref{conj:fukaya}, the pair $(X,\delta)$ is a twisted
complex.
\end{theorem}

\begin{proof}
We need to show that
\begin{align}
  \label{eqn:sum-mu}
  \sum_{m=1}^\infty \mu_{\Sigma \calL}^m(\delta, \cdots, \delta) = 0
\end{align}
and that $X$ admits a filtration with respect to which
$\delta$ is strictly lower triangular.
Conjecture \ref{conj:fukaya} states that the $m \neq 2$ terms of
equation (\ref{eqn:sum-mu}) are zero.
Assuming Conjecture \ref{conj:fukaya}, Theorem
\ref{theorem:d-d-commute} states that the $m=2$ term of
equation (\ref{eqn:sum-mu}) is zero.
We will take the filtration to be the one given by the resolution
degree.
By construction, the differential $\delta$ increases the resolution
degree by 1, so it is lower triangular with respect to this
filtration.
\end{proof}

\subsection{Example twisted complex}
\label{ssec:example-twisted-complex}

Consider the tangle diagram $T$ shown in Figure
\ref{fig:example-tangle}.
There are $m_+ = 2$ positive crossings and $m_- = 0$
negative crossings, corresponding to a bigrading shift
\begin{align*}
  [h_T,q_T] = [-m_-, m_+ - 3m_-] = [0,2].
\end{align*}
The tangle diagram $T$ thus corresponds to the
following object of the twisted Fukaya category:
\begin{eqnarray*}
\begin{tikzcd}
  {} & L_0 \otimes \F[1,4]
  \arrow{dr}{d_{L^+}[1,2]} & {} \\
  L_2 \otimes \F[0,2]
  \arrow{ur}{d_{W^-}[1,2]} \arrow{dr}[swap]{d_{W^-}[1,2]} &
  {} & L_0 \otimes A[2,6], \\
  {} & L_0 \otimes \F[1,4] \arrow{ur}[swap]{d_{R^+}[1,2]} & {}
\end{tikzcd}
\end{eqnarray*}
where the differentials are given by
\begin{align*}
  d_{W^-} &= q_{0,2} \otimes \id_{\F}, &
  d_{L^+} &= a_0 \otimes \etadot, &
  d_{R^+} &= a_0 \otimes \etadot + c_0 \otimes \eta.
\end{align*}

\begin{figure}
  \centering
  \includegraphics[scale=0.5]{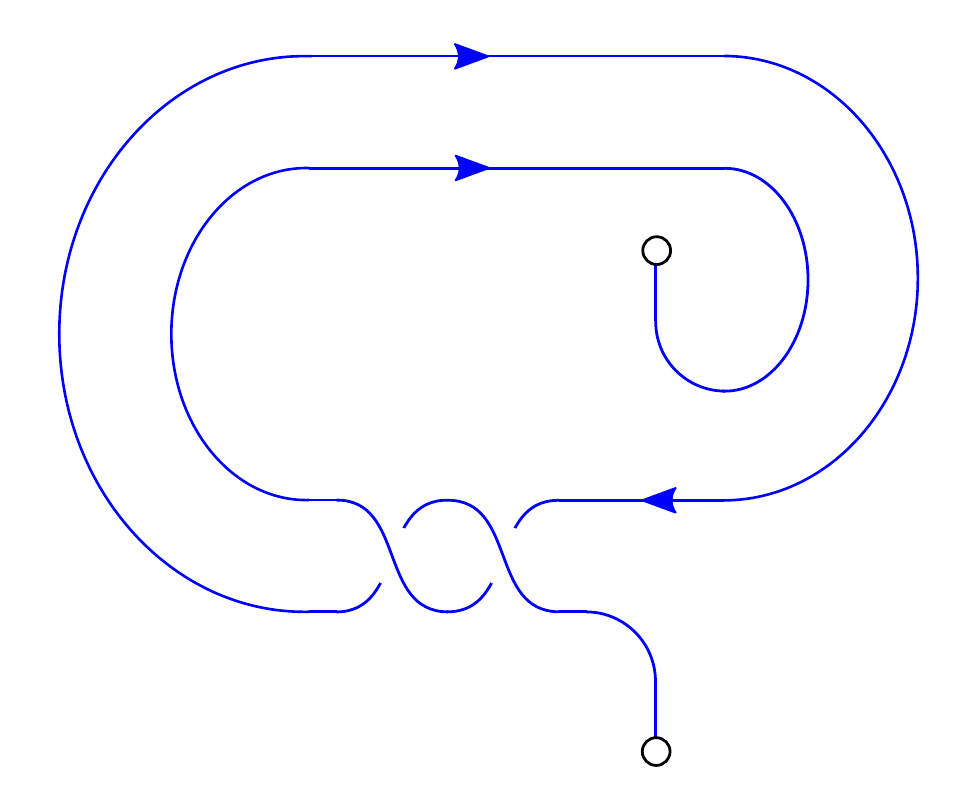}
  \caption{
    \label{fig:example-tangle}
    Example tangle diagram $T$.
  }
\end{figure}

\section{Invariance under Reidemeister moves}
\label{sec:invariance}

Consider an $A_\infty$ category $\calA$.
Let $\Tw \calA$ denote the category of twisted complexes over $\calA$.
We say that two twisted complexes $(X,\delta_X)$ and $(Y,\delta_Y)$ of
are \emph{homotopy equivalent} if there are morphisms
$F:(X,\delta_X) \rightarrow (Y,\delta_Y)$,
$G:(Y,\delta_Y) \rightarrow (X,\delta_X)$,
$K:(X,\delta_X) \rightarrow (X,\delta_X)$, and
$H:(Y,\delta_Y) \rightarrow (Y,\delta_Y)$ such that
\begin{align*}
  &\mu_{\Tw \calA}^1(F) = 0, &
  &\mu_{\Tw \calA}^1(G) = 0, \\
  &\mu_{\Tw \calA}^2(G,F) = \id_{(X,\delta_X)} + \mu_{\Tw A}^1(K), &
  &\mu_{\Tw \calA}^2(F,G) = \id_{(Y,\delta_Y)} + \mu_{\Tw A}^1(H).
\end{align*}
In this case we say that $F$ and $G$ are \emph{homotopy equivalences}
and $H$ and $K$ are \emph{homotopies}.
Our goal in this section is to prove:

\begin{theorem}
\label{theorem:invar}
Assuming the $A_\infty$ operations of $\calL$ are as described in
Conjecture \ref{conj:fukaya}, if two oriented 1-tangle
diagrams $T_1$ and $T_2$ are isotopic rel boundary, then the
corresponding twisted complexes $(X_1,\delta_1)$ and $(X_2,\delta_2)$
are homotopy equivalent.
\end{theorem}

According to Conjecture \ref{conj:fukaya}, the $A_\infty$ operations
$\mu_\calL^m$ are zero for $m \neq 2$.
In this situation the conditions for morphisms $F$ and $G$ to be
homotopy equivalences reduce to
\begin{align*}
  &\delta_Y \circ F = F \circ \delta_X, &
  &\delta_X \circ G = G \circ \delta_Y, \\
  &G \circ F = \id_{(X,\delta_X)} +
  \delta_X \circ K + K \circ \delta_X, &
  &F \circ G = \id_{(Y,\delta_Y)} +
  \delta_Y \circ H + H \circ \delta_Y,
\end{align*}
where $\circ$ denotes the product operation $\mu_{\Sigma\calL}^2$ in
$\Sigma \calL$.
So if Conjecture \ref{conj:fukaya} is correct, the notion of homotopy
equivalence in $\Tw \calL$ is formally identical to the notion of
homotopy equivalence in the category of cochain complexes, only we
interpret $\circ$ as the product operation in $\Sigma \calL$ rather
than as the composition of linear maps.

Using this observation, we can prove Theorem
\ref{theorem:invar} by adapting Bar-Natan's proof of the isotopy
invariance of Khovanov homology in \cite{BarNatan}.
Proving the isotopy invariance of Khovanov homology amounts to showing
that the cochain complexes of link diagrams related by Reidemeister
moves are homotopy equivalent.
Bar-Natan defines the relevant homotopy equivalences in terms of
certain elementary cobordisms in a cobordism category for planar
tangles.
The invariance proof then reduces to showing that these elementary
cobordisms satisfy a small number of relations.
We will show that analogous relations hold for the differentials, cup
maps, and cap maps that we defined in Section \ref{ssec:saddles}.
Using these relations, we will define homotopy equivalences for
Reidemeister moves that are formally identical to those defined by
Bar-Natan.

\subsection{Bar-Natan relations}
\label{ssec:bn-relations}

We prove here the relations we will need involving differentials,
cup maps, and cap maps.
We name the relations by analogy with corresponding relations in
Bar-Natan's cobordism category.

Suppose we create a circle component $C$ with a cup
$\cup:P_n(r) \rightarrow P_n(r+1)$ and then remove $C$ with a cap
$\cap:P_n(r+1) \rightarrow P_n(r)$.
Let $\cup$ and $\cap$ also denote the corresponding cup and cap maps.
We have:

\begin{theorem}
\label{theorem:rel-0}
We have the following sphere relation in $\Sigma \calL$:
\begin{align*}
  \cap \circ \cup = 0.
\end{align*}
\end{theorem}

\begin{proof}
We have
\begin{align*}
  \cap \circ \cup = a_n \epsilon I_r \circ a_n \eta I_r =
  a_n (\epsilon \circ \eta) I_r = 0,
\end{align*}
since $\epsilon \circ \eta = 0$.
\end{proof}

Take the left diagram in in Figure \ref{fig:move-r1} to describe a
portion of a planar tangle $P_n(r)$, where the arc $A$ could either
belong to the arc component of $P_n(r)$ or a circle component of
$P_n(r)$.
Define saddles $S_{A^+}:P_n(r) \rightarrow P_n(r+1)$ and
$S_{A^-}:P_n(r+1) \rightarrow P_n(r)$ that split the indicated circle
component $C$ from and merge $C$ with the arc $A$.
Let $d_{A^\pm}$ denote the differential corresponding to the saddle
$S_{A^\pm}$.
Define a cup $\cup:P_n(r) \rightarrow P_n(r+1)$ that creates $C$ and a
cap $\cap:P_n(r+1) \rightarrow P_n(r)$ that removes $C$.
We will also use $\cup$ and $\cap$ to denote the corresponding cup and
cap maps.
We have:

\begin{theorem}
\label{theorem:rel-1}
We have the following isotopy invariance relations in $\Sigma \calL$:
\begin{align*}
  d_{A^-} \circ \cup = \cap \circ d_{A^+} = a_n I_r.
\end{align*}
We have the following neck cutting relations in $\Sigma \calL$:
\begin{align*}
  d_{A^+} \circ \cap + \cup \circ d_{A^-} &= a_n I_{r+1}, &
  d_{A^-} \circ d_{A^+} &= 0.
\end{align*}
\end{theorem}

\begin{proof}
The cup and cap maps are
\begin{align*}
  \cup &= a_n \eta I_r, &
  \cap &= a_n \epsilon I_r.
\end{align*}
We have the following cases:
\begin{enumerate}
\item
The arc $A$ belongs to a circle component of $P_n(r)$.
Then
\begin{align*}
  d_{A^+} &= d_{C^+} = a_n \Delta I_{r-1}, &
  d_{A^-} &= d_{C^-} = a_n m I_{r-1}.
\end{align*}

\item
The arc $A$ belongs to the arc component of $P_n(r)$ and the
saddles $S_{A^+}$ and $S_{A^-}$ split $C$ from and merge $C$ with the
right side of the arc component.
Then
\begin{align*}
  d_{A^+} &= d_{R^+} = (a_n \etadot + c_n \eta)I_r, &
  d_{A^-} &= d_{R^-} = (a_n \epsilondot + c_n \epsilon) I_r.
\end{align*}

\item
The arc $A$ belongs to the arc component of $P_n(r)$ and the
saddles $S_{A^+}$ and $S_{A^-}$ split $C$ from and merge $C$ with the
left side of the arc component.
Then
\begin{align*}
  d_{A^+} &= d_{L^+} = a_n \etadot I_r, &
  d_{A^-} &= d_{L^-} = a_n \epsilondot I_r.
\end{align*}
\end{enumerate}
The result now follows from straightforward calculations.
\end{proof}

Take the leftmost diagram in Figure \ref{fig:move-r2} to describe a
portion of a planar tangle $P_n(r)$, where the arcs $A$ and $B$ could
belong to either the arc component of $P_n(r)$ or to circle components
of $P_n(r)$.
Define saddles $S_{A^+}:P_n(r) \rightarrow P_n(r+1)$ and
$S_{B^-}:P_n(r+1) \rightarrow P_n(r)$ that spit the indicated circle
component $C$ from arc $A$ and merge $C$ with arc $B$.
Define saddles $S_{=\parallel}:P_n(r) \rightarrow P_{m}(s)$ and
$S_{\parallel =}:P_{m}(s) \rightarrow P_n(r)$ as shown in Figure
\ref{fig:move-r2}.
Let $d_{A^+}$, $d_{B^-}$, $d_{=\parallel}$, and $d_{\parallel =}$
denote the differentials corresponding to the saddles $S_{A^+}$,
$S_{B^-}$, $S_{=\parallel}$, and $S_{\parallel =}$.
Define a cup $\cup:P_n(r) \rightarrow P_n(r+1)$ that creates $C$ and a
cap $\cap:P_n(r+1) \rightarrow P_n(r)$ that removes $C$.
We will also use $\cup$ and $\cap$ to denote the corresponding maps.
We have:

\begin{theorem}
\label{theorem:rel-2}
We have the following isotopy invariance relation in $\Sigma \calL$:
\begin{align*}
  d_{B^-} \circ d_{A^+} = d_{\parallel =} \circ d_{=\parallel}.
\end{align*}
We have the following 4-tube relation in $\Sigma \calL$:
\begin{align*}
  d_{A^+} \circ \cap +
  \cup \circ (d_{\parallel =} \circ d_{=\parallel}) \circ \cap +
  \cup \circ d_{B^-} = a_n I_{r+1}.
\end{align*}
\end{theorem}

\begin{proof}
We have the following cases:
\begin{enumerate}
\item
Arcs $A$ and $B$ both belong to the arc component of $P_n(r)$,
saddle $S_{A^+}$ splits $C$ from the right side of the arc, and
saddle $S_{B^-}$ merges $C$ with the right side of the arc.
We have
\begin{align*}
  \cup &= a_0 \eta I_r, &
  \cap &= a_0 \epsilon I_r, \\
  d_{A^+} &= d_{L^+} = a_n \etadot I_r, &
  d_{B^-} &= d_{L^-} = a_n \epsilondot I_r, \\
  d_{=\parallel} &= d_{R^+} =
  (a_n \etadot + c_n \eta)I_r +
  c_n \etadot I_{r_n} \Sigma_{r_e}, &
  d_{\parallel =} &= d_{R^-} =
  (a_n \epsilondot + c_n \epsilon)I_r +
  c_n \epsilondot I_{r_n} \Sigma_{r_e}.
\end{align*}

\item
Arcs $A$ and $B$ both belong to the arc component of $P_n(r)$,
saddle $S_{A^+}$ splits $C$ from the left side of the arc, and
saddle $S_{B^-}$ merges $C$ with the left side of the arc.
We have
\begin{align*}
  \cup &= a_0 \eta I_r, &
  \cap &= a_0 \epsilon I_r, \\
  d_{A^+} &= d_{R^+} = (a_n \etadot + c_n \eta) I_r, &
  d_{B^-} &= d_{R^-} = (a_n \epsilondot + c_n \epsilon) I_r, \\
  d_{=\parallel} &= d_{L^+} =
  a_n \etadot I_r + c_n \etadot I_{r_n} \Sigma_{r_e}, &
  d_{\parallel =} &= d_{L^-} =
  a_n \epsilondot I_r + c_n \epsilondot I_{r_n} \Sigma_{r_e}.
\end{align*}

\item
Arcs $A$ and $B$ both belong to the arc component of $P_n(r)$,
saddle $S_{A^+}$ splits $C$ from the right side of the arc, and
saddle $S_{B^-}$ merges $C$ with the left side of the arc.
We have
\begin{align*}
  &\cup = a_0 \eta I_r, &
  &\cap = a_0 \epsilon I_r, \\
  &d_{A^+} = d_{R^+} = (a_n \etadot + c_n \eta)I_r, &
  &d_{B^-} = d_{L^-} = a_n \epsilondot I_r, \\
  &d_{=\parallel} = d_{W^+} = p_{n,n+2} I_r, &
  &d_{\parallel =} = d_{W^-} = q_{n,n+2} I_r.
\end{align*}

\item
Arcs $A$ and $B$ both belong to the arc component of $P_n(r)$,
saddle $S_{A^+}$ splits $C$ from the left side of the arc, and
saddle $S_{B^-}$ merges $C$ with the right side of the arc.
\begin{align*}
  &\cup = a_0 \eta I_r, &
  &\cap = a_0 \epsilon I_r, \\
  &d_{A^+} = d_{L^+} = a_n \etadot I_r, &
  &d_{B^-} = d_{R^-} = (a_n \epsilondot + c_n \epsilon) I_r, \\
  &d_{=\parallel} = d_{W^-} = q_{n-2,n} I_r, &
  &d_{\parallel =} = d_{W^+} = p_{n,n-2} I_r.
\end{align*}

\item
Arc $A$ belongs to a circle component of $P_n(r)$, arc $B$ belongs to
the arc component, and saddle $S_{B^-}$ merges $C$ with the right side
of the arc component.
We have
\begin{align*}
  &\cup = a_0 \id \eta I_r, &
  &\cap = a_0 \id \epsilon I_r, \\
  &d_{A^+} = d_{C^+} = a_n \Delta I_r, &
  &d_{B^-} = d_{R^-} = (a_n \id \epsilondot + c_n \id \epsilon) I_r, \\
  &d_{=\parallel} = d_{R^-} =
  (a_n \epsilondot + c_n \epsilon)I_r +
  c_n \epsilondot I_{r_n} \Sigma_{r_e}, &
  &d_{\parallel =} = d_{R^+} =
  (a_n \etadot + c_n \eta) I_r +
  c_n \etadot I_{r_n} \Sigma_{r_e}.
\end{align*}

\item
Arc $A$ belongs to a circle component of $P_n(r)$, arc $B$ belongs to
the arc component, and saddle $S_{B^-}$ merges $C$ with the left side
of the arc component.
We have
\begin{align*}
  &\cup = a_0 \id \eta I_r, &
  &\cap = a_0 \id \epsilon I_r, \\
  &d_{A^+} = d_{C^+} = a_n \Delta I_r, &
  &d_{B^-} = d_{L^-} = a_n \id\epsilondot I_r, \\
  &d_{=\parallel} = d_{L^-} =
  a_n \epsilondot I_r + c_n \epsilondot I_{r_n} \Sigma_{r_e}, &
  &d_{\parallel =} = d_{L^+} =
  a_n \etadot I_r + c_n \etadot I_{r_n} \Sigma_{r_e}.
\end{align*}

\item
Arcs $A$ and $B$ belong to distinct circle components of $P_n(r)$.
We have
\begin{align*}
  & \cup = a_n \id \eta \id I_{r-2}, &
  & \cap = a_n \id \epsilon \id I_{r-2}, \\
  & d_{A^+} = d_{C^+} = a_n \Delta \id I_{r-2}, &
  & d_{B^-} = d_{C^-} = a_n \id m I_{r-2}, \\
  & d_{=\parallel} = d_{C^-} = a_n m I_{r-2}, &
  & d_{\parallel =} = d_{C^+} = a_n \Delta I_{r-2}.
\end{align*}

\item
Arcs $A$ and $B$ belong to the same circle component of $P_n(r)$.
We have
\begin{align*}
  & \cup = \id \eta I_{r-1}, &
  & \cap = \id \epsilon I_{r-1}, \\
  & d_{A^+} = d_{C^+} = a_n \Delta I_{r-1}, &
  & d_{B^-} = d_{C^-} = a_n m I_{r-1}, \\
  & d_{=\parallel} = d_{C^-} = m I_{r-1}, &
  & d_{\parallel =} = d_{C^+} = \Delta I_{r-1}.
\end{align*}

\end{enumerate}

The now result follows from straightforward calculations.
\end{proof}

In \cite{Hedden-3}, twisted complexes over the Fukaya category of the
pillowcase are constructed from 2-tangle diagrams in the disk via an
$A_\infty$ functor whose source is the $A_\infty$ category of twisted
complexes over Bar-Natan's cobordism category.
The relations described in Theorems \ref{theorem:rel-0},
\ref{theorem:rel-1}, and \ref{theorem:rel-2} suggest that it may be
possible to construct an analogous cobordism category and $A_\infty$
functor that would yield the twisted complex $(X,\delta)$ described in
Section \ref{ssec:twisted-complex}.
It would be interesting to attempt to construct such a cobordism
category.

We will now show invariance under the Reidemeister moves shown in
Figure \ref{fig:reidemeister}, under the assumption that the
$A_\infty$ relations for $\calL$ are as described in Conjecture
\ref{conj:fukaya}.

\begin{figure}
  \centering
  \includegraphics[scale=0.45]{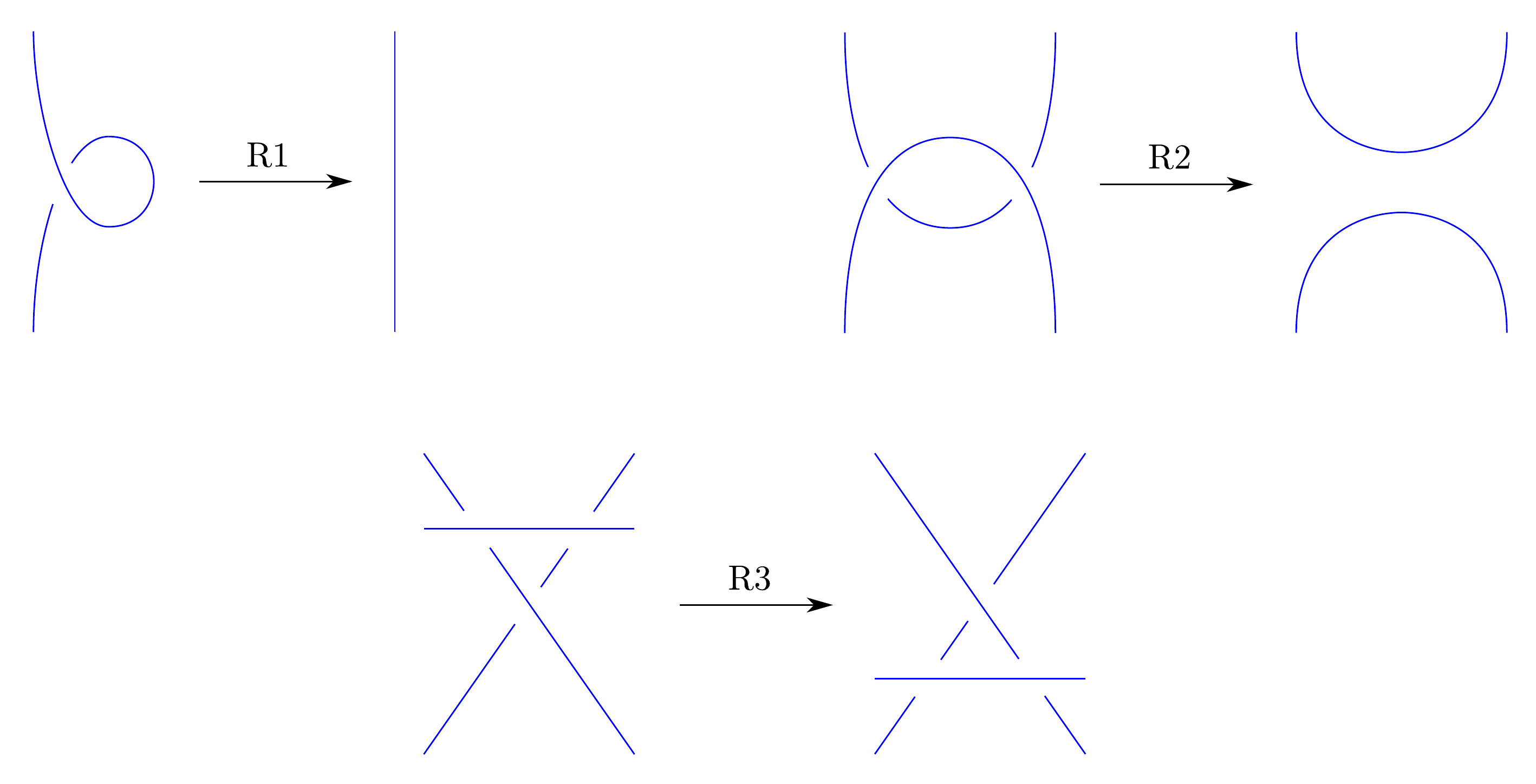}
  \caption{
    \label{fig:reidemeister}
    Reidemeister moves R1, R2, and R3.
  }
\end{figure}

\subsection{Reidemeister move R1}

Suppose we apply the Reidemeister move R1 shown in Figure
\ref{fig:reidemeister} to a tangle diagram $T_A$, resulting in a
tangle diagram $T_B$.
Let
$(X_A,\delta_A)$ and $(X_B,\delta_B)$ denote the twisted objects
corresponding to $T_A$ and $T_B$.
The objects $X_A$ and $X_B$ are formal direct sums of terms, as
described by equation (\ref{eqn:obj-X}).
Each term of $X_B$ corresponds to a planar tangle $P_n(r)$
obtained by resolving $T_B$, and can be viewed as a
subobject $(x_B,0)$ of $(X_B,\delta_B)$ with zero differential:
\begin{eqnarray*}
(x_B,0) = 
\begin{tikzcd}
  L_n \otimes A^{\otimes r}[h_B,q_B],
\end{tikzcd}
\end{eqnarray*}
where the bigrading shift $[h_B,q_B]$ depends on the crossing numbers
of $T_B$ and on the location of $P_n(r)$ within the cube of
resolutions.

If we apply the move R1 in reverse to the planar tangle $P_n(r)$ and
then resolve the resulting crossing, we obtain a saddle
$S_{A^+}:P_n(r) \rightarrow P_n(r+1)$ that splits a circle component
$C$ from the arc $A$, as shown in Figure \ref{fig:move-r1}.
Let
$d_{A^+}:L_n \otimes A^{\otimes r} \rightarrow
L_n \otimes A^{\otimes(r+1)}$ denote
the corresponding differential.
We conclude that for each subobject $(x_B,0)$ of $(X_B,\delta_B)$,
there is a corresponding subobject $(x_A,d^0)$ of $(X_A,\delta_A)$
given by
\begin{eqnarray*}
(x_A,d^0) = 
\begin{tikzcd}
  L_n \otimes A^{\otimes r}[h_A,q_A] \arrow{r}{d^0} &
  L_n \otimes A^{\otimes(r+1)}[h_A+1,q_A+2],
\end{tikzcd}
\end{eqnarray*}
where
\begin{align*}
  (d^0)^{(1,1)} = d_{A^+}[1,2].
\end{align*}

The tangle diagram $T_A$ contains one more negative crossing than
$T_B$, so the relative crossing numbers of $T_A$ and $T_B$ are
$\delta m_+ = 0$ and $\delta m_- = 1$, corresponding to a bigrading
shift
\begin{align*}
  [h_A, q_A] =
  [h_B,q_B] + [-\delta m_-, \delta m_+ - 3 \delta m_-] =
  [h_B,q_B] + [-1, -3].
\end{align*}
We will show that $(X_A,\delta_A)$ and $(X_B,\delta_B)$ are homotopy
equivalent by constructing homotopy equivalences between each
pair of corresponding subobjects $(x_A,0)$ and $(x_B,d^0)$.
The homotopy equivalences will commute with the differentials
$\delta_A$ and $\delta_B$, and will thus combine to give homotopy
equivalences of $(X_A,\delta_A)$ and $(X_B,\delta_B)$.

\begin{figure}
  \centering
  \includegraphics[scale=0.50]{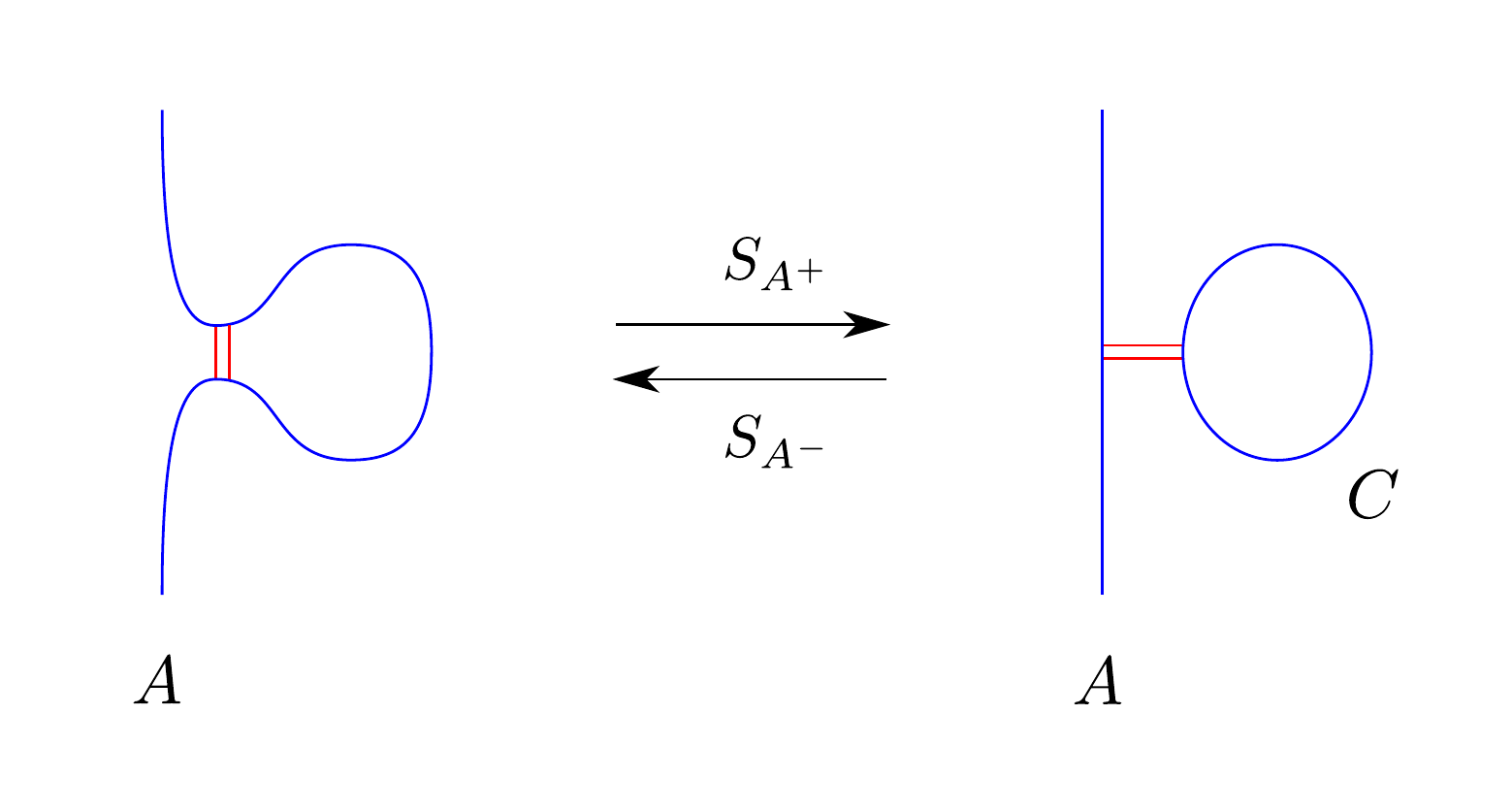}
  \caption{
    \label{fig:move-r1}
    Saddles $S_{A^+}$ and $S_{A^-}$ for Reidemeister move R1.
  }
\end{figure}

We define homotopy equivalences $F:(x_B,0) \rightarrow (x_A,d^0)$ and
$G:(x_A,d^0) \rightarrow (x_B,0)$ and a homotopy
$H:(x_A^0,d^0) \rightarrow (x_A^0, d^0)$:
\begin{eqnarray*}
\begin{tikzcd}
L_n \otimes A^{\otimes r} \arrow[shift left=2]{r}{d^0} &
L_n \otimes A^{\otimes(r+1)} \arrow[shift left=2]{l}{H^1}
\arrow[shift left=2]{d}{G^1} \\
{} &
L_n \otimes A^{\otimes r}. \arrow[shift left=2]{u}{F^1}
\end{tikzcd}
\end{eqnarray*}
For simplicity we have omitted bigrading shifts in the diagram.
The homotopy equivalences are given by
\begin{align*}
  (F^1)^{(0,0)} &= \cup[0,-1], &
  (G^1)^{(0,0)} &= d_{A^-}[0,1],
\end{align*}
where $\cup$ is the cup map for creating $C$ and
$d_{A^-}$ is the differential for the saddle
$S_{A^-}:P_n(r+1) \rightarrow P_n(r)$ that merges $C$ with the arc
$A$.
The homotopy is given by
\begin{align*}
  (H^1)^{(-1,-1)} &= \cap[-1,-2],
\end{align*}
where $\cap$ is the cap map for removing $C$.
Using the Bar-Natan relations described in Theorem
\ref{theorem:rel-1}, we find
\begin{align*}
  &G^1 \circ d^0 = (d_{A^-} \circ d_{A^+})[1,3] = 0, \\
  &G^1 \circ F^1 = d_{A^-} \circ \cup = a_n I_r, \\
  &H^1 \circ d^0 = \cap \circ d_{A^+} = a_n I_r, \\
  &d^0 \circ H^1 + F^1 \circ G^1 =
  d_{A^+} \circ \cap + \cup \circ d_{A^-} = a_n
  I_{r+1},
\end{align*}
so $F$ and $G$ are indeed homotopy equivalences.

The homotopy equivalences we have defined for subobjects are
constructed from differentials and cup maps, which by
Theorems  \ref{theorem:d-d-commute} and \ref{theorem:d-cc-commute}
commute with the differentials $\delta_A$ and $\delta_B$.
It follows that the homotopy equivalences for subobjects combine to
yield homotopy equivalences of the objects $(X_A, \delta_A)$ and
$(X_B,\delta_B)$.

A similar argument shows invariance under R1 with a positive crossing.

\subsection{Reidemeister move R2}

Suppose we apply the Reidemeister move R2 shown in Figure
\ref{fig:reidemeister} to a tangle diagram $T_A$, resulting in a
tangle diagram $T_B$.
Let $(X_A,\delta_A)$ and $(X_B,\delta_B)$ denote the twisted objects
corresponding to $T_A$ and $T_B$.
The objects $X_A$ and $X_B$ are formal direct sums of terms, as
described by equation (\ref{eqn:obj-X}).
Each term of $X_B$ corresponds to a planar tangle $P_n(r)$ obtained by
resolving $T_B$, and can be viewed as a subobject $(x_B,0)$ of
$(X_B,\delta_B)$ with zero differential:
\begin{eqnarray*}
  (x_B,0) =
\begin{tikzcd}
  L_n \otimes A^{\otimes r}[h_B,q_B],
\end{tikzcd}
\end{eqnarray*}
where the bigrading shift $[h_B,q_B]$ depends on the crossing numbers
of $T_B$ and on the location of $P_n(r)$ within the cube of
resolutions.

If we apply the move R2 in reverse to the planar tangle $P_n(r)$ and
then resolve the two resulting crossings, we obtain a cube of
resolutions
\begin{eqnarray*}
\begin{tikzcd}
{} & P_m(s+1)
\arrow[shift left=2]{dr}{S_{B^-}} \\
P_m(s)
\arrow[shift left=2]{ur}{S_{A^+}}
\arrow[shift left=2,swap]{dr}{S_{=\parallel}}
& {} & P_m(s)
\\
{} & P_n(r)
\arrow[shift left=2,swap]{ur}{S_{\parallel =}}
\end{tikzcd}
\end{eqnarray*}
as shown in Figure \ref{fig:move-r2}.
The winding number $m$ and circle number $s$ of the planar tangle
$P_m(s)$ depend on how the move is applied.
As shown in Figure \ref{fig:move-r2}, the saddle $S_{A^+}$ splits the
circle component $C$ from arc $A$ and the saddle $S_{B^-}$ merges the
circle component $C$ with arc $B$.
Let $d_{A^+}$, $d_{B^-}$, $d_{=\parallel}$, and $d_{\parallel =}$
denote the differentials corresponding to the saddles
$S_{A^+}$, $S_{B^-}$, $S_{=\parallel}$, and $S_{\parallel =}$.
We conclude that for each subobject $(x_B,0)$ of $(X_B,\delta_B)$,
there is a corresponding subobject $(x_A, d)$ of
$(X_A, \delta_A)$ given by
\begin{eqnarray*}
(x_A, d) =
\begin{tikzcd}
{} & L_m \otimes A^{\otimes(s+1)}[h_B+1,q_B+2]
\arrow[shift left=2]{dr}{d_{B^-}[1,2]} \\
L_m \otimes A^{\otimes s}[h_B,q_B]
\arrow[shift left=2]{ur}{d_{A^+}[1,2]}
\arrow[shift left=2,swap]{dr}{d_{=\parallel}[1,2]}
& {} & L_m \otimes A^{\otimes s}[h_B+2,q_B+4].
\\
{} & L_n \otimes A^{\otimes r}[h_B+1,q_B+2]
\arrow[shift left=2,swap]{ur}{d_{\parallel =}[1,2]}
\end{tikzcd}
\end{eqnarray*}
The differentials for $(x_A,d)$ are
\begin{align*}
  (d^0)^{(1,1)} &= \left(\begin{array}{c}
    d_{A^+}[1,2] \\
    d_{=\parallel}[1,2]
  \end{array}\right), &
  (d^1)^{(1,1)} &= \left(\begin{array}{cc}
    d_{B^-}[1,2] &
    d_{\parallel =}[1,2]
  \end{array}\right).
\end{align*}

The tangle diagram $T_A$ contains one more negative crossing and one
more positive crossing than $T_B$, so the relative crossing numbers of
$T_A$ and $T_B$ are $\delta m_+ = 1$ and $\delta m_- = 1$,
corresponding to a bigrading shift
\begin{align*}
  [h_A, q_A] =
  [h_B,q_B] + [-\delta m_-, \delta m_+ - 3 \delta m_-] =
  [h_B,q_B] + [-1,-2].
\end{align*}

\begin{figure}
  \centering
  \includegraphics[scale=0.50]{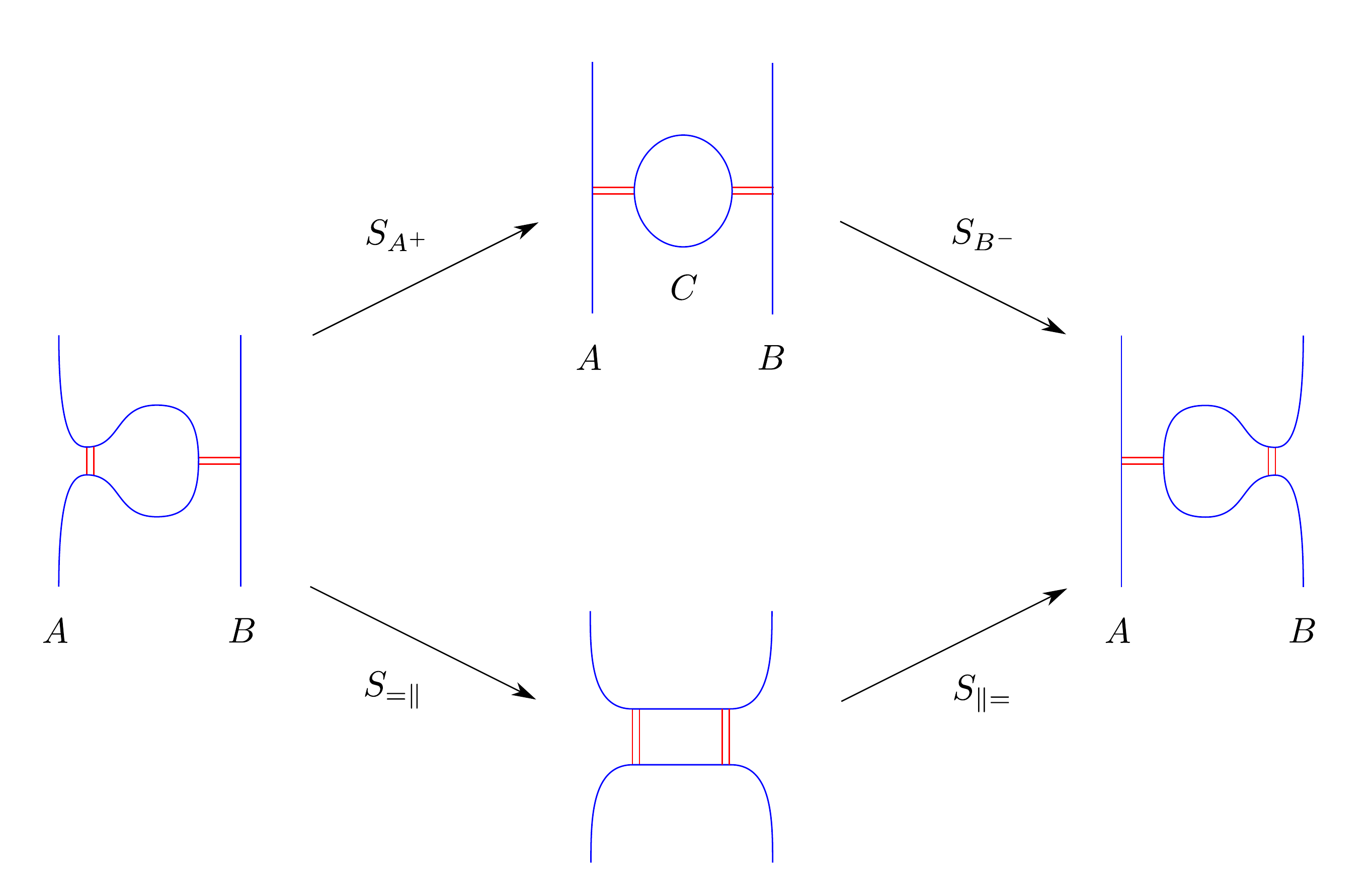}
  \caption{
    \label{fig:move-r2}
    Cube of resolutions for Reidemeister move R2.
  }
\end{figure}

We define homotopy equivalences $F:(x_B,0) \rightarrow (x_A,d^0)$ and
$G:(x_A,d) \rightarrow (x_B,0)$ and a homotopy
$H:(x_A^0,d) \rightarrow (x_A^0, d)$:
\begin{eqnarray*}
\begin{tikzcd}
{} & L_m \otimes A^{\otimes(s+1)}
\arrow[shift left=2]{dr}{d_{B^-}}
\arrow[bend left=20, shift left=2]{rrrd}{G^1}
\arrow[shift left=2]{dl}{\cap} \\
L_m \otimes A^{\otimes s}
\arrow[shift left=2]{ur}{d_{A^+}}
\arrow[shift left=2,swap]{dr}{d_{=\parallel}}
& {} & L_m \otimes A^{\otimes s}
\arrow[shift left=2]{lu}{\cup} & {} & {}
L_n \otimes A^{\otimes r}.
\arrow[bend left=20, shift left=2]{llld}{F^1}
\arrow[bend right=20, shift left=2]{lllu}{F^1}
\\
{} & L_n \otimes A^{\otimes r}
\arrow[shift left=2,swap]{ur}{d_{\parallel =}}
\arrow[bend right=20, shift left=2]{rrru}{G^1}
\end{tikzcd}
\end{eqnarray*}
For simplicity we have omitted bigrading shifts in the diagram.
Here $\cup$ and $\cap$ the cup and cap maps for creating and
removing the circle component $C$.
The homotopy equivalences are given by
\begin{align*}
  (F^1)^{(0,0)} &= \left(\begin{array}{c}
    \cup \circ d_{\parallel =} \\
    a_n I_r
  \end{array}\right), &
  (G^1)^{(0,0)} &= \left(\begin{array}{cc}
    d_{=\parallel} \circ \cap & a_n I_r
  \end{array}\right).
\end{align*}
The homotopy $H$ is given by
\begin{align*}
  (H^1)^{(-1,-1)} &=
  \left(\begin{array}{cc}
    \cap[-1,-2] & 0
  \end{array}\right), &
  (H^2)^{(-1,-1)} &=
  \left(\begin{array}{c}
    \cup[-1,-2] \\ 0
  \end{array}\right).
\end{align*}
Using the isotopy invariance relations described in Theorem
\ref{theorem:rel-1}, we find
\begin{align*}
  &d^1 \circ F^1 =
  ((d_{B^-} \circ \cup) \circ d_{\parallel =} + d_{\parallel =})[1,2]
  = 0, &
  &G^1 \circ d^0 =
  (d_{=\parallel} \circ (\cap \circ d_{A^+}) + d_{=\parallel})[1,2] = 0.
\end{align*}
Using the sphere relation described in Theorem \ref{theorem:rel-0} and
the isotopy invariance relations described in Theorem
\ref{theorem:rel-1}, we find
\begin{align*}
  &H^1 \circ d^0 = \cap \circ d_{A^+} = a_m I_s, \\
  &d^1 \circ H^2 = d_{B^-} \circ \cup = a_m I_s, \\
  &G^1 \circ F^1 =
  d_{=\parallel} \circ (\cap \circ \cup) \circ d_{\parallel =} +
  a_n I_r = a_n I_r,
\end{align*}
and using the 4-tube relation described in Theorem \ref{theorem:rel-2}
we find
\begin{align*}
  \MoveEqLeft
  d^0 \circ H^1 + F^1 \circ G^1 + H^2 \circ d^1 \\
  &=
  \left(\begin{array}{cc}
    d_{A^+} \circ \cap & 0 \\
    d_{=\parallel} \circ \cap & 0
  \end{array}\right) +
  \left(\begin{array}{cc}
    \cup \circ (d_{\parallel =} \circ d_{=\parallel}) \circ \cap &
    \cup \circ d_{\parallel =} \\
    d_{=\parallel} \circ \cap & a_n I_r
  \end{array}\right) +
  \left(\begin{array}{cc}
    \cup \circ d_{B^-} & \cup \circ d_{\parallel =} \\
    0 & 0
  \end{array}\right) \\
  &=
  \left(\begin{array}{cc}
    a_m I_{s+1} & 0 \\
    0 & a_n I_r
  \end{array}
  \right).
\end{align*}
So $F$ and $G$ are indeed homotopy equivalences.
We note also that  using the sphere relation from Theorem
\ref{theorem:rel-0}, we have
\begin{align*}
  H^1 \circ F^1 &= 0.
\end{align*}

The homotopy equivalences we have defined for subobjects are
constructed from differentials, cup maps, and cap maps, which by
Theorems  \ref{theorem:d-d-commute} and \ref{theorem:d-cc-commute}
commute with the differentials $\delta_A$ and $\delta_B$.
It follows that the homotopy equivalences for subobjects combine to
yield homotopy equivalences of the objects $(X_A, \delta_A)$ and
$(X_B,\delta_B)$.

The homotopy equivalence
$(X_A,\delta_A) \rightarrow (X_B,\delta_B)$
that we have constructed has a special property that will be
useful in proving invariance under move R3.
In the category of cochain complexes, a cochain map
$G:(C_A,d_A) \rightarrow (C_B,d_B)$ is said to be a
\emph{strong deformation retraction} if there is a cochain map
$F:(C_B,d_B) \rightarrow (C_A,d_A)$ and a homotopy
$H:(C_A,d_A) \rightarrow (C_A,d_A)$ such that
\begin{align*}
  G\circ F &= I_{(C_B,d_B)}, &
  F\circ G &= I_{(C_A,d_A)} + d_A\circ  H + H\circ  d_A, &
  H\circ F &= 0.
\end{align*}
We can extend this notion to the category of twisted complexes over an
arbitrary $A_\infty$ category $\calA$.
We say that a homotopy equivalence
$G:(X,\delta_X) \rightarrow (Y,\delta_Y)$
of twisted complexes is a \emph{strong deformation retraction} if
there is a homotopy equivalence
$F:(Y,\delta_Y) \rightarrow (X,\delta_X)$ and a homotopy
$H:(X,\delta_X) \rightarrow (X_X,\delta_X)$ such that
\begin{align*}
  &\mu_{\Tw A}^2(G,F) = \id_{(X_Y,\delta_Y)}, &
  &\mu_{\Tw \calA}^2(F,G) = \id_{(X,\delta_X)} +
  \mu_{\Tw \calA}^1(H), &
  &\mu_{\Tw \calA}^2(H,F) = 0.
\end{align*}
According to Conjecture \ref{conj:fukaya}, the $A_\infty$ operations
$\mu_\calL^m$ are zero for $m \neq 0$, so these conditions reduce to
conditions that are formally identical to those used to define the
notion of a strong deformation retraction in the category of cochain
complexes:
\begin{align*}
  &G \circ F = \id_{(X_Y,\delta_Y)}, &
  &F \circ G = \id_{(X,\delta_X)} +
  \delta_Y \circ H + H \circ \delta_X, &
  &H \circ F = 0,
\end{align*}
where $\circ$ denotes the product operation $\mu_{\Sigma\calL}^2$ in
$\Sigma\calL$.
So the homotopy equivalence
$(X_A,\delta_A) \rightarrow (X_B,\delta_B)$ we have constructed is a
strong deformation retraction.

\subsection{Reidemeister move R3}

An argument presented by Bar-Natan in Section 4.3 of \cite{BarNatan}
now proves invariance under move R3 using invariance under move R2
and the fact that the homotopy equivalence for R2 is a strong
deformation retraction.
Bar-Natan's argument is formulated in terms of the category of twisted
complexes over a cobordism category modulo local relations, but given
Conjecture \ref{conj:fukaya}, which states that only product
operations of $\calL$ are nonzero, it carries over directly to
$\Tw \calL$.
This completes the proof of Theorem \ref{theorem:invariance}

\section{Links in $S^3$}
\label{sec:s3}

In Section \ref{sec:complexes}, we showed that given an oriented
1-tangle diagram $T$ we can construct corresponding twisted
complexes $(X,\delta_+)$ and $(X,\delta_-)$.
We can close $T$ with an overpass arc $A_+$ or underpass arc $A_-$, as
shown in Figure \ref{fig:s3-lp-lm}, to obtain diagrams of oriented
links $L_{T^+}$ and $L_{T^-}$ in $S^3$.
Our goal in this section is to use the twisted complexes
$(X,\delta_+)$ and $(X,\delta_-)$ to recover the reduced Khovanov
homology for these links.

Recall from Section \ref{sec:lagrangians-s3} that for links in $S^3$ 
an overpass arc $A_+$ and an underpass arc $A_-$ correspond to
Lagrangians $\Wbar_0$ and $\Wbar_1$ in $R^*(T^2,2)$.
For convenience, we will denote these Lagrangians as $W_+$ and $W_-$:
\begin{align*}
  &W_+ = \Wbar_0, &
  &W_- = \Wbar_1.
\end{align*}
Recall from Section \ref{ssec:a-infinity} that given a Lagrangian $M$
in $R^*(T^2,2)$ we can define an $A_\infty$ functor
$\Tw \calG_M:\Tw \calL \rightarrow \Ch$.
We a define cochain complex $(C_\pm, \partial_\pm)$ for the link
$L_{T^\pm}$ by applying the functor $\Tw \calG_{W_\pm}$ to the
twisted object $(X,\delta_\pm)$ and shifting the bigrading:
\begin{align*}
  (C_\pm, \partial_\pm) =
  (\Tw \calG_{W_\pm}((X,\delta_\pm))[h_A,q_A].
\end{align*}

The bigrading shift $[h_A,q_A]$ accounts for crossings between the arc
$A_\pm$ and the tangle diagram $T$, and is computed as follows.
Assume that the arc intersects the tangle diagram $T$ transversely.
Let $a_+$ and $a_-$ denote the number of positive and negative
intersection points between the arc and $T$.
Note that $a_+$ and $a_-$ depend on the position of the arc in
relation to $T$, but the difference $a_+ - a_-$ does not.
The bigrading shift $[h_A,q_A]$ is then given by
\begin{align*}
  [h_A,q_A] =
  \left\{
  \begin{array}{ll}
    {[r,4r]} &
    \quad \mbox{for $a_+ - a_- = 2r$ even,} \\
          {[r,4r+2]} & \quad
          \mbox{for $a_+ - a_- = 2r + 1$ odd.}
  \end{array}
  \right.
\end{align*}
Recall that $(X,\delta_\pm)$ already includes a bigrading shift
$[h_T,q_T]$ that accounts for crossings of $T$ with itself.

Given certain assumptions regarding the $A_\infty$ operations of
$\calL$ and the gradings of the generators of the vector spaces
$(W_\pm, L_n)$, we can explicitly construct the cochain complex
$(C_\pm,\partial_\pm)$ from the tangle diagram $T$.
We first describe the assumptions that we will need.

\subsection{$A_\infty$ operations and gradings of generators}
\label{sec:s3-ops-gens}

Recall from Theorems \ref{theorem:Wbar0} and \ref{theorem:Wbar1} that
the vector spaces $(W_\pm, L_n)$ are
1-dimensional and are each generated by a vector with positive
orientation grading:
\begin{align*}
  (W_+,L_n) &= (\Wbar_0, L_n) = \langle w_{n,+}^{(+)} \rangle, &
  (W_-, L_n) &= (\Wbar_1, L_n) = \langle w_{n,-}^{(+)} \rangle,
\end{align*}
where for convenience we have defined
\begin{align*}
  w_{n,+} &= \wbar_{n,0}, &
  w_{n,-} &= \wbar_{n,1}.
\end{align*}
We assign an integer homological grading $h$ and an integer quantum
grading $q$ to each generator, as indicated by the
superscripts $(h,q)$:
\begin{align*}
  (W_\pm, L_n) &=
  \left\{
  \begin{array}{ll}
    \langle w_{n,\pm}^{(\pm \frac{1}{2}n, \pm 2n)} \rangle &
    \quad \mbox{for $n$ even,} \\
    \langle w_{n,\pm}^{(\frac{1}{2}(1 \pm n), \pm 2n)} \rangle &
    \quad \mbox{for $n$ odd,}
  \end{array}
  \right.
\end{align*}

To construct the cochain complex $(C_\pm, \partial_\pm)$, we need to
know $A_\infty$ operations involving the generators
$w_{n,\pm}$ of $(W_\pm,L_n)$ as well as the generators
$a_n$, $c_n$, $p_{n+2,n}^{(\mp)}$, and $q_{n-2,n}^{(\pm)}$ that appear
in $\delta_\pm$.
We make the following conjecture regarding these operations:

\begin{conjecture}
\label{conj:s3-ops}
We have the following product operations:
\begin{align*}
  \mu^2(a_n, w_{n,\pm}) &= w_{n,\pm}.
\end{align*}
We have the following $\mu^3$ operations:
\begin{align*}
  &\mu^3(c_{n-2}, q_{n-2,n}^{(+)}, w_{n,+}) =
  \mu^3(q_{n-2,n}^{(+)}, c_n, w_{n,+}) =
  w_{n-2,+}, \\
  &\mu^3(c_{n+2}, p_{n+2,n}^{(+)}, w_{n,-}) =
  \mu^3(p_{n+2,n}^{(+)}, c_n, w_{n,-}) =
  w_{n+2,-}.
\end{align*}
All other operations of the form
\begin{align*}
  &\mu^m(x_{m-1}, \cdots, x_1, w_{n,+}), &
  &\mu^m(x_{m-1}, \cdots, x_1, w_{n,-})
\end{align*}
for
$x_{m-1}, \cdots, x_1 \in
\{a_n,\, c_n,\, p_{n+2,n}^{(\mp)},\, q_{n-2,n}^{(\pm)} \mid n \in
\Ints \}$ are zero.
\end{conjecture}

It is straightforward to check that Conjecture \ref{conj:s3-ops} is
consistent with the orientation gradings of the generators.
As described in Appendix \ref{sec:pillowcase}, Conjecture
\ref{conj:s3-ops} is a natural generalization of a corresponding
statement regarding the Fukaya category of $R^*(S^2,4)$.
One motivation for Conjecture \ref{conj:s3-ops} is the following
result:

\begin{theorem}
If the $A_\infty$ operations respect bigradings, then all operations
of the form
\begin{align*}
  &\mu^m(x_{m-1}, \cdots, x_1, w_{n,+}), &
  &\mu^m(x_{m-1}, \cdots, x_1, w_{n,-}),
\end{align*}
for $x_{m-1}, \cdots, x_1 \in
\{a_n,\, c_n,\, p_{n+2,n}^{(\mp)},\, q_{n-2,n}^{(\pm)} \mid n \in
\Ints \}$
must be zero, with the exception of
\begin{align*}
  &\mu^2(a_n,w_{n,+}), &
  &\mu^3(c_{n-2}, q_{n-2,n}^{(+)}, w_{n,+}), &
  &\mu^3(q_{n-2,n}^{(+)}, c_n, w_{n,+}), \\
  &\mu^2(a_n,w_{n,-}), &
  &\mu^3(c_{n+2}, p_{n+2,n}^{(+)}, w_{n,-}), &
  &\mu^3(p_{n+2,n}^{(+)}, c_n, w_{n,-}).
\end{align*}
\end{theorem}

\begin{proof}
Let $n_a$, $n_c$, $n_p$, and $n_q$ denote the number of $x_k$'s that
belong to the sets
\begin{align*}
  &\{a_n \mid n \in \Ints \}, &
  &\{c_n \mid n \in \Ints \}, &
  &\{p_{n+2,n}^{(\pm)} \mid n \in \Ints \}, &
  &\{q_{n-2,n}^{(\pm)} \mid n \in \Ints \}.
\end{align*}
We have
\begin{align}
  \label{eqn:sum-ns}
  n_a + n_c + n_p + n_q  = m - 1.
\end{align}
If the $A_\infty$ operation involving $w_{n,+}$ is nonzero, then we
must have
\begin{align}
  \label{eqn:nonzero-op}
  \mu^m(x_{m-1}, \cdots, x_1, w_{n,+}^{(h,q)}) =
  w_{n-2r,+}^{(h-r,q-4r)}
\end{align}
for some integer $r$.
We must balance the winding numbers on both sides of equation
(\ref{eqn:nonzero-op}), so:
\begin{align}
  \label{eqn:ops-w}
  n_q - n_p &= r.
\end{align}
Given our assumption that the $A_\infty$ relations respect bigradings,
we have
\begin{align}
  \label{eqn:ops-h-q}
  2n_c + n_p + n_q + m - 2 &= 4r, &
  m - 2 &= r.
\end{align}
From equations (\ref{eqn:sum-ns}), (\ref{eqn:ops-w}), and
(\ref{eqn:ops-h-q}), it follows that
\begin{align*}
  n_a + n_c + 2n_p = 1.
\end{align*}
So $n_p = 0$ and either $(n_a, n_c) = (1,0)$, in which case
\begin{align*}
  n_a &= 1, & n_c &= 0, & n_p &= 0, & n_q &= 0, & m &= 2, & r &= 0,
\end{align*}
or $(n_a,n_c) = (0,1)$, in which case
\begin{align*}
  n_a &= 0, & n_c &= 1, & n_p &= 0, & n_q &= 1, & m &= 3, & r &= 1.
\end{align*}
The operations
\begin{align*}
  &\mu^3(c_{n-2}, q_{n-2,n}^{(-)}, w_{n,+}), &
  &\mu^3(q_{n-2,n}^{(-)}, c_n, w_{n,+}),
\end{align*}
must be zero, since otherwise they would have to be
$w_{n-2,+}$, which is inconsistent with orientation gradings.
The proof for the $A_\infty$ operations involving $w_{n,-}$ is
similar.
\end{proof}

\subsection{Chain complex}

Using Conjecture \ref{conj:s3-ops}, the description of the functor
$\Tw \calG_{W_\pm}$ from Section \ref{ssec:a-infinity}, and the
description of the twisted complex $(X,\delta_\pm)$ from Section
\ref{ssec:twisted-complex}, we can
explicitly construct the cochain complex $(C_\pm,\partial_\pm)$ for
the link $L_{T^\pm}$.
The vector space $C_\pm$ is given by
\begin{align*}
  C_\pm = \bigoplus_{i \in  I}
  ((W_\pm, L_{n_i}) \otimes
  A^{\otimes c_i}[r(i)+h_T+h_A, 2r(i)+q_T+q_A]).
\end{align*}
There are no $\mu^1$ operations, so the differentials $\partial_\pm$
are due entirely to $\delta_\pm$.
The $\mu^2$ operations give \emph{short differentials} corresponding
to single saddles, and the $\mu^3$ operations give
\emph{long differentials} corresponding to pairs of successive
saddles.

For both $L_{T^+}$ and $L_{T^-}$, the $\mu^2$ operations involving
$w_{n,\pm}$ give the following short differentials:
\begin{enumerate}
\item
For a saddle
$P_n(r+1) \rightarrow P_n(r+2)$
that splits one circle into two circles:
\begin{align*}
  & \partial_{C^+}:
  (W_\pm,L_n) \otimes A^{\otimes(r+1)} \rightarrow
  (W_\pm,L_n) \otimes A^{\otimes(r+2)}[1,2], &
  & \partial_{C^+} =
  ((w_{n,\pm} \mapsto w_{n,\pm}) \otimes \Delta I_r)[1,2].
\end{align*}
For a saddle
$P_n(r+2) \rightarrow P_n(r+1)$
that merges two circles into one circle:
\begin{align*}
  & \partial_{C^-}:
  (W_\pm,L_n) \otimes A^{\otimes(r+2)} \rightarrow
  (W_\pm,L_n) \otimes A^{\otimes(r+1)}[1,2], &
  & \partial_{C^-} =
  ((w_{n,\pm} \mapsto w_{n,\pm}) \otimes m I_r)[1,2].
\end{align*}

\item
For a saddle
$P_n(r) \rightarrow P_n(r+1)$
that splits a circle from either side of the arc component:
\begin{align*}
  & \partial_{A^+}:
  (W_\pm,L_n) \otimes A^{\otimes r} \rightarrow
  (W_\pm,L_n) \otimes A^{\otimes(r+1)}[1,2], &
  & \partial_{A^+} =
  ((w_{n,\pm} \mapsto w_{n,\pm}) \otimes \etadot I_r)[1,2].
\end{align*}
For a saddle
$P_n(r+1) \rightarrow P_n(r)$
that merges a circle with either side of the arc component:
\begin{align*}
  & \partial_{A^-}:
  (W_\pm,L_n) \otimes A^{\otimes(r+1)} \rightarrow
  (W_\pm,L_n) \otimes A^{\otimes r}[1,2], &
  & \partial_{A^-} =
  ((w_{n,\pm} \mapsto w_{n,\pm}) \otimes \epsilondot I_r)[1,2].
\end{align*}
\end{enumerate}

For $L_{T^+}$, the $\mu^3$ operations involving $w_{n,+}$ give the
following long differentials:
\begin{enumerate}
\item
For a saddle
$P_n(r) \rightarrow P_{n-2}(r)$
that decreases the winding number by two followed
by a saddle $P_{n-2}(r) \rightarrow P_{n-2}(r+1)$ that splits a circle
from the right side of the arc component, or for a saddle
$P_n(r) \rightarrow P_n(r+1)$
that splits a circle from the right side of the arc component
followed by a saddle
$P_n(r+1) \rightarrow P_{n-2}(r+1)$
that decreases the winding number by two:
\begin{align*}
  & \partial_{R^+W^-}:
  (W_+,L_n) \otimes A^{\otimes r} \rightarrow
  (W_+,L_{n-2}) \otimes A^{\otimes(r+1)}[2,4], &
  & \partial_{R^+W^-} =
  ((w_{n,+} \mapsto w_{n-2,+}) \otimes \eta I_r)[2,4].
\end{align*}

\item
For a saddle
$P_n(r+1) \rightarrow P_{n-2}(r+1)$
that decreases the winding number by two followed
by a saddle $P_{n-2}(r+1) \rightarrow P_{n-2}(r)$ that merges a circle
with the right side of the arc component, or for a saddle
$P_n(r+1) \rightarrow P_n(r)$ that merges a circle with the right side
of the arc component followed by a saddle
$P_n(r) \rightarrow P_{n-2}(r)$
that decreases the winding number by two:
\begin{align*}
  & \partial_{R^-W^-}:
  (W_+,L_n) \otimes A^{\otimes(r+1)} \rightarrow
  (W_+,L_{n-2}) \otimes A^{\otimes r}[2,4], &
  & \partial_{R^-W^-} =
  ((w_{n,+} \mapsto w_{n-2,+}) \otimes \epsilon I_r)[2,4].
\end{align*}

\end{enumerate}

For $L_{T^-}$, the $\mu^3$ operations involving $w_{n,-}$ give the
following long differentials:

\begin{enumerate}
\item
For a saddle
$P_n(r) \rightarrow P_{n+2}(r)$
that increases the winding number by two followed
by a saddle $P_{n+2}(r) \rightarrow P_{n+2}(r+1)$ that splits a circle
from the right side of the arc component, or for a saddle
$P_n(r) \rightarrow P_n(r+1)$ that splits a circle
from the right side of the arc component followed by a saddle
$P_n(r+1) \rightarrow P_{n+2}(r+1)$
that increases the winding number by two:
\begin{align*}
  & \partial_{R^+W^+}:
  (W_-,L_n) \otimes A^{\otimes r} \rightarrow
  (W_-,L_{n+2}) \otimes A^{\otimes(r+1)}[2,4], &
  & \partial_{R^+W^+} =
  ((w_{n,-} \mapsto w_{n+2,-}) \otimes \eta I_r)[2,4].
\end{align*}

\item
For a saddle
$P_n(r+1) \rightarrow P_{n+2}(r+1)$
that increases the winding number by two followed
by a saddle $P_{n+2}(r+1) \rightarrow P_{n+2}(r)$ that merges a circle
with the right side of the arc component, or for a saddle
$P_n(r+1) \rightarrow P_n(r)$ that merges a circle
with the right side of the arc component followed by a saddle
$P_n(r) \rightarrow P_{n+2}(r)$
that increases the winding number by two:
\begin{align*}
  & \partial_{R^-W^+}:
  (W_-,L_n) \otimes A^{\otimes r} \rightarrow
  (W_-,L_{n+2}) \otimes A^{\otimes(r+1)}[2,4], &
  & \partial_{R^-W^+} =
  ((w_{n,-} \mapsto w_{n+2,-}) \otimes \epsilon I_r)[2,4].
\end{align*}
\end{enumerate}

In fact, the long differentials due to the $\mu^3$ operations contain
additional terms involving the nonlocal operator $\Sigma_r$.
However, in a cube of resolutions for a tangle diagram these terms
always occur in canceling pairs corresponding to saddles on opposite
sides of the arc component, so we can omit them from the expressions
for the differentials.
We note that for tangle diagrams with loop number 0, there are no long
differentials and the cochain complex is the usual cochain complex for
reduced Khovanov homology, where the marked component is the one
containing the arc component $T$, and with the convention that the
differential has quantum grading 1, rather than the usual quantum
grading of 0.

\begin{theorem}
We have
\begin{align*}
  \Khr^{(h,q)}(L_{T^\pm}) =  H^{(h,h+q)}((C_\pm,\partial_\pm)).
\end{align*}
\end{theorem}

\begin{proof}
This is shown in \cite{Boozer-3}, although we need to translate
conventions.
Here we define generators $w_{n,+}$ and $w_{n,-}$ for an overpass and
underpass arc, and we shift bigradings by $[h_A,q_A]$ to account for
crossings of the arc with the tangle.
In \cite{Boozer-3}, the corresponding cochain complex is defined in
terms of generators $P_{\ell,n}^+$ or $P_{\ell,n}^-$ for an
overpass or underpass arc, where $\ell = a_+ + a_-$, and we shift
bigradings by $[h_A^P,q_A^P]$ to account for crossings of the arc with
the tangle diagram.
The bigrading $(h,q)$ of $P_{\ell,n}^\pm$ is
\begin{align*}
  ((1/2)(\ell \pm n),\, (1/2)(\ell \pm 3n)).
\end{align*}
The bigrading shift $[h_A^P,q_A^P]$ is given by
\begin{align*}
  [h_A^P, q_A^P] = [-a_-, a_+ - 2a_-].
\end{align*}
The generators of the two complexes are thus related by
\begin{align*}
  &P_{\ell,n}^+[h_A^P,q_A^P] & \longleftrightarrow &
  &w_{n,+} [h_A,q_A], \\
  &P_{\ell,n}^-[h_A^P,q_A^P] & \longleftrightarrow &
  &w_{n,-} [h_A,q_A].
\end{align*}
The convention for the bigrading of the differential is
$(1,0)$ in \cite{Boozer-3} and $(1,1)$ here, so we want to show that
the bigradings of the generators are related by
\begin{align*}
  &(h,q) & \longleftrightarrow &
  &(h,h+q).
\end{align*}
This is a straightforward calculation.
\end{proof}

\subsection{Example: unknot in $S^3$}
\label{ssec:example-unknot}

We can close the tangle diagram $T$ shown in Figure
\ref{fig:example-tangle} with an overpass arc $A_+$ as shown in Figure
\ref{fig:s3-lp-lm} to obtain the link $L_{T^+}$, which is the unknot.
The number of positive and negative crossings between $T$ and $A_+$ is
$a_+ = 0$ and $a_- = 2$, so $a_+ - a_- = -2 = 2r$ for $r = -1$,
corresponding to a bigrading shift
\begin{align*}
  [h_A, q_A] = [r, 4r] = [-1,-4].
\end{align*}
We obtain the following cochain complex $(C_+,\partial_+)$:
\begin{eqnarray*}
\begin{tikzcd}
  {} & \langle w_{0,+}^{(0,0)} \rangle \otimes \F[0,0]
  \arrow{dr}{((w_{0,+} \mapsto w_{0,+}) \otimes \etadot)[1,2]} & {}
  \\
  \langle w_{2,+}^{(1,4)} \rangle \otimes \F[-1,-2]
  \arrow{rr}{((w_{2,+} \mapsto w_{0,+})\otimes\eta)[2,4]} &
  {} & \langle w_{0,+}^{(0,0)} \rangle \otimes A[1,2]. \\
  {} & \langle w_{0,+}^{(0,0)} \rangle \otimes \F[0,0]
  \arrow{ur}[swap]{((w_{0,+} \mapsto w_{0,+}) \otimes \etadot)[1,2]}
  & {}
\end{tikzcd}
\end{eqnarray*}
So we obtain the reduced Khovanov homology $\F[0,0]$ for the unknot.

\subsection{Example: right trefoil in $S^3$}
\label{ssec:example-trefoil}

We can close the tangle diagram $T$ shown in Figure
\ref{fig:example-tangle} with an underpass arc $A_-$ as shown in
Figure \ref{fig:s3-lp-lm} to obtain the link $L_{T^-}$, which is the
right trefoil.
The number of positive and negative crossings between $T$ and $A_-$ is
$a_+ = 2$ and $a_- = 0$, so $a_+ - a_- = 2 = 2r$ for $r = 1$,
corresponding to a bigrading shift
\begin{align*}
  [h_A, q_A] = [r,4r] = [1,4].
\end{align*}
We obtain the following cochain complex $(C_-,\partial_-)$:
\begin{eqnarray*}
\begin{tikzcd}
  {} & \langle w_{0,-}^{(0,0)} \rangle \otimes \F[2,8]
  \arrow{dr}{((w_{0,-} \mapsto w_{0,-})\otimes \etadot)[1,2]}
  & {} \\
  \langle w_{2,-}^{(-1,-4)} \rangle \otimes \F[1,6] &
  {} & \langle w_{0,-}^{(0,0)} \rangle \otimes A[3,10]. \\
  {} & \langle w_{0,-}^{(0,0)} \rangle \otimes \F[2,8]
  \arrow{ur}[swap]{((w_{0,-} \mapsto w_{0,-})\otimes \etadot)[1,2]}
  & {}
\end{tikzcd}
\end{eqnarray*}
So we obtain the reduced Khovanov homology
$\F[0,2] \oplus \F[2,6] \oplus \F[3,8]$ for the right trefoil.

\section{Links in $S^2 \times S^1$}
\label{sec:s2xs1}

In Section \ref{sec:complexes}, we showed that given an oriented
1-tangle diagram $T$ we can construct corresponding twisted
complexes $(X,\delta_+)$ and $(X,\delta_-)$.
We can construct links $L_{T^+}'$ and $L_{T^-}'$ in $S^2 \times S^1$ by
closing the tangle diagram $T$ with an overpass arc $A_+$ or underpass
arc $A_-$ as shown in Figure \ref{fig:s3-lp-lm}.
The links $L_{T^+}'$ and $L_{T^-}'$ are isotopic, since one can
isotope $A_+$ to $A_-$ by moving $A_+$ around the $S^2$ factor, and
we will let $L_{T^0}$ denote either of these isotopic links.
For links in $S^2 \times S^1$, an overpass arc $A_+$ and an underpass
$A_-$ correspond to the same unperturbed Lagrangian $W_0$ in
$R^*(T^2,2)$.
Recall from Section \ref{ssec:a-infinity} that given a Lagrangian $M$
in $R^*(T^2,2)$ we can define an $A_\infty$ functor
$\Tw \calG_M:\Tw \calL \rightarrow \Ch$.
We define a cochain complex $(C_0,\partial_0)$ by applying the functor
$\Tw \calG_{W_0}$ to the twisted complex $(X,\delta_-)$:
\begin{align*}
  (C_0, \partial_0) = (\Tw \calG_{W_0})((X,\delta_-)).
\end{align*}
Our goal in this section is to explicitly describe this cochain complex
and to investigate its cohomology.
Given a tangle diagram $T$, we define a bigraded vector space that we
call the \emph{generalized reduced Khovanov homology of $T$}:
\begin{align*}
  \Khr^{(h,q)}(S^2 \times S^1, T) = H^{(h,h+q)}((C_0,\partial_0)).
\end{align*}

One might hope that the generalized reduced Khovanov homology depends
only on the isotopy class of the link $L_{T^0}$ and not on its
description as the closure of the particular tangle diagram $T$.
From Theorem \ref{theorem:invar}, it follows that the generalized
reduced Khovanov homology is an isotopy invariant of the tangle diagram
$T$.
But it is possible for nonisotopic tangle diagrams $T_1$ and $T_2$ to
yield isotopic links $L_{T_1^0}$ and $L_{T_2^0}$ in $S^2 \times S^1$,
and example calculations show that in this situation the generalized
reduced Khovanov homology for $T_1$ and $T_2$ need not be the same.
In all the examples we have checked, however, the dependence on the
tangle diagram is reflected only in the bigradings of generators, and
the generalized reduced Khovanov homology does agree if we collapse
bigradings from $\Ints$ to $\Ints_2$.
We thus make the following conjecture:

\begin{conjecture}
\label{conj:khr-gen}
If $L_{T_1^0}$ and $L_{T_2^0}$ are isotopic links in $S^2 \times S^1$,
then with bigradings collapsed to $\Ints_2$ we have
\begin{align*}
  \Khr(S^2 \times S^1, T_1) = \Khr(S^2 \times S^1, T_2).
\end{align*}
\end{conjecture}

We will focus on links obtained from a tangle diagram $T$ with loop
number 0 or 2.
To compute the cochain complexes for such links, we need to know the
bigradings of the generators of the vector spaces $(W_0,L_n)$ for
$n \in \{0, \pm 2\}$ and the $A_\infty$ operations involving these
generators.

\subsection{$A_\infty$ operations and gradings of generators}

Recall from Corollary \ref{cor:gens-s2-s1} that the vector spaces
$(W_0,L_0)$ and $(W_0,L_{\pm 2})$ are 2-dimensional.
We assign integer bigradings to the generators as follows:
\begin{align*}
  (W_0, L_0) &=
  \langle \alpha_0^{(0,0)},\, \beta_0^{(0,-2)} \rangle, &
  (W_0, L_{\pm 2}) &=
  \langle \sigma_{\pm 2, 0}^{(0,-1)},\, \tau_{\pm 2, 0}^{(0,-1)} \rangle.
\end{align*}
Based on Conjecture \ref{conj:prod-unperturbed} for the product
operations of these generators, we make the following conjecture:

\begin{conjecture}
\label{conj:ops-s2-s1}
We have the following product operations:
\begin{align*}
  &\mu^2(a_0,-):
  (W_0,L_0) \rightarrow (W_0,L_0), &
  &x \mapsto x, \\
  &\mu^2(a_{\pm 2},-):
  (W_0,L_{\pm 2}) \rightarrow (W_0,L_{\pm 2}), &
  &x \mapsto x, \\
  &\mu^2(c_0,-):
  (W_0,L_0) \rightarrow (W_0,L_0), &
  &\alpha_0 \mapsto \beta_0, &
  &\beta_0 \mapsto 0, \\
  &\mu^2(c_{\pm 2},-):
  (W_0,L_{\pm 2}) \rightarrow (W_0,L_{\pm 2}), &
  &x \mapsto 0, \\
  &\mu^2(p_{2,0}^{(-)},-):
  (W_0,L_0) \rightarrow (W_0,L_2), &
  &\alpha_0 \mapsto \tau_{2,0}, &
  &\beta_0 \mapsto 0, \\
  &\mu^2(p_{2,0}^{(+)},-):
  (W_0,L_0) \rightarrow (W_0,L_2), &
  &\alpha_0 \mapsto 0, &
  &\beta_0 \mapsto \sigma_{2,0}, \\
  &\mu^2(q_{2,0}^{(+)},-):
  (W_0,L_0) \rightarrow (W_0,L_2), &
  &\alpha_0 \mapsto \sigma_{-2,0}, &
  &\beta_0 \mapsto 0, \\
  &\mu^2(q_{2,0}^{(-)},-):
  (W_0,L_0) \rightarrow (W_0,L_2), &
  &\alpha_0 \mapsto \tau_{-2,0}, &
  &\beta_0 \mapsto 0, \\
  &\mu^2(p_{0,-2}^{(-)},-):
  (W_0,L_{-2}) \rightarrow (W_0,L_0), &
  &\sigma_{-2,0} \mapsto \beta_0, &
  &\tau_{-2,0} \mapsto 0, \\
  &\mu^2(p_{0,-2}^{(+)},-):
  (W_0,L_{-2}) \rightarrow (W_0,L_0), &
  &\sigma_{-2,0} \mapsto 0, &
  &\tau_{-2,0} \mapsto \beta_0, \\
  &\mu^2(q_{0,2}^{(+)},-):
  (W_0,L_2) \rightarrow (W_0,L_0), &
  &\sigma_{2,0} \mapsto 0, &
  &\tau_{2,0} \mapsto \beta_0, \\
  &\mu^2(q_{0,2}^{(-)},-):
  (W_0,L_2) \rightarrow (W_0,L_0), &
  &\sigma_{2,0} \mapsto \beta_0, &
  &\tau_{2,0} \mapsto 0.
\end{align*}
All operations $\mu^m(x_{m-1}, \cdots, x_1, y)$ for $m \neq 2$,
$x_{m-1}, \cdots, x_1 \in
\{a_n,\, c_n,\, p_{n+2,n},\, q_{n-2,n} \mid n \in \Ints \}$, and
$y \in (W_0, L_0)$ or $y \in (W_0, L_{\pm 2})$ are zero.
\end{conjecture}

We can depict the product operations in
Conjecture \ref{conj:ops-s2-s1} as
\begin{eqnarray*}
\begin{tikzcd}
  \sigma_{-2,0}^{(-)} \arrow{dr}[swap]{p_{0,-2}^{(-)}} &[2 em]
  \alpha_0^{(-)} \arrow{d}{c_0^{(-)}} \arrow{dr}{p_{2,0}^{(-)}}
  \arrow{l}[swap]{q_{-2,0}^{(+)}} &[2 em]
  \sigma_{2,0}^{(-)} \\
  \tau_{-2,0}^{(+)} &[2 em]
  \beta_0^{(+)} &[2 em]
  \tau_{2,0}^{(+)} \arrow{l}[swap]{q_{0,2}^{(+)}}
\end{tikzcd}
&&
\begin{tikzcd}
  \sigma_{-2,0}^{(-)} &[2 em]
  \alpha_0^{(-)}
  \arrow{d}{c_0^{(-)}}
  \arrow{r}{p_{2,0}^{(+)}}
  \arrow{dl}[swap]{q_{-2,0}^{(-)}} &[2 em]
  \sigma_{2,0}^{(-)}
  \arrow{dl}{q_{0,2}^{(-)}} \\
  \tau_{-2,0}^{(+)}
  \arrow{r}[swap]{p_{0,-2}^{(+)}} &[2 em]
  \beta_0^{(+)} &[2 em]
  \tau_{2,0}^{(+)}.
\end{tikzcd}
\end{eqnarray*}

There is a sense in which the product operations in Conjecture
\ref{conj:ops-s2-s1} are consistent with
the product operations in Conjecture \ref{conj:prod-perturbed} for
perturbed Lagrangians.
We define a bigraded vector space
\begin{align*}
  V &= \langle e^{(0,0)},\, v^{(0,0)} \rangle.
\end{align*}
We define an isomorphism
$(W_0, L_n) \otimes V \rightarrow (L_0, L_n)$ of bigraded vector
spaces:
\begin{align*}
  \alpha_0 \otimes e &\mapsto a_0, &
  \alpha_0 \otimes v &\mapsto b_0, &
  \beta_0 \otimes e &\mapsto c_0, &
  \beta_0 \otimes v &\mapsto d_0, \\
  \sigma_{2,0} \otimes e &\mapsto s_{2,0} + \sbar_{2,0}, &
  \sigma_{2,0} \otimes v &\mapsto r_{2,0}, &
  \tau_{2,0} \otimes e &\mapsto r_{2,0} + \rbar_{2,0}, &
  \tau_{2,0} \otimes v &\mapsto \sbar_{2,0}, \\
  \sigma_{-2,0} \otimes e &\mapsto s_{-2,0} + \sbar_{-2,0}, &
  \sigma_{-2,0} \otimes v &\mapsto r_{-2,0}, &
  \tau_{-2,0} \otimes e &\mapsto r_{-2,0} + \rbar_{-2,0}, &
  \tau_{-2,0} \otimes v &\mapsto s_{-2,0}.
\end{align*}
Using the product operations described in Conjecture
\ref{conj:prod-perturbed} and Conjecture \ref{conj:ops-s2-s1}, it is
straightforward to check that we have the following commutative
diagrams:

\begin{eqnarray*}
\begin{tikzcd}
  (W_0, L_0) \otimes V
  \arrow{d}{\cong}
  \arrow{r}{\mu^2(p_{2,0}^{(\pm)},-) \otimes 1_V} &[5em]
  (W_0, L_2) \otimes V
  \arrow{d}{\cong} \\
  (L_0, L_0) 
  \arrow{r}{\mu_\calL^2(p_{2,0}^{(\pm)},-)} &[5em]
  (L_0, L_2),
\end{tikzcd}
&&
\begin{tikzcd}
  (W_0, L_2) \otimes V
  \arrow{d}{\cong}
  \arrow{r}{\mu^2(q_{0,2}^{(\pm)},-) \otimes 1_V} &[5em]
  (W_0, L_0) \otimes V
  \arrow{d}{\cong} \\
  (L_0, L_2) 
  \arrow{r}{\mu_\calL^2(q_{0,2}^{(\pm)},-)} &[5em]
  (L_0, L_0),
\end{tikzcd}
\\
\begin{tikzcd}
  (W_0, L_0) \otimes V
  \arrow{d}{\cong}
  \arrow{r}{\mu^2(q_{-2,0}^{(\pm)},-) \otimes 1_V} &[5em]
  (W_0, L_{-2}) \otimes V
  \arrow{d}{\cong} \\
  (L_0, L_0) 
  \arrow{r}{\mu_\calL^2(q_{-2,0}^{(\pm)},-)} &[5em]
  (L_0, L_{-2}),
\end{tikzcd}
&&
\begin{tikzcd}
  (W_0, L_{-2}) \otimes V
  \arrow{d}{\cong}
  \arrow{r}{\mu^2(p_{0,-2}^{(\pm)},-) \otimes 1_V} &[5em]
  (W_0, L_0) \otimes V
  \arrow{d}{\cong} \\
  (L_0, L_{-2}) 
  \arrow{r}{\mu_\calL^2(p_{0,-2}^{(\pm)},-)} &[5em]
  (L_0, L_0),
\end{tikzcd}
\\
\begin{tikzcd}
  (W_0, L_0) \otimes V
  \arrow{d}{\cong}
  \arrow{r}{\mu^2(c_0,-) \otimes 1_V} &[5em]
  (W_0, L_0) \otimes V
  \arrow{d}{\cong} \\
  (L_0, L_0) 
  \arrow{r}{\mu_\calL^2(c_0,-)} &[5em]
  (L_0, L_0),
\end{tikzcd}
&&
\begin{tikzcd}
  (W_0, L_{\pm 2}) \otimes V
  \arrow{d}{\cong}
  \arrow{r}{\mu^2(c_{\pm 2},-) \otimes 1_V} &[5em]
  (W_0, L_{\pm 2}) \otimes V
  \arrow{d}{\cong} \\
  (L_0, L_{\pm 2}) 
  \arrow{r}{\mu_\calL^2(c_{\pm 2},-)} &[5em]
  (L_0, L_{\pm 2}).
\end{tikzcd}
\end{eqnarray*}

\subsection{Cochain complex}

Using Conjecture \ref{conj:ops-s2-s1}, the description of the
functor $\calG_{W_0}$ from Section \ref{ssec:a-infinity}, and the
description of the twisted complex $(X,\delta_-)$ from Section
\ref{ssec:twisted-complex}, we can explicitly construct the cochain
complex $(C_0,\partial_0)$ for the tangle diagram $T$.
The vector space $C_0$ is given by
\begin{align*}
  C_0 = \bigoplus_{i \in  I}
  ((W_0,L_{n_i}) \otimes A^{\otimes c_i}[r(i)+h_T, 2r(i)+q_T]).
\end{align*}
The differential $\partial_0$ is a sum of differentials corresponding
to saddles in the cube of resolutions of $T$.
The differential corresponding to a saddle
$S:P_{n_1}(r_1) \rightarrow P_{n_2}(r_2)$ is given by
\begin{align*}
  &\partial_S:(W_0,L_{n_1}) \otimes A^{\otimes r_1} \rightarrow
  (W_0,L_{n_2}) \otimes A^{\otimes r_2}[1,2], &
  \partial_S = \mu_{\Sigma \calL}^2(d_S[1,2], -).
\end{align*}

We could also define a cochain complex by applying the functor
$\calG_{W_0}$ to the twisted object $(X,\delta_+)$, but given the
$A_\infty$ operations in Conjecture \ref{conj:ops-s2-s1} the resulting
complex is homotopy equivalent to $(C_0,\partial_0)$.

We obtain a nontrivial cochain complex only when the loop number of the
tangle diagram $T$ is even, due to the following result:

\begin{theorem}
For a tangle diagram $T$ with odd loop number, the cochain complex
$(C_0,\partial_0)$ vanishes.
\end{theorem}

\begin{proof}
For odd loop number, the morphism spaces that appear in $C_0$ have the
form $(W_0,L_n)$ for $n$ odd.
But $(W_0,W_n)$ vanishes for odd $n$, since there are no traceless
$SU(2)$ representations for a link that winds around the solid torus
an odd number of times, hence $(W_0,L_n)$ vanishes.
\end{proof}

\subsection{Example: double cycle in $S^2 \times S^1$}
\label{ssec:example-s2-s1-1}

We can close the tangle diagram $T$ shown in Figure
\ref{fig:example-tangle} with an overpass or underpass arc to obtain a
double cycle $L_{T^0}$ in $S^2 \times S^1$.
The cochain complex $(C_0,\partial_0)$ is given by
\begin{eqnarray*}
\begin{tikzcd}
  {} & (W_0, L_0) \otimes \F[1,4]
  \arrow{dr}{\partial_{L^+}} & {} \\
  (W_0, L_2) \otimes \F[0,2]
  \arrow{ur}{\partial_{W^-}}
  \arrow{dr}[swap]{\partial_{W^-}} &
  {} & (W_0, L_0) \otimes A[2,6], \\
  {} & (W_0, L_0) \otimes \F[1,4]
  \arrow{ur}[swap]{\partial_{R^+}} & {}
\end{tikzcd}
\end{eqnarray*}
where
\begin{align*}
  &\partial_{W^-}:
  (W_0,L_2) \otimes \F \rightarrow (W_0,L_0) \otimes \F, &
  &\sigma_{2,0} \otimes 1 \mapsto \beta_0 \otimes 1, &
  &\tau_{2,0} \otimes 1 \mapsto 0, \\
  &\partial_{L^+}:
  (W_0,L_0) \otimes \F \rightarrow (W_0,L_0) \otimes A, &
  &\alpha_0 \otimes 1 \mapsto \alpha_0 \otimes x, &
  &\beta_0 \otimes 1 \mapsto \beta_0 \otimes x, \\
  &\partial_{R^+}:
  (W_0,L_0) \otimes \F \rightarrow (W_0,L_0) \otimes A, &
  &\alpha_0 \otimes 1 \mapsto \alpha_0 \otimes x + \beta_0 \otimes e, &
  &\beta_0 \otimes 1 \mapsto \beta_0 \otimes x.
\end{align*}
So the generalized reduced Khovanov homology of $T$ is
\begin{align*}
  \Khr(S^2 \times S^1, T) = \F[0,1] \oplus \F[2,5].
\end{align*}

The link $L_{T^0}$ is isotopic to the link $L_{P_2^0}$ obtained by
closing the planar tangle $P_2$, oriented in either direction, with an
overpass or underpass arc.
The cochain complex for the tangle diagram $P_2$ is
\begin{align*}
  (W_0,L_2) \otimes \F[0,0]
\end{align*}
with zero differential, so the generalized reduced Khovanov homology
of $P_2$ is
\begin{align*}
  \Khr(S^2 \times S^1, P_2) = 2\F[0,-1].
\end{align*}
So if we collapse bigradings to $\Ints_2$, we have
\begin{align*}
  \Khr(S^2 \times S^1, T) =
  \Khr(S^2 \times S^1, P_2) =
  2\F[0,1],
\end{align*}
as is consistent with Conjecture \ref{conj:khr-gen}.

\subsection{Example: two cycles and unknot in $S^2 \times S^1$}
\label{ssec:example-s2-s1-2}

Consider the 1-tangle diagrams $T_1$ and $T_2$ shown in Figure
\ref{fig:example-tangle-pair}.
The corresponding links $L_{T_1^0}$ and $L_{T_2^0}$ in
$S^2 \times S^1$ are isotopic, and describe two cycles oriented in the
same direction and an unlinked unknot.
We find
\begin{align*}
  \Khr(S^2 \times S^1, T_1) &=
  \F[0,1] \oplus \F[1,1] \oplus \F[1,3] \oplus \F[2,3], \\
  \Khr(S^2 \times S^1, T_2) &=
  \F[0,-1] \oplus \F[0,1] \oplus \F[1,1] \oplus \F[1,3].
\end{align*}
So if we collapse bigradings to $\Ints_2$, we have
\begin{align*}
  \Khr(S^2 \times S^1, T_1) =
  \Khr(S^2 \times S^1, T_2) =
  2\F[0,1] \oplus 2\F[1,1],
\end{align*}
as is consistent with Conjecture \ref{conj:khr-gen}.

\begin{figure}
  \centering
  \includegraphics[scale=0.5]{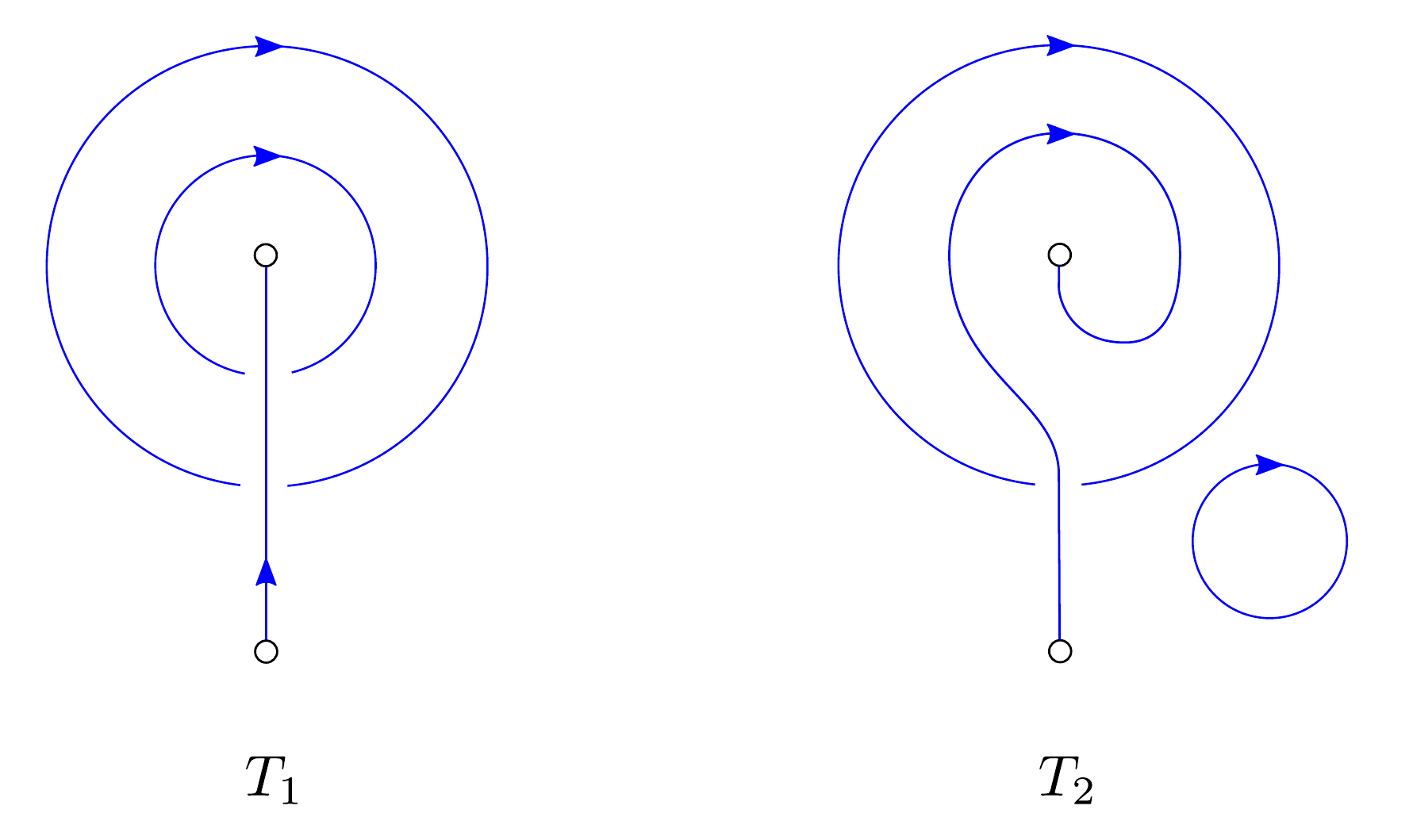}
  \caption{
    \label{fig:example-tangle-pair}
    Example tangle diagrams $T_1$ and $T_2$.
    The links $L_{T_1^0}$ and $L_{T_2^0}$ in $S^2 \times S^1$ are
    isotopic.
  }
\end{figure}

We define a 1-tangle diagram $T_3$ by flipping the orientation of one
of the two circles in $T_1$ that loops around the annulus.
We define 1-tangle diagram $T_4$ by flipping the orientation of the
single circle in $T_2$ that loops around the annulus.
The corresponding links $L_{T_3^0}$ and $L_{T_4^0}$ in
$S^2 \times S^1$ are isotopic, and describe two cycles oriented in the
opposite direction and an unlinked unknot.
We find
\begin{align*}
  \Khr(S^2 \times S^1, T_3) &=
  \F[-1,-2] \oplus \F[0,-2] \oplus \F[0,0] \oplus \F[1,0], \\
  \Khr(S^2 \times S^1, T_4) &=
  \F[-1,-4] \oplus \F[-1,-2] \oplus \F[0,-2] \oplus \F[0,0].
\end{align*}
So if we collapse bigradings to $\Ints_2$, we have
\begin{align*}
  \Khr(S^2 \times S^1, T_3) =
  \Khr(S^2 \times S^1, T_4) =
  2\F[1,0] \oplus 2\F[0,0],
\end{align*}
as is consistent with Conjecture \ref{conj:khr-gen}.

\subsection{Example: double cycle and unknot in $S^2 \times S^1$}
\label{ssec:example-s2-s1-3}

Consider the 1-tangle diagrams $T_1$ and $T_2$ shown in Figure
\ref{fig:example-tangle-pair-2}.
The corresponding links $L_{T_1^0}$ and $L_{T_2^0}$ in
$S^2 \times S^1$ are isotopic, and describe a double cycle and an
unlinked unknot.
We find
\begin{align*}
  \Khr(S^2 \times S^1, T_1) &=
  \F[0,2] \oplus 2\F[2,4] \oplus \F[2,6], \\
  \Khr(S^2 \times S^1, T_2) &=
  2\F[0,-2] \oplus 2\F[0,0].
\end{align*}
So if we collapse gradings to $\Ints_2$ we have
\begin{align*}
  \Khr(S^2 \times S^1, T_1) =
  \Khr(S^2 \times S^1, T_2) =
  4\F[0,0],
\end{align*}
as is consistent with Conjecture \ref{conj:khr-gen}.

\begin{figure}
  \centering
  \includegraphics[scale=0.5]{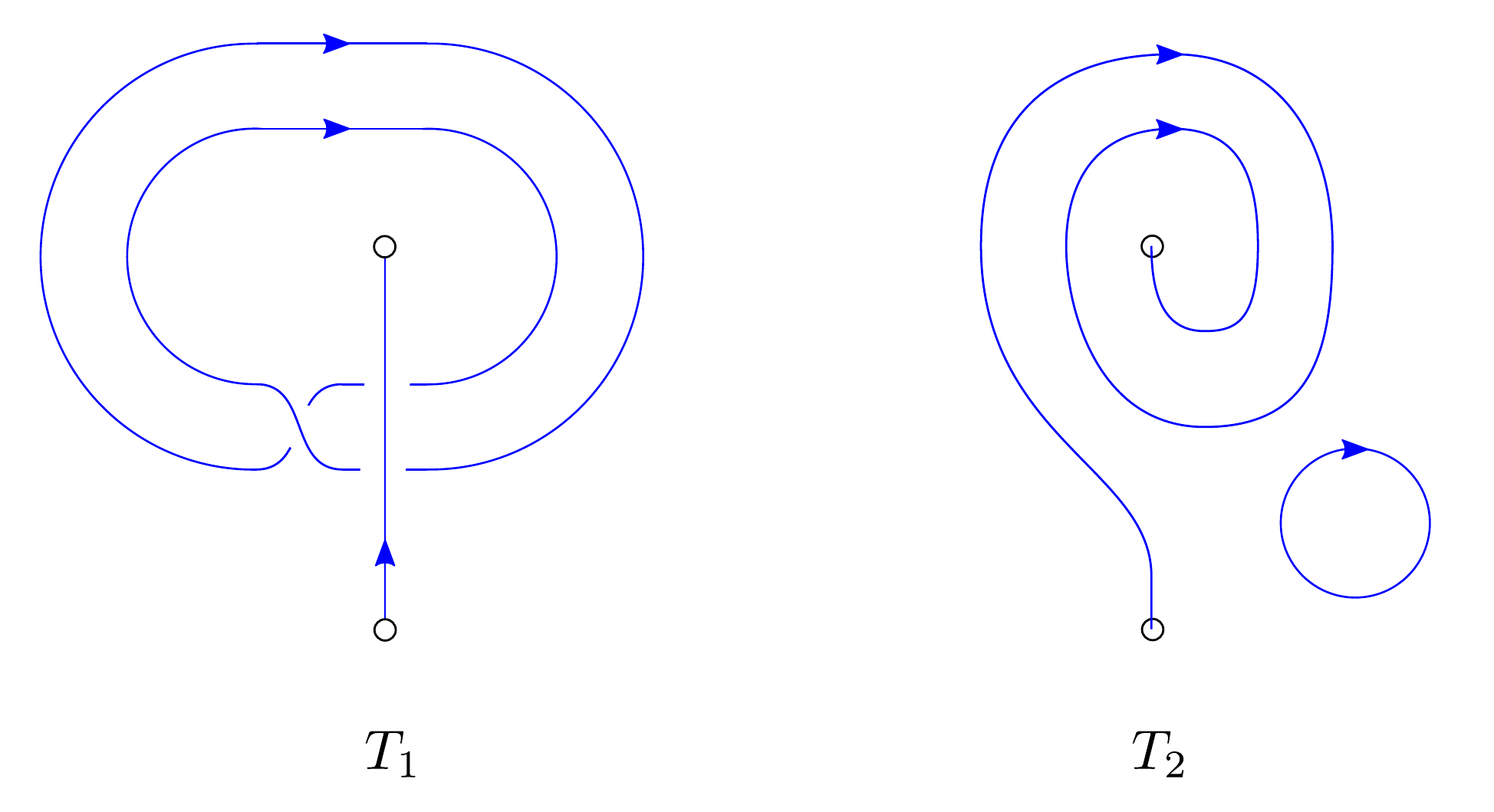}
  \caption{
    \label{fig:example-tangle-pair-2}
    Example tangle diagrams $T_1$ and $T_2$.
    The links $L_{T_1^0}$ and $L_{T_2^0}$ in $S^2 \times S^1$ are
    isotopic.
  }
\end{figure}

\subsection{Tangle diagrams with loop number 0}

Consider a tangle diagram $T$ with loop number 0.
The corresponding link $L_{T^0}$ is contained inside an open 3-ball
in $S^2 \times S^1$.
By collapsing the complement of the open 3-ball to a point, we obtain
a link $L_{T^0}'$ in $S^3$.
We can also view $L_{T^0}'$ as the link obtained by closing $T$ with
either an overpass arc $A_+$ or an underpass arc $A_-$, as shown in
Figure \ref{fig:s3-lp-lm}, and interpreting the resulting diagram as
describing a link in $S^3$.
In this section we relate the generalized reduced Khovanov homology of
the tangle diagram $T$ to the reduced Khovanov homology of the link
$L_{T^0}'$ in $S^3$.

Let $(X,\delta)$ denote the twisted complex corresponding to the
tangle diagram $T$.
We can define a new twisted complex $(X,\delta')$ by modifying the
differential $\delta$ as follows.
Recall that $\delta$ is a sum of differentials $d_{ji}$ corresponding
to saddles in the cube of resolution of $T$.
The differentials $d_{ji}$ have the form $d_{C^\pm}$, $d_{R^\pm}$, or
$d_{L^\pm}$, as defined in Section \ref{ssec:saddles}.
Since $T$ has loop number 0, there are no differentials of
the form $d_{W^\pm}$.
Define new differentials $d_{ji}'$ by eliminating the terms in
$d_{ji}$ that contain $c_0$.
That is, we replace differentials of the form $d_{C^\pm}$,
$d_{R^\pm}$, and $d_{L^\pm}$ by $d_{C^\pm}'$,
$d_{R^\pm}'$, and $d_{L^\pm}'$, where
\begin{align*}
  &d_{C^+}' = d_{C^+} = a_0 \otimes \Delta I_{r-1}, &
  &d_{C^-}' = d_{C^-} = a_0 \otimes m I_{r-1}, \\
  &d_{R^+}' = d_{L^+}' = a_0 \otimes \etadot I_r, &
  &d_{R^-}' = d_{L^-}' = a_0 \otimes \epsilondot I_r.
\end{align*}
We then define $\delta'$ to be the sum of the new differentials
$d_{ji}'$.
Let $(C_T,d_T)$ denote the cochain complex for the reduced Khovanov
homology of $L_{T^0}'$, where the marked component is the one
containing the arc component of $T$, and with the convention that
$d_T$ has quantum grading 1, rather than the usual quantum grading of
0.
Given the construction of $X$ and our definition $\delta'$, we have
that $X = L_0 \otimes C_T$ and $\delta' = a_0 \otimes d_T$.

\begin{theorem}
Assuming Conjecture \ref{conj:fukaya} for the $A_\infty$ operations of
$\calL$, the pair $(X,\delta')$ is a twisted complex.
\end{theorem}

\begin{proof}
Assuming Conjecture \ref{conj:fukaya} for the $A_\infty$ operations of
$\calL$, we have
\begin{align*}
  \sum_{m=1}^\infty \mu_{\Sigma \calL}^m(\delta', \cdots, \delta') =
  \mu_{\Sigma \calL}^2(\delta', \delta') =
  a_0 \otimes (d_T \circ d_T) = 0.
\end{align*}
We can define a filtration on $X$ using the resolution degree.
The differential $\delta'$ increases the resolution degree by 1, so it
is lower triangular with respect to this filtration.
\end{proof}

\begin{theorem}
\label{theorem:loop-0-complexes}
Assuming Conjecture \ref{conj:fukaya} for the $A_\infty$ operations of
$\calL$, the twisted complexes $(X,\delta)$ and $(X,\delta')$ are
homotopy equivalent.
\end{theorem}

\begin{proof}
Each term of $X$ corresponds to a planar tangle $P_0(k)$ obtained by
resolving the nonplanar tangle $T$, and can be viewed as a subobject
$(x,0)$ of $(X,\delta)$ with zero differential:
\begin{eqnarray*}
(x,0) =
\begin{tikzcd}
  L_n \otimes A^{\otimes k}[h,q],
\end{tikzcd}
\end{eqnarray*}
where the bigrading shift $[h,q]$ depends on the crossing numbers
of $T$ and on the location of $P_0(k)$ within the cube of resolutions.
We have a corresponding identical subobject $(x',0)$ of
$(X,\delta')$.

For planar tangles that are obtained by resolving a tangle diagram
with loop number 0, we can distinguish between circle components to
the left and to the right of the arc component.
Let $\ell$ and $r$ denote the number of circle components to the left
and right of the arc component of $P_0(k)$, so $k = \ell + r$.
Define homotopy equivalences
$F:(x,0) \rightarrow (x',0)$ and
$G:(x',0) \rightarrow (x,0)$:
\begin{align*}
  F = G = a_0 I_{\ell+r} + c_0 I_\ell \Sigma_r.
\end{align*}
We have
\begin{align*}
  F \circ G = G \circ F = a_0 I_{\ell + r}.
\end{align*}
Since the differentials are zero for the subobjects $(x,0)$ and
$(x',0)$, it is clear that $F$ and $G$ are homotopy equivalences.

We combine the homotopy equivalences for subobjects to obtain
morphisms $(X,\delta) \rightarrow (X,\delta')$ and
$(X,\delta') \rightarrow (X,\delta)$.
To show that these morphisms are are homotopy equivalences, we must
check that they commute with the differentials $\delta$ and $\delta'$.
The differential $\delta$ the sum of differentials
$d:x_1 \rightarrow x_2$
corresponding to saddles.
For each such term, there is a corresponding term
$d':x_1' \rightarrow x_2'$ of
$\delta'$.
We want to show that the following diagram commutes:
\begin{eqnarray*}
\begin{tikzcd}
  x_1' \arrow{d}[shift left=4]{G_1} \arrow{r}{d'} &
  x_2' \arrow{d}[shift left=4]{G_2} \\
  x_1 \arrow[shift left=4]{u}{F_1} \arrow{r}{d} &
  x_2. \arrow[shift left=4]{u}{F_2}
\end{tikzcd}
\end{eqnarray*}
For a differential $d$ of the form $d_{C^\pm}$, commutativity follows
from the identities in equation (\ref{eqn:identities-sigma}).
For a differential $d$ of the form $d_{R^+}$, we have
\begin{align*}
  d &= d_{R^+} =
  a_0 I_{\ell+r} \etadot + c_0 I_{\ell+r} \eta +
  c_0 I_{\ell_n} \Sigma_{\ell_e} I_r \etadot, &
  d' &= d_{R^+}' =
  a_0 I_{\ell+r} \etadot, \\
  F_1 &= G_1 =
  a_0 I_{\ell+r} + c_0 I_\ell \Sigma_r, &
  F_2 &= G_2 =
  a_0 I_{\ell+r+1} + c_0 I_{\ell_n} \Sigma_{\ell_e + r + 1}.
\end{align*}
We check:
\begin{align*}
  &d' \circ F_1 = F_2 \circ d =
  a_0 I_{\ell+r} \etadot + c_0 I_\ell \Sigma_r \etadot, \\
  &d \circ G_1 = G_2 \circ d' =
  a_0 I_{\ell+r} \etadot + c_0 I_{\ell+r} \eta +
  c_0 I_{\ell_n} \Sigma_{\ell_e + r} \etadot.
\end{align*}
Similar calculations show commutativity for differentials of the form
$d_{R^-}$ and $d_{L^{\pm}}$
\end{proof}

The generalized reduced Khovanov homology of the tangle diagram $T$ is
related to the reduced Khovanov homology of the link $L_{T^0}'$ in
$S^3$ by the following result:

\begin{theorem}
\label{theorem:khr-s2-s1}
We have an isomorphism of bigraded vector spaces
\begin{align*}
  \Khr(S^2 \times S^1,T) = (W_0, L_0) \otimes \Khr(L_0').
\end{align*}
\end{theorem}

\begin{proof}
This follows from Theorem \ref{theorem:loop-0-complexes} and the fact
that
\begin{align*}
  \calG_{W_0}((X,\delta')) =
  ((W_0, L_0) \otimes C_T,\, \id_{(W_0,L_0)} \otimes d_T).
\end{align*}
\end{proof}

Theorem \ref{theorem:khr-s2-s1} shows that if the generalized reduced
Khovanov homology with bigradings collapsed to $\Ints_2$ is indeed a
link invariant, then it is a natural generalization of reduced
Khovanov homology for links in $S^3$.
Since the reduced Khovanov homology of an unknot in
$S^3$ is $\F[0,0]$, we could view $(W_0,L_0)$ as additional cohomology
for $S^2 \times S^1$ that describes the topology of this space.

\subsection{Tangle diagrams with loop number 2}

Given a 1-tangle diagram $T$ with loop number 2, we can define a
1-tangle diagram $T_\tau$ by adding a full twist to $T$ as shown in
Figure \ref{fig:diagram-twist}.
The diagrams $T$ and $T_\tau$ are typically not isotopic, but the
corresponding links $L_{T_0}$ and $L_{T_\tau^0}$ in $S^2 \times S^1$
are isotopic, since one can unwind the full twist by moving one of the
strands around the $S^2$ factor as shown in Figure
\ref{fig:s2-s1-twist}.
Our goal in this section is to prove:

\begin{theorem}
\label{theorem:twist}
With bigradings collapsed from $\Ints$ to $\Ints_2$, there is an
isomorphism of bigraded vector spaces
\begin{align*}
  \Khr(S^2 \times S^1, T) \rightarrow \Khr(S^2 \times S^1, T_\tau).
\end{align*}
\end{theorem}

\begin{proof}
Let $(C_0,\partial_0)$ and $(C_0^\tau, \partial_0^\tau)$ denote the
cochain complexes for the tangle diagrams $T$ and $T_\tau$.
The 1-tangle diagram $T$ in the annulus can be viewed as a 3-tangle
diagram $T_D$ in the disk that has been closed with two arcs wrapping
around the annulus.
We can thus describe the planar resolution of $T$ in terms of the
planar resolution of $T_D$.
The vector space $C_0$ is a direct sum of subspaces corresponding
to planar tangles that occur in this resolution.
Each subspace corresponds to a planar 3-tangle $P_D$ obtained by
resolving $T_D$, and can be viewed as a subcomplex $(c,0)$ of
$(C_0,\partial_0)$ with zero differential.
Let $P_n(r)$ denote the planar 1-tangle in the annulus obtained by
closing $P_D$ with two arcs that wrap around the annulus.
The subcomplex $(c,0)$ is then given by
\begin{eqnarray*}
  (c,0) &=
\begin{tikzcd}
  (W_0, L_n) \otimes A^{\otimes r}[h,q],
\end{tikzcd}
\end{eqnarray*}
where the bigrading shift $[h,q]$ depends on the crossing numbers of
$T$ and on the location of $P_D$ within the cube of resolutions.
If we apply a full twist to $P_n(r)$ and then resolve the two
resulting crossings, we obtain the cube of resolutions shown in Figure
\ref{fig:resolve-twist}:
\begin{eqnarray*}
\begin{tikzcd}
  {} & P_m(s)
  \arrow{dr}{S_1'} & {} \\
  P_n(r)
  \arrow{ru}{S_2}
  \arrow{rd}[swap]{S_1} & {} &
  P_m(s+1), \\
  {} & P_m(s)
  \arrow{ur}[swap]{S_2'}
\end{tikzcd}
\end{eqnarray*}
where the planar tangle $P_m(s)$ depends on $P_D$.
Thus for each such subcomplex $(c,0)$ of $(C_0,\partial_0)$, there is
a corresponding subcomplex $(c_\tau,d)$ of $(C_0^\tau,\partial_0^\tau)$
given by
\begin{eqnarray*}
  (c_\tau, d) &=
\begin{tikzcd}
  {} & (W_0, L_m) \otimes A^{\otimes s}[h_\tau+1,q_\tau]
  \arrow{dr}{\partial_{S_1'}} & {} \\
  (W_0, L_n) \otimes A^{\otimes r}[h_\tau,q_\tau]
  \arrow{ru}{\partial_{S_2}}
  \arrow{rd}[swap]{\partial_{S_1}} & {} &
  (W_0, L_m) \otimes A^{\otimes(s+1)}[h_\tau,q_\tau], \\
  {} & (W_0, L_m) \otimes A^{\otimes s}[h_\tau+1,q_\tau]
  \arrow{ur}[swap]{\partial_{S_2'}}
\end{tikzcd}
\end{eqnarray*}
where we have collapsed bigradings to $\Ints_2$.
The relative crossing numbers of the tangle diagrams $T_\tau$ and $T$
are either $\delta m_+ = 2$ and $\delta m_- = 0$ or
$\delta m_+ = 0$ and $\delta m_- = 2$.
In either case, we have
\begin{align*}
  [h_\tau, q_\tau] =
  [h,q] + [-\delta m_-,\, \delta m_+ - 3 \delta m_-] = [0,0].
\end{align*}
Since we are collapsing homological gradings to $\Ints_2$, we can
express the subcomplex $(c_\tau,d)$ as
\begin{eqnarray*}
  (c_\tau, d) &=
\begin{tikzcd}
\left(\begin{array}{c}
  (W_0, L_m) \otimes A^{\otimes(s+1)}[h,q] \\
  (W_0, L_n) \otimes A^{\otimes r}[h,q]
\end{array}\right)
\arrow[shift left=2]{r}{d^0}  &
\left(\begin{array}{c}
  (W_0, L_m) \otimes A^{\otimes s}[h+1,q] \\
  (W_0, L_m) \otimes A^{\otimes s}[h+1,q]
\end{array}\right),
\arrow[shift left=2]{l}{d^1}
\end{tikzcd}
\end{eqnarray*}
where
\begin{align*}
  d^0 &=
  \left(\begin{array}{cc}
    0 & \partial_{S_2} \\
    0 & \partial_{S_1} \\
  \end{array}
  \right), &
  d^1 &=
  \left(\begin{array}{cc}
    \partial_{S_1'} & \partial_{S_2'} \\
    0 & 0
  \end{array}
  \right).
\end{align*}

\begin{figure}
  \centering
  \includegraphics[scale=0.40]{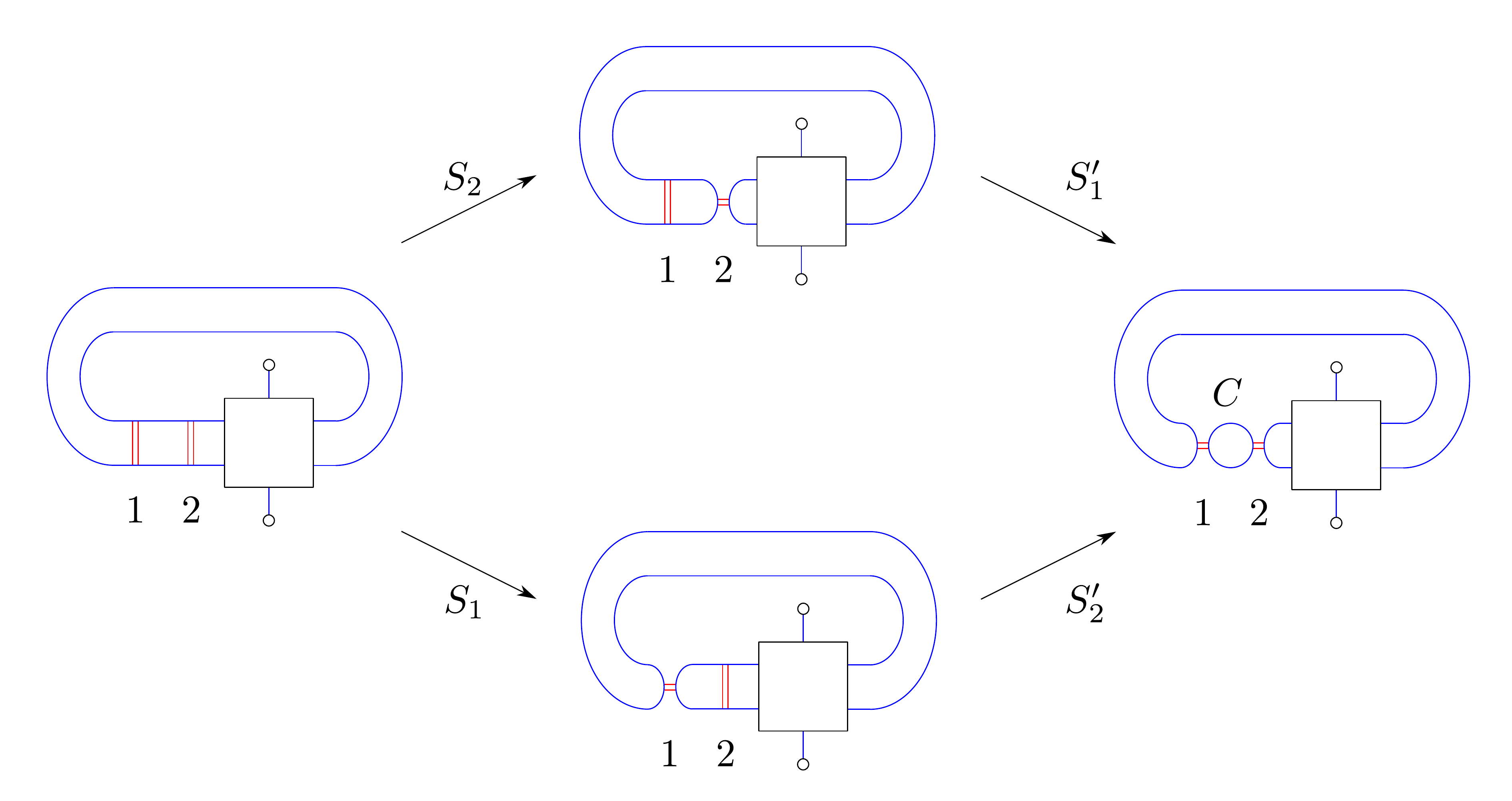}
  \caption{
    \label{fig:resolve-twist}
    Resolution of a full twist.
  }
\end{figure}

For each isotopy class of planar 3-tangle $P_D$ in the disk, we define
homotopy equivalences $F:(c,0) \rightarrow (c_\tau,d)$ and
$G:(c_\tau,d) \rightarrow (c,0)$ of the corresponding subobjects of
$(c,0)$ and $(c_\tau,d)$:
\begin{eqnarray}
\label{diag:twist-he}
\begin{tikzcd}
\left(\begin{array}{c}
(W_0, L_m) \otimes A^{\otimes(s+1)} \\
(W_0, L_n) \otimes A^{\otimes r}
\end{array}\right)
\arrow[shift left=2]{r}{d^0,H^0} \arrow{d}[shift left=4]{G^0} &[2em]
\left(\begin{array}{c}
(W_0, L_m) \otimes A^{\otimes s} \\
(W_0, L_m) \otimes A^{\otimes s}
\end{array}\right)
\arrow[shift left=2]{l}{d^1,H^1} \\
(W_0, L_n) \otimes A^{\otimes r}. \arrow[shift left=4]{u}{F^0}
\end{tikzcd}
\end{eqnarray}
The specific homotopy equivalences are rather complicated and are
described in Appendix \ref{sec:chain-homotopy-equivalences}.

We combine the homotopy equivalences for subcomplexes to obtain linear
maps $C_0 \rightarrow C_0^\tau$ and $C_0^\tau \rightarrow C_0$.
We now want to show that these maps commute with the differentials
$\partial_0$ and $\partial_0^\tau$.
The differential $\partial_0$ is a sum of terms corresponding to
saddles that occur in the cube of resolutions of the 3-tangle $T_D$.
Given a saddle $S:P_{D1} \rightarrow P_{D2}$, let $d_S$ denote the
corresponding term of $\partial_0$.
We have a corresponding term $\dtilde_S$ of $\partial_0^\tau$ and a
diagram
\begin{eqnarray*}
\begin{tikzcd}
\left(\begin{array}{c}
(W_0, L_{m_1}) \otimes A^{\otimes(s_1+1)} \\
(W_0, L_{n_1}) \otimes A^{\otimes r_1}
\end{array}\right)
\arrow{d}[shift left=4]{G^0}
\arrow{r}{\dtilde_S} &[2em]
\left(\begin{array}{c}
(W_0, L_{m_2}) \otimes A^{\otimes(s_2+1)} \\
(W_0, L_{m_2}) \otimes A^{\otimes r_2}
\end{array}\right)
\arrow{d}[shift left=4]{G^0} \\
(W_0, L_{n_1}) \otimes A^{\otimes r_1}
\arrow{r}{d_S}
\arrow[shift left=4]{u}{F^0} &
(W_0, L_{n_2}) \otimes A^{\otimes r_2}.
\arrow[shift left=4]{u}{F^0}
\end{tikzcd}
\end{eqnarray*}
We show that this diagram commutes for each type of saddle, as
described in Appendix \ref{sec:chain-homotopy-equivalences}.
Thus the maps define homotopy equivalences
$(C_0,\partial_0) \rightarrow (C_0^\tau,\partial_0^\tau)$ and
$(C_0^\tau,\partial_0^\tau) \rightarrow (C_0,\partial_0)$.
We have verified the necessary calculations with Mathematica, since
they are rather involved.
\end{proof}

\begin{appendix}

% Don't list subsections of Appendix in the table-of-contents:
\addtocontents{toc}{\protect\setcounter{tocdepth}{1}}

\section{The Fukaya category of $R^*(S^2,4)$}
\label{sec:pillowcase}

We describe here some results regarding the Fukaya category of
the pillowcase $R^*(S^2,4)$, so they can be compared with
corresponding results for $R^*(T^2,2)$.
These results are either derived in \cite{Hedden-3}, or are easily
obtained using the methods described there.

As described in Section \ref{ssec:pillowcase}, the traceless character
variety $R(S^2,4)$ is a 2-sphere with four reducible points.
One of these four reducible points is special, in that the action of
the mapping class group $\MCG_4(S^2)$ on $R(S^2,4)$ fixes the special
point and permutes the other three reducible points.
We will depict $R^*(S^2,4)$ as a plane with three punctures, where the
point at infinity is the puncture corresponding to the special point.

We define 2-tangles in the disk $T_{-1}$, $T_1$, $T_\times$, and
$T_\timesbar$ as shown in Figure \ref{fig:disk-planar-tangles}.
The planar 2-tangles $T_{\pm 1}$ are analogous to the planar 1-tangles
in the annulus $P_n$ that we defined in Section \ref{ssec:saddles}.
The nonplanar 2-tangles $T_\times$ and $T_\timesbar$ are analogous to
the overpass arc $A_+$ and underpass arc $A_-$ that we use to define
links in $S^3$, since closing $T_{\pm 1}$ with $T_\times$ or
$T_\timesbar$ yields an unknot in $S^3$ just as closing $P_n$ with
$A_+$ or $A_-$ yields an unknot in $S^3$.

\begin{figure}
  \centering
  \includegraphics[scale=0.55]{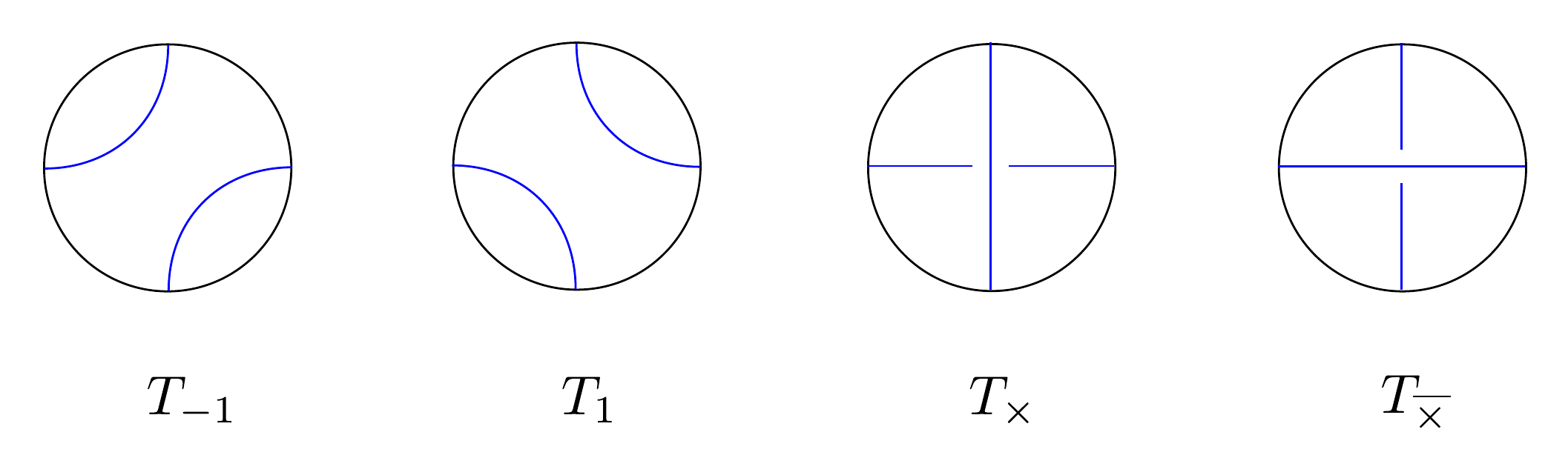}
  \caption{
    \label{fig:disk-planar-tangles}
    Tangles $T_{-1}$, $T_1$, $T_\times$, and $T_\timesbar$ in the disk.
  }
\end{figure}

To the tangle $T_{\pm 1}$ we associate a compact Lagrangian
$L_{\pm 1}$ and to the tangles $T_\times$ and $T_\timesbar$ we
associate noncompact Lagrangians $W_\times$ and $W_\timesbar$ as shown
in Figure \ref{fig:pillowcase-lagrangians}, which should be compared
to Figure \ref{fig:lagrangians} for $R_2^*(T^2,2)$.
As discussed in \cite{Hedden-3}, one can define a Fukaya category
$\calP$ whose objects are the compact Lagrangians $L_{\pm 1}$.
The pillowcase has vanishing first Chern class, and the
Lagrangians $L_{\pm 1}$, $W_\times$, and $W_{\timesbar}$ all have
vanishing Maslov class, so the morphism spaces of $\calP$ carry an
integer Maslov grading.
One can choose graded lifts of the Lagrangians such that the morphism
spaces for $\calP$ are
\begin{align*}
  (L_{\pm 1}, L_{\pm 1}) &=
  \langle a_{\pm 1}^{(0)},\, b_{\pm 1}^{(3)},\,
  c_{\pm 1}^{(-2)},\, d_{\pm 1}^{(1)} \rangle, &
  (L_{-1}, L_1) &= \langle p_{1,-1}^{(-1)},\, q_{1,-1}^{(3)} \rangle, &
  (L_1, L_{-1}) &= \langle p_{-1,1}^{(3)},\, q_{-1,1}^{(-1)} \rangle.
\end{align*}
We note that the Maslov gradings are consistent with Poincar\'{e}
duality, as described at the end of Section \ref{ssec:ops-fukaya}.
The morphism spaces for $\calP$ should be compared with the
corresponding morphism spaces for $\calL$ described in Corollary
\ref{cor:fukaya-gens}.
The morphism spaces involving the Lagrangians $W_\times$ and
$W_\timesbar$ are
\begin{align*}
  (W_\times, L_{\pm 1}) &=
  \langle w_{\pm 1}^{(\pm 2)} \rangle, &
  (W_\timesbar, L_{\pm 1}) &=
  \langle \wbar_{\pm 1}^{(\mp 2)} \rangle.
\end{align*}
These spaces should be compared to the analogous morphism spaces
$(W_+,L_n)$ and $(W_-,L_n)$ for $\calL$ described in Section
\ref{sec:s3-ops-gens}.

\begin{figure}
  \centering
  \includegraphics[scale=0.65]{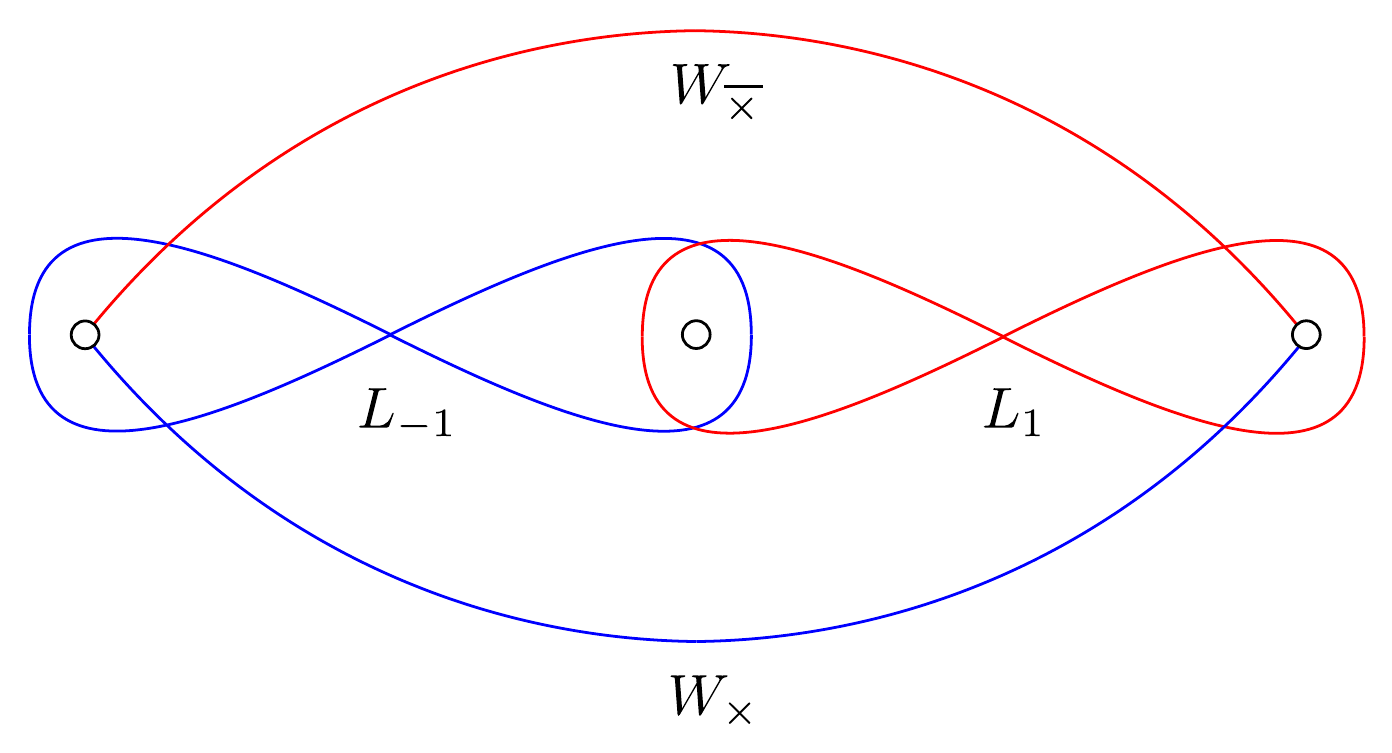}
  \caption{
    \label{fig:pillowcase-lagrangians}
    Lagrangians $L_1$, $L_{-1}$, $W_\times$, and $W_\timesbar$ in the
    pillowcase $R^*(S^2,4)$.
    The three circles indicate three of the puncture points of
    $R^*(S^2,4)$; the fourth puncture point is the point at infinity.
  }
\end{figure}

The $A_\infty$ operations derived in \cite{Hedden-3} prove the
following result for $\calP$, which should be compared to
Conjecture \ref{conj:fukaya} for $\calL$:

\begin{theorem}
The generators $a_{\pm 1}$ are identity elements that satisfy
\begin{align*}
  &\mu_\calP^2(a_{\pm 1},x) = x, &
  &\mu_\calP^2(y,a_{\pm 1}) = y
\end{align*}
whenever these operations are defined.
We have the following product operations:
\begin{align*}
  &\mu_\calP^2(p_{1,-1}, q_{-1,1}) = c_1, &
  &\mu_\calP^2(q_{-1,1}, p_{1,-1}) = c_{-1}, \\
  &\mu_\calP^2(c_1,c_1) =
  \mu_\calP^2(c_1, p_{1,-1}) = \mu_\calP^2(p_{1,-1}, c_{-1}) = 0, &
  & \mu_\calP^2(c_{-1},c_{-1}) =
  \mu_\calP^2(c_{-1}, q_{-1,1}) = \mu_\calP^2(q_{-1,1}, c_1) = 0.
\end{align*}
All operations of the form
$\mu_\calP^m(x_m, \cdots, x_1)$
for $m \neq 2$ and
$x_m, \cdots, x_1 \in
\{a_{\pm 1},\, c_{\pm 1},\, p_{1,-1},\, q_{-1,1}\}$ are zero.
\end{theorem}

We have the following result for $R^*(S^2,4)$, which should be
compared to Conjecture \ref{conj:s3-ops} for $R^*(T^2,2)$:

\begin{theorem}
We have the following product operations:
\begin{align*}
  \mu^2(a_{\pm 1}, w_{\pm 1}) &= w_{\pm 1}, &
  \mu^2(a_{\pm 1}, \wbar_{\pm 1}) &= \wbar_{\pm 1}.
\end{align*}
We have the following $\mu^3$ operations:
\begin{align*}
  &\mu^3(c_{-1}, q_{-1,1}, w_1) =
  \mu^3(q_{-1,1}, c_1, w_1) =
  w_{-1}, &
  &\mu^3(c_1, p_{1,-1}, \wbar_{-1}) =
  \mu^3(p_{1,-1}, c_{-1}, \wbar_1) =
  \wbar_1.
\end{align*}
All other operations of the form
\begin{align*}
  &\mu^m(x_{m-1}, \cdots, x_1, w_{\pm 1}), &
  &\mu^m(x_{m-1}, \cdots, x_1, \wbar_{\pm 1})
\end{align*}
for
$x_{m-1}, \cdots, x_{m-1} \in
\{a_{\pm 1},\, c_{\pm 1},\, p_{1,-1},\, q_{-1,1}\}$ are zero.
\end{theorem}

\begin{proof}
This is a straightforward application of the methods described
in \cite{Hedden-3}.
\end{proof}

\section{Homotopy equivalences}
\label{sec:chain-homotopy-equivalences}

Here we describe the homotopy equivalences and differentials
used to prove Theorem \ref{theorem:twist}.
For economy of space we do not describe differentials that merge
circles, since they are straightforward modifications of corresponding
differentials that split circles.

We define planar 3-tangles in the disk $P_2$, $P_0^L$, $P_0^C$,
$P_0^R$, and $P_{-2}$ as shown in Figure \ref{fig:disk-tangles}.
Any planar 3-tangle in the disk is isotopic to one of these five
planar tangles, together with some number of circle components.
We let $\ell$, $r$, and $c$ denote the number of circle
components in the regions labeled $\ell$, $r$, and $c$ in Figure
\ref{fig:disk-tangles}, and we use the notation
\begin{align*}
  &P_2(\ell,r), &
  &P_0^L(c,\ell,r), &
  &P_0^C(\ell,r), &
  &P_0^R(\ell,r,c), &
  &P_{-2}(\ell,r)
\end{align*}
to describe a planar 3-tangle whose arc components are isotopic to
$P_2$, $P_0^L$, $P_0^C$, $P_0^R$, and $P_{-2}$, and which contains
the indicated number of circle components.

\begin{figure}
  \centering
  \includegraphics[scale=0.50]{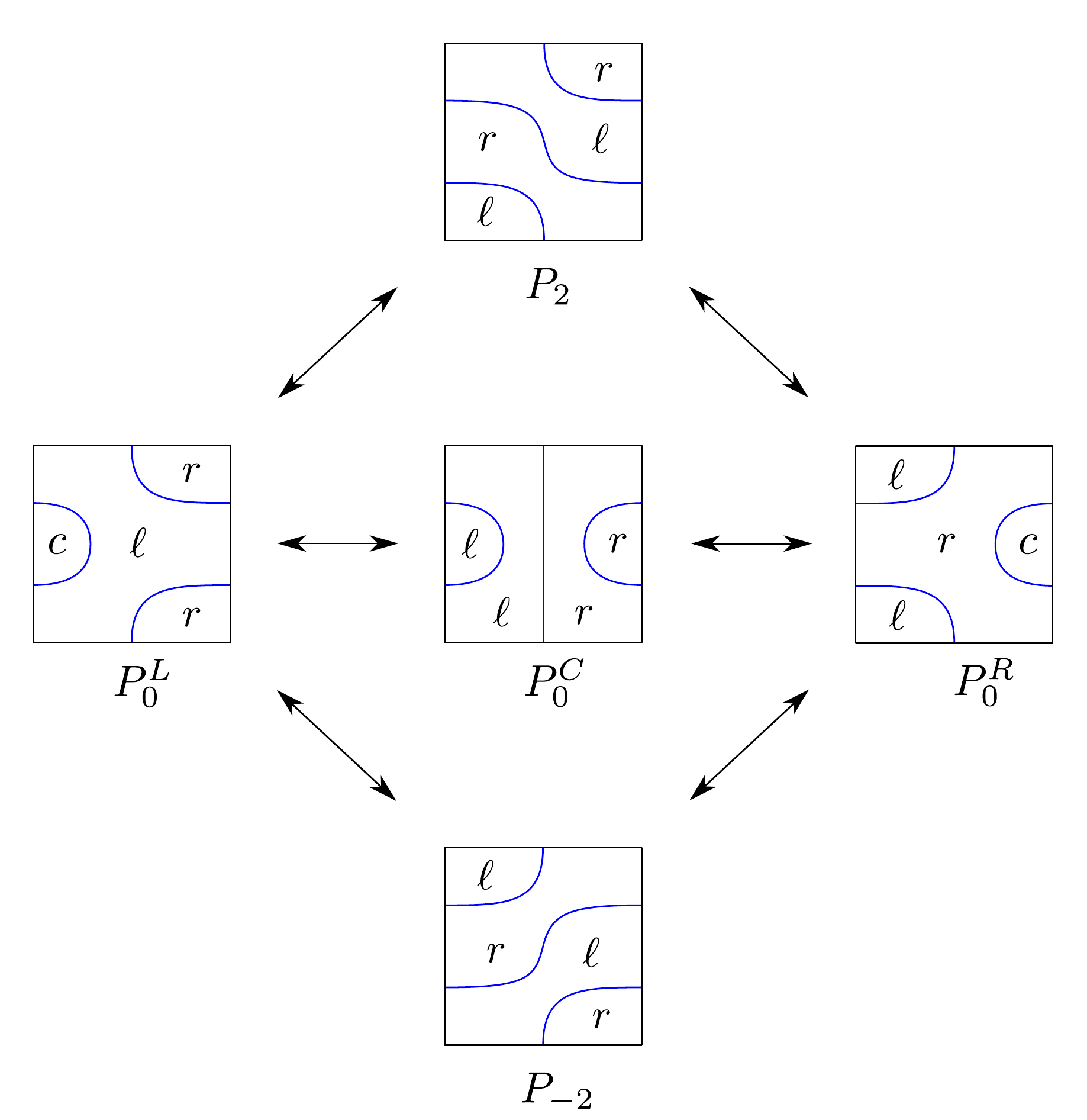}
  \caption{
    \label{fig:disk-tangles}
    Planar 3-tangles in the disk.
    Pairs of planar tangles connected by arrows are related by
    saddles.
  }
\end{figure}

We define linear maps
$\id_{\alpha \alpha}^{(0,0)}$,
$\id_{\beta \beta}^{(0,0)}$,
$\id_{\beta \alpha}^{(0,-2)}$,
$\id_{\alpha \beta}^{(0,2)}$:
\begin{align*}  
  &\id_{\alpha \alpha}:(W_0,L_0) \rightarrow (W_0,L_0), &
  &\alpha_0 \mapsto \alpha_0, &
  &\beta_0 \mapsto 0, \\
  &\id_{\beta \beta}:(W_0,L_0) \rightarrow (W_0,L_0), &
  &\alpha_0 \mapsto 0, &
  &\beta_0 \mapsto \beta_0, \\
  &\id_{\beta \alpha}:(W_0,L_0) \rightarrow (W_0,L_0), &
  &\alpha_0 \mapsto \beta_0, &
  &\beta_0 \mapsto 0, \\
  &\id_{\alpha \beta}:(W_0,L_0) \rightarrow (W_0,L_0), &
  &\alpha_0 \mapsto 0, &
  &\beta_0 \mapsto \alpha_0.
\end{align*}
We define a \emph{projection map} $\id_{ee}^{(0,0)}$:
\begin{align*}
  &\id_{ee}:A \rightarrow A, &
  &\id_{ee}(e) = e, & \id_{ee}(x) = 0.
\end{align*}
We define a \emph{swap map}
$\tau_{a,b}^{(0,0)}$:
\begin{align*}
  & \tau_{a,b}:A^{\otimes(a+b)} \rightarrow A^{\otimes(a+b)}, &
  v_1 \otimes \cdots \otimes v_a \otimes
  w_1 \otimes \cdots \otimes w_b \mapsto
  w_1 \otimes \cdots \otimes w_b \otimes
  v_1 \otimes \cdots \otimes v_a.
\end{align*}
For simplicity, we denote the maps $\mu^2(a_n,-)$ and $\mu^2(c_n,-)$
as $a_n$ and $c_n$.

\subsection{Chain homotopy equivalence for $P_2(\ell,r)$}

For a planar 3-tangle $P_2(\ell,r)$, diagram (\ref{diag:twist-he}) has
the form
\begin{eqnarray}
\label{diag:twist-he-p2}
\begin{tikzcd}
\left(\begin{array}{c}
(W_0, L_0) \otimes A^{\otimes(k+1)} \\
(W_0, L_2) \otimes A^{\otimes k}
\end{array}\right)
\arrow[shift left=2]{r}{d^0,H^0} \arrow{d}[shift left=4]{G^0} &[2em]
\left(\begin{array}{c}
(W_0, L_0) \otimes A^{\otimes k} \\
(W_0, L_0) \otimes A^{\otimes k}
\end{array}\right)
\arrow[shift left=2]{l}{d^1,H^1} \\
(W_0, L_2) \otimes A^{\otimes k}, \arrow[shift left=4]{u}{F^0}
\end{tikzcd}
\end{eqnarray}
where $k = \ell + r$.
We define linear maps
\begin{align*}  
  &\id_{\sigma\alpha}:(W_0,L_0) \rightarrow (W_0,L_2), &
  &\alpha_0 \mapsto \sigma_{2,0}, &
  &\beta_0 \mapsto 0, \\
  &\id_{\sigma\beta}:(W_0,L_0) \rightarrow (W_0,L_2), &
  &\alpha_0 \mapsto 0, &
  &\beta_0 \mapsto \sigma_{2,0}, \\
  &\id_{\alpha\sigma}:(W_0,L_2) \rightarrow (W_0,L_0), &
  &\sigma_{2,0} \mapsto \alpha_0, &
  &\tau_{2,0} \mapsto 0, \\
  &\id_{\beta\sigma}:(W_0,L_2) \rightarrow (W_0,L_0), &
  &\sigma_{2,0} \mapsto \beta_0, &
  &\tau_{2,0} \mapsto 0, \\
  &\id_{\tau\tau}:(W_0,L_2) \rightarrow (W_0,L_2), &
  &\sigma_{2,0} \mapsto 0, &
  &\tau_{2,0} \mapsto \tau_{2,0}.
\end{align*}
The differentials in diagram (\ref{diag:twist-he-p2}) are
\begin{align*}
  d^0 &=
  \left(\begin{array}{cc}
    0 & \id_{\beta \sigma} I_{\ell+r} \\
    0 & \id_{\beta \sigma} I_{\ell+r}
  \end{array}\right), &
  d^1 &=
  \left(\begin{array}{cc}
    (a_0 \etadot + c_0 \eta) I_{\ell+r} & a_0 \etadot I_{\ell+r} \\
    0 & 0
  \end{array}\right).
\end{align*}
The homotopy equivalences in diagram (\ref{diag:twist-he-p2})
are
\begin{align*}
  F^0 &=
  \left(\begin{array}{c}
    \id_{\alpha \sigma}\eta I_{\ell+r}  +
    \id_{\beta \sigma} \eta I_\ell \Sigma_r \\
    \id_{\tau\tau} I_{\ell+r}
  \end{array}\right), &
  G^0 &=
  \left(\begin{array}{cc}
    \id_{\sigma \alpha} \epsilondot I_{\ell+r} &
    \id_{\tau \tau} I_{\ell+r}
  \end{array}\right).
\end{align*}
The homotopy in diagram (\ref{diag:twist-he-p2}) is
\begin{align*}
  H^0 &=
  \left(\begin{array}{cc}
    \id_{\alpha\beta} \epsilondot I_{\ell+r} +
    \id_{\alpha\alpha} \epsilondot I_\ell \Sigma_r & 0 \\
    (a_0 \epsilon + \id_{\alpha\beta}\epsilondot)I_{\ell+r} +
    \id_{\alpha\alpha}\epsilondot I_\ell \Sigma_r & 0
  \end{array}\right), &
  H^1 &=
  \left(\begin{array}{cc}
    0 & 0 \\
    \id_{\sigma \beta} I_{\ell+r} & 0
  \end{array}\right).
\end{align*}
The differentials for a saddle that splits a circle in region $\ell$
from the arc component are
\begin{align*}
  d_+ &= a_0 \etadot I_{\ell+r}, &
  \dtilde_+ &=
  \left(\begin{array}{cc}
    a_0 \id  \etadot I_{\ell+r} +
    c_0 \id \etadot I_{\ell} \Sigma_{r_e} I_{r_n} & 0 \\
    0 & d_+
  \end{array}\right).
\end{align*}
The differentials for a saddle that splits a circle in region $r$ from
the arc component are
\begin{align*}
  d_+ &= a_0 I_{\ell+r} \etadot, &
  \dtilde_+ &=
  \left(\begin{array}{cc}
    a_0 \id I_{\ell+r} \etadot + c_0 \id I_{\ell+r} \eta +
    c_0 \id  I_{\ell_n} \Sigma_{\ell_e} I_r \etadot & 0 \\
    0 & d_+
  \end{array}\right).
\end{align*}
The differentials for a saddle
$P_2(\ell_n+\ell_e,r) \rightarrow P_0^R(\ell_n,r,\ell_e)$ are
\begin{align*}
  d_{R2} &= \id_{\beta \sigma} I_{\ell_n} \tau_{\ell_e,r}, &
  \dtilde_{R2} &=
  \left(\begin{array}{cc}
    (a_0 \etadot \id + c_0 \eta \id) I_{\ell_n} \tau_{\ell_e,r} +
    c_0 \etadot \id I_{\ell_n} (\tau_{\ell_e,r} \circ \Sigma_{\ell_e} I_r) & 0 \\
    0 & d_{R2}
  \end{array}\right).
\end{align*}
The differentials for a saddle
$P_2(\ell,r_e+r_n) \rightarrow P_0^L(r_e,\ell,r_n)$ are
\begin{align*}
  d_{L2} &= \id_{\beta \sigma} \tau_{\ell,r_e} I_{r_n}, &
  \dtilde_{L2} &=
  \left(\begin{array}{cc}
    a_0 \etadot \id \tau_{\ell,r_e} I_{r_n} +
    c_0 \etadot \id (\tau_{\ell,r_e} \circ I_\ell \Sigma_{r_e}) I_{r_n} & 0 \\
    0 & d_{L2}
  \end{array}\right).
\end{align*}
Differentials that split one circle into two circles or merge two
circles into one circle commute with $F^0$ and $G^0$ due to the
identities in equation (\ref{eqn:identities-sigma}).

\subsection{Chain homotopy equivalence for $P_0^R(\ell,r,c)$}

For a planar 3-tangle $P_0^R(\ell,r,c)$, diagram (\ref{diag:twist-he})
has the form
\begin{eqnarray}
\label{diag:twist-he-p0r}
\begin{tikzcd}
\left(\begin{array}{c}
(W_0, L_0) \otimes A^{\otimes(k+2)} \\
(W_0, L_0) \otimes A^{\otimes k}
\end{array}\right)
\arrow[shift left=2]{r}{d^0, H^0} \arrow{d}[shift left=4]{G^0} &[2 em]
\left(\begin{array}{c}
(W_0, L_0) \otimes A^{\otimes(k+1)} \\
(W_0, L_0) \otimes A^{\otimes(k+1)}
\end{array}\right)
\arrow[shift left=2]{l}{d^1, H^1} \\
(W_0, L_0) \otimes A^{\otimes k}, \arrow[shift left=4]{u}{F^0}
\end{tikzcd}
\end{eqnarray}
where $k = \ell + r + c$.
The differentials
in diagram (\ref{diag:twist-he-p0r}) are
\begin{align*}
  d^0 &=
  \left(\begin{array}{cc}
    0 & a_0 \etadot I_{\ell+r+c} \\
    0 & a_0 \etadot I_{\ell+r+c}
  \end{array}\right), &
  d^1 &=
  \left(\begin{array}{cc}
    a_0 \Delta I_{\ell+r+c} & a_0 \id \etadot I_{\ell+r+c} \\
    0 & 0
  \end{array}\right).
\end{align*}
the homotopy equivalences
in diagram (\ref{diag:twist-he-p0r}) are
\begin{align*}
  F^0 &=
  \left(\begin{array}{c}
    a_0 \eta \eta I_{\ell+r+c} +
    a_0 \etadot \eta I_\ell \Sigma_r I_c +
    \id_{\beta\beta} \etadot \eta I_{\ell+r} \Sigma_c +
    \id_{\alpha\beta} \etadot \eta I_{\ell+r+c} +
    \id_{\beta\alpha} (\eta \eta I_{\ell+r} \Sigma_c +
    \etadot \eta I_\ell \Sigma_r \Sigma_c) \\
    0
  \end{array}\right), \\
  G^0 &=
  \left(\begin{array}{cc}
    a_0 \epsilondot \epsilondot I_{\ell+r+c} +
    \id_{\beta\alpha} \epsilondot \epsilondot I_{\ell+r} \Sigma_c & 0
  \end{array}\right).
\end{align*}
The homotopy
in diagram (\ref{diag:twist-he-p0r}) is
\begin{align*}
  H^0 &=
  \left(\begin{array}{cc}
    a_0 (\epsilon \id) I_{\ell+r+c} +
    \id_{\alpha\beta} (\id_{ee} \circ m) I_{\ell+r+c} +
    a_0 (\id_{ee} \circ m) I_{\ell} \Sigma_{r+c}) &
    0 \\
    a_0 (\id_{ex} \circ m) I_{\ell+r+c} +
    \id_{\alpha\beta} (\id_{ee} \circ m) I_{\ell+r+c} +
    a_0 (\id_{ee} \circ m) I_\ell \Sigma_{r+c} &
    0
  \end{array}\right), \\
  H^1 &=
  \left(\begin{array}{cc}
    0 & 0 \\
    0 & a_0 \epsilon I_{\ell+r+c}
  \end{array}\right).
\end{align*}
The differentials for a saddle that splits a circle in region $\ell$
from the (left) of the arc component are
\begin{align*}
  d_+ &=
  a_0 \etadot I_{\ell+r+c} +
  c_0 \etadot I_\ell \Sigma_{r_e} I_{r_n + c}, &
  \dtilde_+ &=
  \left(\begin{array}{cc}
    a_0 \id \id \etadot I_{\ell+r+c} +
    c_0 \id\id \etadot I_\ell \Sigma_{r_e} I_{r_n+c} & 0 \\
    0 & d_+
  \end{array}\right).
\end{align*}
The differentials for a saddle that splits a circle in region $r$ from
the right of the arc component are
\begin{align*}
  d_+ &=
  a_0 I_\ell \etadot I_{r+c} +
  c_0 I_\ell \eta I_{r+c} +
  c_0 I_{\ell_n} \Sigma_{\ell_e} \etadot I_{r+c}, \\
  \dtilde_+ &=
  \left(\begin{array}{cc}
    a_0 \id \id I_\ell \etadot I_{r+c} +
    c_0 \id \id I_\ell \eta I_{r+c} +
    c_0 \id \id I_{\ell_n} \Sigma_{\ell_e} \etadot I_{r+c} & 0 \\
    0 & d_+
  \end{array}\right).
\end{align*}
The differentials for a saddle that splits a circle in region $r$ from
the left of the arc component are
\begin{align*}
  d_+ &=
  a_0 I_{\ell+r} \etadot I_c +
  c_0 I_{\ell+r} \etadot \Sigma_{c_e} I_{c_n}, &
  \dtilde_+ &=
  \left(\begin{array}{cc}
    a_0
    (\eta \epsilondot \id I_{\ell+r} \etadot +
    \etadot \epsilondot \id I_{\ell+r} \eta +
    \etadot \epsilon \id I_{\ell+r} \etadot) I_c & 0 \\
    0 & d_+
  \end{array}\right).
\end{align*}
The differentials for a saddle that splits a circle in region $c$ from
the (right) of the arc component are
\begin{align*}
  d_+ &=
  a_0 I_{\ell+r+c} \etadot +
  c_0 I_{\ell+r+c} \eta +
  c_0 I_{\ell+r_n} \Sigma_{r_e} I_c \etadot, \\
  \dtilde_+ &=
  \left(\begin{array}{cc}
    a_0 (
    \eta \epsilondot \id I_{\ell+r+c} \etadot +
    \etadot \epsilondot \id I_{\ell+r+c} \eta +
    \etadot \epsilon \id I_{\ell+r+c} \etadot) & 0 \\
    0 & d_+
  \end{array}\right).
\end{align*}
The differentials for a saddle
$P_0^R(\ell,r,c) \rightarrow P_2(\ell+c,r)$ are
\begin{align*}
  d_{2R} &= \id_{\sigma \alpha} I_\ell \tau_{r,c}, &
  \dtilde_{2R} &=
  \left(\begin{array}{cc}
    (a_0 \epsilondot \id + c_0 \epsilon \id) I_\ell \tau_{r,c} +
    c_0 \epsilondot \id I_\ell (\tau_{r,c} \circ I_r \Sigma_c) & 0 \\
    0 & d_{2R}
  \end{array}\right).
\end{align*}
The differentials for a saddle
$P_0^R(\ell,r_e + r_n,c) \rightarrow P_0^C(\ell+r_e,r_n+c)$ are
\begin{align*}
  d_{CR} &=
  a_0 \etadot I_{\ell+ r_e+r_n+c} +
  c_0 \etadot I_\ell \Sigma_{r_e} I_{r_n+c} +
  c_0 \etadot I_{\ell+r_e+r_n} \Sigma_c, \\
  \dtilde_{CR} &=
  \left(\begin{array}{cc}
    a_0 \etadot \tau_{1,1} I_{\ell+r_e+r_n+c} +
    c_0 \etadot \tau_{1,1} I_\ell \Sigma_{r_e} I_{r_n+c} &
    0 \\
    0 & d_{CR}
  \end{array}\right).
\end{align*}
The differentials for a saddle
$P_0^R(\ell,r,c) \rightarrow P_{-2}(\ell+c,r)$ are
\begin{align*}
  d_{-2R} &= \id_{\tau \alpha} I_\ell \tau_{r,c}, &
  \dtilde_{-2R} &=
  \left(\begin{array}{cc}
    (a_0 \epsilondot \id + c_0 \epsilon \id) I_\ell \tau_{r,c} +
    c_0 \epsilondot \id I_\ell (\tau_{r,c} \circ I_r \Sigma_c) & 0 \\
    0 & d_{-2R}
  \end{array}\right).
\end{align*}
Differentials that split one circle into two circles or merge two
circles into one circle commute with $F^0$ and $G^0$ due to the
identities in equation (\ref{eqn:identities-sigma}).

\subsection{Chain homotopy equivalence for $P_0^L(c,\ell,r)$}

For a planar 3-tangle $P_0^L(c,\ell,r)$, diagram (\ref{diag:twist-he})
has the form
\begin{eqnarray*}
\label{diag:twist-he-p0l}
\begin{tikzcd}
\left(\begin{array}{c}
(W_0, L_0) \otimes A^{\otimes(k+2)} \\
(W_0, L_0) \otimes A^{\otimes k}
\end{array}\right)
\arrow[shift left=2]{r}{d^0, H^0} \arrow{d}[shift left=4]{G^0} &[2 em]
\left(\begin{array}{c}
(W_0, L_0) \otimes A^{\otimes(k+1)} \\
(W_0, L_0) \otimes A^{\otimes(k+1)}
\end{array}\right)
\arrow[shift left=2]{l}{d^1, H^1} \\
(W_0, L_0) \otimes A^{\otimes k}, \arrow[shift left=4]{u}{F^0}
\end{tikzcd}
\end{eqnarray*}
where $k = c + \ell + r$.
The differentials
in diagram (\ref{diag:twist-he-p0l}) are
\begin{align*}
  d^0 &=
  \left(\begin{array}{cc}
    0 & (a_0 \etadot + c_0 \eta) I_{c+\ell+r} \\
    0 & (a_0 \etadot + c_0 \eta) I_{c+\ell+r}
  \end{array}\right), &
  d^1 &=
  \left(\begin{array}{cc}
    (a_0 \id \etadot + c_0 \id \eta) I_{c+\ell+r} &
    a_0 \Delta I_{c+\ell+r} \\
    0 & 0
  \end{array}\right).
\end{align*}
The homotopy equivalences
in diagram (\ref{diag:twist-he-p0l}) are
\begin{align*}
  F^0 &=
  \left(\begin{array}{c}
    (\id_{\alpha\beta} \etadot \eta  +
    \id_{\alpha\alpha} \eta \eta) I_{c+\ell+r} +
    \id_{\alpha\alpha} \etadot \eta \Sigma_c I_{\ell+r} +
    (a_0 \etadot \eta +
    \id_{\beta\alpha} \eta \eta) I_{c+\ell} \Sigma_r +
    \id_{\beta\alpha} \etadot \eta \Sigma_c I_\ell \Sigma_r  \\
    0
  \end{array}\right), \\
  G^0 &=
  \left(\begin{array}{cc}
    (a_0 \epsilondot \epsilondot +
    \id_{\beta\alpha} (\epsilon \epsilondot + \epsilondot \epsilon))
    I_{c+\ell+r} +
    \id_{\beta\alpha} \epsilondot \epsilondot \Sigma_c I_{\ell+r} &
    0
  \end{array}\right).
\end{align*}
The homotopies
in diagram (\ref{diag:twist-he-p0l}) are
\begin{align*}
  H^0 &=
  \left(\begin{array}{cc}
    (\id_{\beta\beta} (\id_{ex} \circ m) +
    \id_{\alpha\beta} \eta \epsilondot \epsilondot) I_{c+\ell+r} +
    a_0 m I_{c+\ell} \Sigma_r &
    0 \\
    (\id_{\alpha\alpha} \id \epsilon +
    \id_{\beta\beta} \epsilon \id +
    \id_{\alpha\beta} \eta \epsilondot \epsilondot) I_{c+\ell+r} +
    a_0 m I_{c+\ell} \Sigma_r &
    0
  \end{array}
  \right), \\
  H^1 &=
  \left(\begin{array}{cc}
    0 & 0 \\
    a_0 \epsilon I_{c+\ell+r} +
    (a_0 \epsilondot +
    \id_{\beta\alpha} \epsilon) I_{c+\ell} \Sigma_r &
    0
  \end{array}\right).
\end{align*}
The differentials for a saddle that splits a circle in region $c$ from
the left of the arc component are
\begin{align*}
  d_+ &=
  a_0 \etadot I_{c+\ell+r} +
  c_0 \etadot I_c \Sigma_{\ell_e} I_{\ell_n+r}, &
  \dtilde_+ &=
  \left(\begin{array}{cc}
    a_0 (
    \eta \epsilondot \id \etadot +
    \etadot \epsilondot \id \eta +
    \etadot \epsilon \id \etadot) I_{c+\ell+r} & 0 \\
    0 & d_+
  \end{array}\right).
\end{align*}
The differentials for a saddle that splits a circle in region $\ell$
from the right of the arc component are
\begin{align*}
  d_+ &=
  (a_0 I_c \etadot + c_0 I_c \eta) I_{\ell+r} +
  c_0 I_{c_n} \Sigma_{c_e} \etadot I_{\ell+r}, &
  \dtilde_+ &=
  \left(\begin{array}{cc}
    a_0 (
    \eta \epsilondot \id I_c \etadot +
    \etadot \epsilondot \id I_c \eta +
    \etadot \epsilon \id I_c \etadot) I_{\ell+r} & 0 \\
    0 & d_+
  \end{array}\right).
\end{align*}
The differentials for a saddle that splits a circle in region $\ell$
from the left of the arc component are
\begin{align*}
  d_+ &=
  a_0 I_{c+\ell} \etadot I_r +
  c_0 I_{c+\ell} \etadot \Sigma_{r_e} I_{r_n}, &
  \dtilde_+ &=
  \left(\begin{array}{cc}
    a_0 \id\id I_{c+\ell} \etadot I_r +
    c_0 \id\id I_{c+\ell} \etadot \Sigma_{r_e} I_{r_n}, & 0 \\
    0 & d_+
  \end{array}\right).
\end{align*}
The differentials for a saddle that splits a circle in region $r$ from
the (right) of the arc component are
\begin{align*}
  d_+ &=
  a_0 I_{c+\ell+r} \etadot + c_0 I_{c+\ell+r} \eta +
  c_0 I_{c+\ell_n} \Sigma_{\ell_e} I_r \etadot, \\
  \dtilde_+ &=
  \left(\begin{array}{cc}
    a_0 \id\id I_{c+\ell+r} \etadot +
    c_0 \id\id I_{c+\ell+r} \eta +
    c_0 \id \id I_{c+\ell_n} \Sigma_{\ell_e} I_r \etadot & 0 \\
    0 & d_+
  \end{array}\right).
\end{align*}
The differentials for a saddle
$P_0^L(c,\ell,r) \rightarrow P_2(\ell,r+c)$ are
\begin{align*}
  d_{2L} &= \id_{\sigma \alpha} \tau_{c,\ell} I_r, &
  \dtilde_{2L} &=
  \left(\begin{array}{cc}
    a_0 \epsilondot \id \tau_{c,\ell} I_r +
    c_0 \epsilondot \id (\tau_{c,\ell} \circ \Sigma_c I_\ell) I_r & 0 \\
    0 & d_{2L}
  \end{array}\right).
\end{align*}
The differentials for a saddle
$P_0^L(c,\ell_n+\ell_e,r) \rightarrow P_0^C(c+\ell_n, \ell_e+r)$:
\begin{align*}
  d_{CL} &=
  a_0 \etadot I_{c+\ell_n+\ell_e+r} +
  c_0 \eta I_{c+\ell_n+\ell_e+r} +
  c_0 \etadot \Sigma_c I_{\ell_n+\ell_e+r} +
  c_0 \etadot I_{c+\ell_n} \Sigma_{l_e} I_r, \\
  \dtilde_{CL} &=
  \left(\begin{array}{cc}
    a_0 \id \id \etadot I_{c+\ell_n+\ell_e+r} +
    c_0 \id \id \eta I_{c+\ell_n+\ell_e+r} +
    c_0 \id \id \etadot I_{c+\ell_n} \Sigma_{\ell_e} I_r &
    0 \\
    0 & d_{CL}
  \end{array}\right).
\end{align*}
The differentials for a saddle
$P_0^L(c,\ell,r) \rightarrow P_{-2}(\ell,r+c)$ are
\begin{align*}
  d_{-2L} &= \id_{\tau \alpha} \tau_{c,\ell} I_r, &
  \dtilde_{-2L} &=
  \left(\begin{array}{cc}
    a_0 \epsilondot \id \tau_{c,\ell} I_r +
    c_0 \epsilondot \id (\tau_{c,\ell} \circ \Sigma_c I_\ell) I_r & 0 \\
    0 & d_{-2L}
  \end{array}\right).
\end{align*}
Differentials that split one circle into two circles or merge two
circles into one circle commute with $F^0$ and $G^0$ due to the
identities in equation (\ref{eqn:identities-sigma}).

\subsection{Chain homotopy equivalence for $P_0^C(\ell,r)$}

For a planar 3-tangle $P_0^C(\ell,r)$, diagram (\ref{diag:twist-he})
has the form
\begin{eqnarray}
\label{diag:twist-he-p0c}
\begin{tikzcd}
\left(\begin{array}{c}
(W_0, L_0) \otimes A^{\otimes(k+3)} \\
(W_0, L_0) \otimes A^{\otimes(k+1)}
\end{array}\right)
\arrow[shift left=2]{r}{d^0, H^0} \arrow{d}[shift left=4]{G^0} &[2 em]
\left(\begin{array}{c}
(W_0, L_0) \otimes A^{\otimes(k+2)} \\
(W_0, L_0) \otimes A^{\otimes(k+2)}
\end{array}\right)
\arrow[shift left=2]{l}{d^1, H^1} \\
(W_0, L_0) \otimes A^{\otimes(k+1)}, \arrow[shift left=4]{u}{F^0}
\end{tikzcd}
\end{eqnarray}
where $k = \ell + r$.
We define a 3-fold product map
\begin{align*}
  &m_3:A^{\otimes 3} \rightarrow A, &
  &m_3 = m \circ \id m.
\end{align*}
The differentials
in diagram (\ref{diag:twist-he-p0c}) are
\begin{align*}
  d^0 &=
  \left(\begin{array}{cc}
    0 & a_0 \Delta I_{\ell+r} \\
    0 & a_0 \Delta I_{\ell+r}
  \end{array}\right), &
  d^1 &=
  \left(\begin{array}{cc}
    a_0 \id \Delta I_{\ell+r} & a_0 \Delta \id I_{\ell+r} \\
    0 & 0
  \end{array}\right).
\end{align*}
The homotopy equivalences:
in diagram (\ref{diag:twist-he-p0c}) are
\begin{align*}
  F^0 &=
  \left(\begin{array}{cc}
    a_0 \id \eta \eta I_{\ell+r} +
    \id_{\alpha\beta} (\id \eta \id \circ \Delta) I_{\ell+r} +
    a_0 (\id \eta \id \circ \Delta) I_\ell \Sigma_r \\
    0
  \end{array}\right), \\
  G^0 &=
  \left(\begin{array}{cc}
    a_0 (\id \epsilondot \epsilondot + \epsilondot \id \epsilondot +
    \epsilondot \epsilondot \id) I_{\ell+r} & 0
  \end{array}\right).
\end{align*}
The homotopy
in diagram (\ref{diag:twist-he-p0c}) is
\begin{align*}
  H^0 &=
  \left(\begin{array}{cc}
    (a_0 \id \id \epsilon +
    \id_{\alpha\beta} m_3 \eta) I_{\ell+r} +
    a_0 m_3 \eta I_\ell \Sigma_r & 0 \\
    (a_0 \id (\id_{ex} \circ m) +
    \id_{\alpha\beta} m_3 \eta) I_{\ell+r} +
    a_0 m_3 \eta I_\ell \Sigma_r & 0
  \end{array}\right), &
  H^1 &=
  \left(\begin{array}{cc}
    0 & 0 \\
    0 & a_0 \id \epsilon I_{\ell+r}
  \end{array}\right).
\end{align*}
The differentials for a saddle that splits a circle in region $\ell$
from the left or right of the circle containing the left arc is
\begin{align*}
  d_+ &=
  a_0 (
  \eta \epsilondot \etadot +
  \etadot \epsilondot \eta +
  \etadot \epsilon \etadot) I_{\ell+r}, &
  \dtilde_+ &=
  \left(\begin{array}{cc}
  a_0 (
  \eta \epsilondot \id \id \etadot +
  \etadot \epsilondot \id \id \eta +
  \etadot \epsilon \id \id \etadot) I_{\ell+r} & 0 \\
  0 & d_+
  \end{array}\right).
\end{align*}
The differentials for a saddle that splits a circle in region $\ell$
from the (left) of the arc component are
\begin{align*}
  d_+ &=
  a_0 \id \etadot I_{\ell+r} +
  c_0 \id \etadot I_\ell \Sigma_{r_e} I_{r_n}, &
  \dtilde_+ &=
  \left(\begin{array}{cc}
  a_0 \id \id \id \etadot I_{\ell+r} +
  c_0 \id \id \id \etadot I_\ell \Sigma_{r_e} I_{r_n} & 0 \\
  0 & d_+
  \end{array}\right).
\end{align*}
The differentials for a saddle that splits a circle in region $r$ from
the (right) of the arc component are
\begin{align*}
  d_+ &=
  a_0 \id I_{\ell+r} \etadot +
  c_0 \id I_{\ell+r} \eta +
  c_0 \id I_{\ell_n} \Sigma_{l_e} I_r \etadot, \\
  \dtilde_+ &=
  \left(\begin{array}{cc}
  a_0 \id \id \id I_{\ell+r} \etadot +
  c_0 \id \id \id I_{\ell+r} \eta +
  c_0 \id \id \id I_{\ell_n} \Sigma_{l_e} I_r \etadot & 0 \\
  0 & d_+
  \end{array}\right).
\end{align*}
The differentials for a saddle that splits a circle in region $r$ from
the left or right of the circle containing the right arc are
\begin{align*}
  d_+ &=
  a_0 (
  \eta \epsilondot I_{\ell+r}\etadot +
  \etadot \epsilondot I_{\ell+r} \eta +
  \etadot \epsilon I_{\ell+r} \etadot), \\
  \dtilde_+ &=
  \left(\begin{array}{cc}
  a_0 (
  \id \id \eta \epsilondot I_{\ell+r} \etadot +
  \id \id \etadot \epsilondot I_{\ell+r} \eta +
  \id \id \etadot \epsilon I_{\ell+r} \etadot) & 0 \\
  0 & d_+
  \end{array}\right).
\end{align*}
The differentials for a saddle
$P_0^C(\ell_n+\ell_e,r_n+r_e) \rightarrow P_0^R(\ell_n,\ell_e+r_n,r_e)$ are
\begin{align*}
  d_{RC} &=
  a_0 \epsilondot I_{\ell_n+\ell_e+r_n+r_e} +
  c_0 \epsilondot I_{\ell_n} \Sigma_{\ell_e} I_{r_n+r_e} +
  c_0 \epsilondot I_{\ell_n+\ell_e+r_n} \Sigma_{r_e}, \\
  \dtilde_{RC} &=
  \left(\begin{array}{cc}
    a_0 \epsilondot \tau_{1,1} I_{\ell_n+\ell_e+r_n+r_e} +
    c_0 \epsilondot \tau_{1,1} I_{\ell_n} \Sigma_{\ell_e} I_{r_n+r_e} &
    0 \\
    0 & d_{RC}
  \end{array}\right).
\end{align*}
The differentials for a saddle
$P_0^C(\ell_e+\ell_n, r_e+r_n) \rightarrow P_0^L(\ell_e,\ell_n+r_e,r_n)$:
\begin{align*}
  d_{LC} &=
  a_0 \epsilondot I_{\ell_e+\ell_n+r_e+r_n} +
  c_0 \epsilon I_{\ell_e+\ell_n+r_e+r_n} +
  c_0 \epsilondot \Sigma_{\ell_e} I_{\ell_n+r_e+r_n} +
  c_0 \epsilondot I_{\ell_e+\ell_n} \Sigma_{r_e} I_{r_n}, \\
  \dtilde_{LC} &=
  \left(\begin{array}{cc}
    a_0 \id \id \epsilondot I_{\ell_e+\ell_n+r_e+r_n} +
    c_0 \id \id \epsilon I_{\ell_e+\ell_n+r_e+r_n} +
    c_0 \id \id \epsilondot I_{\ell_e+\ell_n} \Sigma_{r_e} I_{r_n} &
    0 \\
    0 & d_{LC}
  \end{array}\right).
\end{align*}
Differentials that split one circle into two circles or merge two
circles into one circle commute with $F^0$ and $G^0$ due to the
identities in equation (\ref{eqn:identities-sigma}).

\subsection{Chain homotopy equivalence for $P_{-2}(\ell,r)$}

For a planar 3-tangle $P_{-2}(\ell,r)$, diagram (\ref{diag:twist-he})
has the form
\begin{eqnarray}
\label{diag:twist-he-m2}
\begin{tikzcd}
\left(\begin{array}{c}
(W_0, L_0) \otimes A^{\otimes(k+1)} \\
(W_0, L_{-2}) \otimes A^{\otimes k}
\end{array}\right)
\arrow[shift left=2]{r}{d^0,H^0} \arrow{d}[shift left=4]{G^0} &[2em]
\left(\begin{array}{c}
(W_0, L_0) \otimes A^{\otimes k} \\
(W_0, L_0) \otimes A^{\otimes k}
\end{array}\right)
\arrow[shift left=2]{l}{d^1,H^1} \\
(W_0, L_{-2}) \otimes A^{\otimes k}, \arrow[shift left=4]{u}{F^0}
\end{tikzcd}
\end{eqnarray}
where $k = \ell + r$.
We define linear maps
\begin{align*}
  &\id_{\tau\alpha}:(W_0,L_0) \rightarrow (W_0,L_{-2}), &
  &\alpha_0 \mapsto \tau_{-2,0}, &
  &\beta_0 \mapsto 0, \\
  &\id_{\tau\beta}:(W_0,L_0) \rightarrow (W_0,L_{-2}), &
  &\alpha_0 \mapsto 0, &
  &\beta_0 \mapsto \tau_{-2,0}, \\
  &\id_{\alpha\tau}:(W_0,L_{-2}) \rightarrow (W_0,L_0), &
  &\sigma_{-2,0} \mapsto 0, &
  &\tau_{-2,0} \mapsto \alpha_0, \\
  &\id_{\beta\tau}:(W_0,L_{-2}) \rightarrow (W_0,L_0), &
  &\sigma_{-2,0} \mapsto 0, &
  &\tau_{-2,0} \mapsto \beta_0, \\
  &\id_{\sigma\sigma}:(W_0,L_{-2}) \rightarrow (W_0,L_{-2}), &
  &\sigma_{-2,0} \mapsto \sigma_{-2,0}, &
  &\tau_{2,0} \mapsto 0.
\end{align*}
The differentials in diagram (\ref{diag:twist-he-m2}) are
\begin{align*}
  d^0 &=
  \left(\begin{array}{cc}
    0 & \id_{\beta \tau} I_{\ell+r} \\
    0 & \id_{\beta \tau} I_{\ell+r}
  \end{array}\right), &
  d^1 &=
  \left(\begin{array}{cc}
    (a_0 \etadot + c_0 \eta) I_{\ell+r} & a_0 \etadot I_{\ell+r} \\
    0 & 0
  \end{array}\right).
\end{align*}
The homotopy equivalences in diagram (\ref{diag:twist-he-m2})
are
\begin{align*}
  F^0 &=
  \left(\begin{array}{c}
    \id_{\alpha \tau}\eta I_{\ell+r}  +
    \id_{\beta \tau} \eta I_\ell \Sigma_r \\
    \id_{\sigma\sigma} I_{\ell+r}
  \end{array}\right), &
  G^0 &=
  \left(\begin{array}{cc}
    \id_{\tau \alpha} \epsilondot I_{\ell+r} &
    \id_{\sigma \sigma} I_{\ell+r}
  \end{array}\right).
\end{align*}
The homotopy in diagram (\ref{diag:twist-he-m2}) is
\begin{align*}
  H^0 &=
  \left(\begin{array}{cc}
    \id_{\alpha\beta} \epsilondot I_{\ell+r} +
    \id_{\alpha\alpha} \epsilondot I_\ell \Sigma_r & 0 \\
    (a_0 \epsilon + \id_{\alpha\beta}\epsilondot)I_{\ell+r} +
    \id_{\alpha\alpha}\epsilondot I_\ell \Sigma_r & 0
  \end{array}\right), &
  H^1 &=
  \left(\begin{array}{cc}
    0 & 0 \\
    \id_{\tau \beta} I_{\ell+r} & 0
  \end{array}\right).
\end{align*}
The differentials for a saddle that splits a circle in region $\ell$
from the arc component are
\begin{align*}
  d_+ &= a_0 \etadot I_{\ell+r}, &
  \dtilde_+ &=
  \left(\begin{array}{cc}
    a_0 \id  \etadot I_{\ell+r} +
    c_0 \id \etadot I_{\ell} \Sigma_{r_e} I_{r_n} & 0 \\
    0 & d_+
  \end{array}\right).
\end{align*}
The differentials for a saddle that splits a circle in region $r$ from
the arc component are
\begin{align*}
  d_+ &= a_0 I_{\ell+r} \etadot, &
  \dtilde_+ &=
  \left(\begin{array}{cc}
    a_0 \id I_{\ell+r} \etadot + c_0 \id I_{\ell+r} \eta +
    c_0 \id  I_{\ell_n} \Sigma_{\ell_e} I_r \etadot & 0 \\
    0 & d_+
  \end{array}\right).
\end{align*}
The differentials for a saddle
$P_{-2}(\ell_n+\ell_e,r) \rightarrow P_0^R(\ell_n,r,\ell_e)$ are
\begin{align*}
  d_{R-2} &= \id_{\beta \tau} I_{\ell_n} \tau_{\ell_e,r}, &
  \dtilde_{R-2} &=
  \left(\begin{array}{cc}
    (a_0 \etadot \id + c_0 \eta \id) I_{\ell_n} \tau_{\ell_e,r} +
    c_0 \etadot \id I_{\ell_n} (\tau_{\ell_e,r} \circ \Sigma_{\ell_e} I_r) & 0 \\
    0 & d_{R-2}
  \end{array}\right).
\end{align*}
The differentials for a saddle
$P_{-2}(\ell,r_e+r_n) \rightarrow P_0^L(r_e,\ell,r_n)$ are
\begin{align*}
  d_{L-2} &= \id_{\beta \tau} \tau_{\ell,r_e} I_{r_n}, &
  \dtilde_{L-2} &=
  \left(\begin{array}{cc}
    a_0 \etadot \id \tau_{\ell,r_e} I_{r_n} +
    c_0 \etadot \id (\tau_{\ell,r_e} \circ I_\ell \Sigma_{r_e}) I_{r_n} & 0 \\
    0 & d_{L-2}
  \end{array}\right).
\end{align*}
Differentials that split one circle into two circles or merge two
circles into one circle commute with $F^0$ and $G^0$ due to the
identities in equation (\ref{eqn:identities-sigma}).

\end{appendix}

\bibliographystyle{abbrv}
\bibliography{arxiv-fukaya-r-t2-2}

\end{document}